%% file: max-card-matching.tex
\documentclass[a4paper,reqno]{amsart}
\usepackage{svn}
\SVN $Revision: 1145 $ 
\SVN $Date: 2010-11-30 14:10:37 +0100 (Die, 30. Nov 2010) $
\newif\ifprivate

\usepackage{array}
\usepackage{graphicx}
\usepackage{subfigure}
\usepackage[unicode]{hyperref}
\usepackage{ifpdf}
\usepackage{a4wide}
\usepackage{rotating}

\newcommand{\calA}{\mathcal{A}}
\newcommand{\calB}{\mathcal{B}}
\newcommand{\calS}{\mathcal{S}}
\DeclareMathOperator{\CF}{CF}
\DeclareMathOperator{\CFP}{K}
\newcommand{\floor}[1]{\left\lfloor#1\right\rfloor}
\newcommand{\lambdabar}{\overline{\lambda}}
\newcommand{\liref}[2]{Lemma~\ref{#1}~(\ref{#2})}
\newcommand{\lireftwo}[3]{Lemma~\ref{#1}~(\ref{#2} and~\ref{#3})}
\newcommand{\lirefthree}[4]{Lemma~\ref{#1}~(\ref{#2}, \ref{#3} and~\ref{#4})}
\newcommand{\rref}[1]{(\ref{#1})}
\newcommand{\rholimit}{\rho_{\mathit{lim}}}
\newcommand{\type}{\mathop{\mathsf{type}}}

\newtheorem{lemma}{Lemma}[section]
\newtheorem{proposition}[lemma]{Proposition}
\newtheorem{theorem}[lemma]{Theorem}
\theoremstyle{remark}
\newtheorem{remark}[lemma]{Remark}

\theoremstyle{definition}
\newtheorem{definition}[lemma]{Definition}

\newtheorem*{todo}{Todo}

\numberwithin{equation}{section}
\newcounter{replacement}
\renewcommand{\thereplacement}{R\arabic{replacement}}

\ifpdf
\fi

\makeatletter
\@namedef{subjclassname@2010}{%
  \textup{2010} Mathematics Subject Classification}
\makeatother

\allowdisplaybreaks
\begin{document}

\title{The number of maximum matchings in a tree}

\author{Clemens Heuberger}\thanks{C. Heuberger is supported by the Austrian
Science Foundation FWF, project S9606, that is part of the
Austrian National Research Network ``Analytic Combinatorics
and Probabilistic Number Theory.''
This paper was partly written while C.~Heuberger was a visitor at Stellenbosch University.
}
\address{Institut f\"ur Mathematik B\\Technische Universit\"at Graz\\Austria}
\email{clemens.heuberger@tugraz.at}
\author{Stephan Wagner}\thanks{This material is based upon work supported financially by the National Research Foundation of South Africa under grant number 70560.}
\address{Department of Mathematical Sciences\\Stellenbosch University\\South Africa}
\email{swagner@sun.ac.za}
\begin{abstract}
We determine upper and lower bounds for the number of maximum matchings (i.e., matchings of maximum cardinality) $m(T)$ of a tree $T$ of given order. While the trees that attain the lower bound are easily characterised, the trees with largest number of maximum matchings show a very subtle structure. We give a complete characterisation of these trees and derive that the number of maximum matchings in a tree of order $n$ is at most $O(1.391664^n)$ (the precise constant being an algebraic number of degree $14$). As a corollary, we improve on a recent result by G\'orska and Skupie\'n on the number of maximal matchings (maximal with respect to set inclusion). 
\end{abstract}
\date{\today}

\maketitle
\ifprivate \thispagestyle{myheadings}\pagestyle{myheadings} \markboth{\jobname{}
rev. \SVNRevision{} ---
  \SVNDate{} \SVNTime}{\jobname{} rev. \SVNRevision{} --- \SVNDate{} \SVNTime}
\fi

\section{Introduction and statement of main results}

Many problems in graph theory can be described as follows: for a certain class of graphs and a graph parameter, determine the largest and smallest possible value of the parameter, given the order of a graph (and possibly other conditions). One family that is particularly well-studied in this regard is the family of trees, not only because of their simplicity, but also in view of their many applications in various areas of science.

On the other hand, lots of natural graph parameters are defined as the number of vertex or edge subsets of a certain kind; we mention, for example, the number of independent vertex subsets \cite{lin1995trees,prodinger1982fibonacci}, the number of matchings \cite{gutman1980graphs}, the number of dominating or efficient dominating sets \cite{brod2006australasian,brod2008recurrence} or the number of subtrees \cite{kirk2008largest,szekely2007binary}. Some of them play an important role in applications as well, for instance the number of matchings that is known as \emph{Hosoya index} in mathematical chemistry \cite{gutman1986mathematical,hosoya1986topological} and is also connected to the \emph{monomer-dimer model} of statistical physics \cite{heilmann1972theory}. The same can be said of the number of independent sets, which is studied under the name \emph{Merrifield-Simmons index} in chemistry \cite{merrifield1989topological} and which is related to Hard Models in physics \cite{baxter1980hardsquare}. For both these parameters, the minimum and maximum among all trees of given order are well known and are obtained for the star and the path respectively. A tremendous number of publications deals with related problems, concerning restricted classes of trees or tree-like graphs; the interested reader is referred to \cite{wagner2010maxima} and the references therein.

It is natural to consider variants of these graph parameters: instead of the number of matchings, one might be interested in the number of maximal matchings (maximal with respect to inclusion) or maximum matchings (matchings of largest possible cardinality). The same holds, of course, for the number of independent sets.

The number of maximal independent sets is treated in \cite{sagan1988independent,wilf1986number}---the maximum turns out to occur for an extended star. More recently, maximal matchings were studied by G\'orska and Skupie\'n \cite{gorska2007trees}, who determined exponential upper and lower bounds for the maximum number of maximal matchings among all trees of given order. To the best of our knowledge, however, there are no analogous results on the number of maximum matchings, i.e., matchings of largest possible cardinality. Clearly, any maximum (cardinality) matching is also maximal with respect to inclusion, but the converse is not true. In fact, graphs for which every maximal matching is also a maximum matching are known as \emph{equimatchable} \cite{lovasz1986matching}.

In the following, we denote the number of maximum matchings in a graph $G$ by $m(G)$. Our goal is to characterise the trees of given order $n$ for which the maximum and the minimum of this parameter are attained. This problem also has an algebraic interpretation: it is well known that the characteristic polynomial of a tree $T$ of order $|T| = n$ coincides with the \emph{matching polynomial} \cite{lovasz1986matching}
$$\phi(T,x) = \sum_{k=0}^{\lfloor n/2 \rfloor} (-1)^k a_k(T) x^{n-2k},$$
where $a_k(T)$ is the number of matchings of cardinality $k$ in $T$. This is a special case of a general theorem on the coefficients of the characteristic polynomial---see for instance \cite{cvetkovic1995spectra}. It follows that $m(T)$ is precisely the (absolute value of the) last nonzero coefficient of $\phi(T,x)$ and thus the product of the absolute values of all nonzero eigenvalues. In this sense, $m(T)$ is a multiplicative analogue of the so-called \emph{energy} of a graph \cite{gutman1999energy,gutman1986mathematical}, which is defined as the sum of the absolute values of all eigenvalues.

The lower bound for $m(T)$ is almost trivial, and the trees that attain it can also be characterised easily:

\begin{theorem}\label{thm:lower}
For any tree $T$ of even order $n$, $m(T) \geq 1$ with equality if and only if $T$ has a perfect matching. For a tree $T$ of odd order $n > 1$, $m(T) \geq 2$ with equality if and only if $T$ is obtained from a tree $T'$ of order $n-1$ with a perfect matching by doubling one of the leaves (i.e., choosing a leaf $v$ and attaching a second leaf to $v$'s unique neighbour).
\end{theorem}

We note that a path of even order is an example of a tree of even order admitting a perfect matching.

The analogous problem asking for the largest possible number of maximum matchings appears to be much harder. The bound provided by G\'orska and Skupie\'n for the number of maximal matchings immediately provides an upper bound for the number of maximum matchings, so that we have $m(T) = O(1.395337^n)$ (the constant being a root of the algebraic equation $x^4 - 2x - 1$) by the result stated in \cite{gorska2007trees}. We improve this to the following:

\begin{theorem}\label{thm:upper}
For $n \neq \{6,34\}$, there is a unique tree $T_n^*$ of order $n$ that maximises $m(T)$. For $n = 6$ and $n = 34$, there are two such trees. Asymptotically,
$$m(T_n^*)\sim c_{n\bmod 7}\lambda^{n/7},$$
where $\lambda = \frac12 (11+\sqrt{85}) \approx 10.1097722286464$ is the larger root of the polynomial $x^2-11x+9$ and the constants $c_j$, $j\in\{0,\ldots,6\}$, are given in Table~\ref{tab:asymptotic-constants}.
\end{theorem}

\begin{table}[htbp]
  \newlength{\fractionheight}
  \settoheight{\fractionheight}{$\displaystyle\frac{11\alpha-18}{85\alpha^{1/7}}$}
  \begin{equation*}
  \begin{array}{c|>{\rule{0pt}{1.3\fractionheight}\displaystyle}cc}
    j&\multicolumn{1}{c}{c_j}&\\\hline
    0&\frac{67\lambda-71}{765}&\approx0.792620574273610\\
    1&\frac{11\lambda-18}{85\lambda^{1/7}}&\approx0.787947762616490\\
    2&\frac{101047\lambda-90171}{614125\lambda^{2/7}}&\approx0.783080426542439\\
    3&\frac{4996\lambda-4448}{21675\lambda^{3/7}}&\approx0.788434032505851\\
    4&\frac{27\lambda-21}{85\lambda^{4/7}}&\approx0.790280714748050\\
    5&\frac{3209\lambda-2817}{7225\lambda^{5/7}}&\approx0.785510324593434\\
    6&\frac{6451616\lambda-5743408}{10440125\lambda^{6/7}}&\approx0.784269603628599
  \end{array}
\end{equation*}
  \caption{Constants $c_j$ in the asymptotics of $m(T_n^*)$.}
  \label{tab:asymptotic-constants}
\end{table}

While the improvement in the constant (from $1.395337$ to $\lambda^{1/7} \approx 1.391664$) seems modest, the main part of the theorem is the characterisation of the trees $T_n^*$, which will be stated explicitly in Section~\ref{sec:description}. Figure~\ref{fig:optimal-tree-181} shows $T_{181}^*$ as an example. Since maximum matchings are automatically maximal matchings, the theorem also improves on the lower bound for the maximum number of maximal matchings that was given by G\'orska and Skupie\'n in \cite{gorska2007trees}, which is $\Omega(1.390972^n)$ (the precise constant being $\sqrt[14]{51+5\sqrt{102}}$).

\begin{figure}[htbp]
  \centering
  \includegraphics{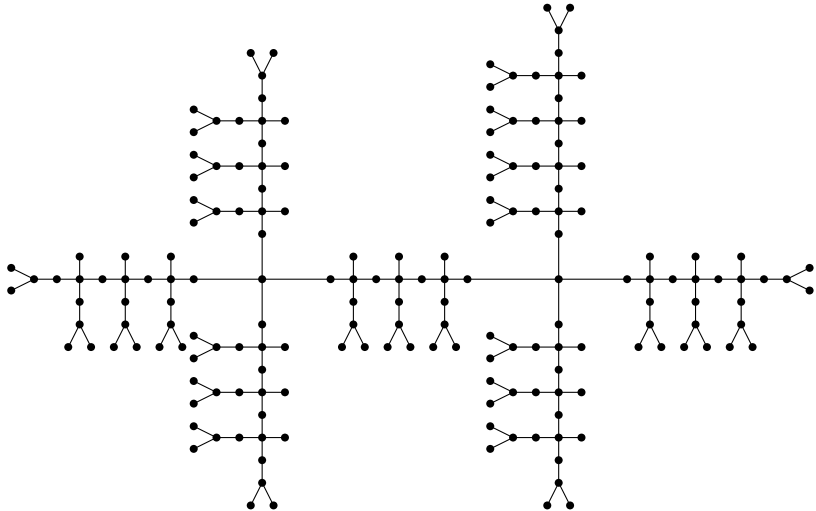}
  \caption{Unique optimal tree of order 181.}
  \label{fig:optimal-tree-181}
\end{figure}

The paper is organised as follows: in the following section, we deal with the simple lower bound (Theorem~\ref{thm:lower}), the rest is devoted to the proof of Theorem~\ref{thm:upper}. The structure of the ``optimal'' trees $T_n^*$ is described explicitly in Section~\ref{sec:description}, making use of the concept of an \emph{outline graph}. Then, some important preliminary results (Section~\ref{sec:pre}) and information about the local structure (Section~\ref{sec:local}) are gathered. The global structure is discussed in Section~\ref{sec:global}. The proof is rather long and technical---one of the reasons we consider this inevitable is the fact that seven different cases occur in the structure of the optimal trees, and that there is also a number of exceptions from the general pattern (note the case $n = 34$ in Theorem~\ref{thm:upper}: the precise characterisation of the structure is only valid for $n \geq 35$). Another reason is that there are many trees that almost reach the upper bound, as can be seen from some of the estimates made on the way to our main result.

\section{The lower bound}

Let us start with the simple lower bound; as stated in Theorem~\ref{thm:lower}, the minimum of $m(T)$ is either $1$ or $2$, depending on the parity of the order:

\begin{proof}
In the case of even $n$, the inequality is trivial, so that we only have to determine the cases of equality. If $T$ has a perfect matching, then this perfect matching can be reconstructed uniquely, starting from the leaves. Hence equality holds in this case. Otherwise, consider a tree $T$ of order $n$ and a maximum matching $M$. Since it is not a perfect matching, there is a vertex $v$ that is not covered by the matching. Now choose an arbitrary neighbour $w$ of $v$. Then $w$ must be covered by the matching $M$, since one could otherwise add the edge $vw$ to $M$ to obtain a larger matching, contradicting the choice of $M$. Now replace the edge that covers $w$ by the edge $vw$ to obtain a second matching of the same cardinality as $M$, which shows that $m(T) \geq 2$ unless $T$ has a perfect matching.

Now let us determine which trees of odd order satisfy $m(T) = 2$. Consider once again a maximum matching. Since the above argument can be carried out for any vertex that is not covered by $M$, we can only have $m(T) = 2$ if there is exactly one vertex $v$ that is not covered. Furthermore, $v$ must be a leaf: otherwise, we could apply the exchange procedure for each of its neighbours to obtain at least $3$ distinct maximum matchings. Let $w$ be $v$'s unique neighbour and assume that $w$ is covered by an edge $v'w$ in $M$. Then $v'$ must also be a leaf, since we could otherwise replace $v'w$ by $vw$ and repeat the argument. This shows that equality can only hold in the described case.
\end{proof}

As we will see in the following sections, the analogous question for the maximum of $m(T)$ is much harder and requires a completely different approach. Let us first give a precise description of the trees $T_n^*$ in Theorem~\ref{thm:upper}.

\section{The upper bound: description of the optimal trees}\label{sec:description}
As mentioned in the introduction, we define $m(T)$ to be the number of matchings
of maximal cardinality of a tree $T$. A tree $T$ is called an \emph{optimal tree}
if it maximises $m(T)$ over all trees of the same order.

The results on the global structure are formulated in terms of \emph{leaves},
\emph{forks}, and \emph{chains}.

\begin{definition}\label{definition:chain}
  \begin{enumerate}
  \item The graph of order $1$ is also denoted by $L$ (leaf).
  \item The rooted tree in Figure~\ref{definition:chain:fork}
    (with root $r$) is denoted by $F$ (fork).
  \item Chains are defined recursively: for a rooted tree $(T,r)$, we define the rooted tree $(CT,s)$ as in Figure~\ref{definition:chain:chain}. For $k\ge 1$ and a rooted tree $T$, we set
    \begin{equation*}
      C^kT:=C(C^{k-1}T) \text{ and }C^0T=T.
    \end{equation*}
  \end{enumerate}
\end{definition}
\begin{figure}[htbp]
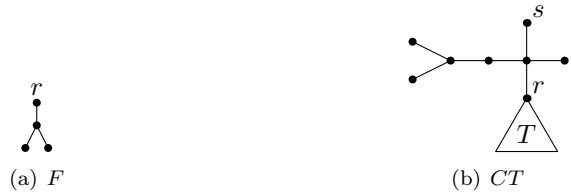

  \centering
  \rule{0cm}{0cm}\hfill\subfigure[$F$\label{definition:chain:fork}]{\parbox[t]{3cm}{\centering\includegraphics{fork.1}}}\hfill
  \subfigure[$CT$\label{definition:chain:chain}]{\parbox[t]{3cm}{\centering\includegraphics{chain.1}}}
  \hfill\rule{0cm}{0cm}
  \caption{Fork and chain (Definition~\ref{definition:chain})}
  \label{fig:chain}
\end{figure}

Using these definitions, we can see five copies of $C^3F$ and one copy of
$C^4F$ as rooted subtrees of the optimal tree in
Figure~\ref{fig:optimal-tree-181}.

Formulating as much as possible using the notations $L$, $F$ and $C^k$ turns
out to give compact representations for optimal trees. Let us formalise this
concept:

\begin{definition}\label{definition:outline-graph}
  Let $T$ be a tree. We construct the \emph{outline graph} of $T$ as follows:
  first, all occurrences $C^kF$ and $C^\ell L$ as rooted subtrees of $T$ are replaced
  by special leaves ``$C^kF$'' and ``$C^\ell L$'',
  respectively (where replacement takes place by decreasing order of the
  replaced rooted subtree). In a second step, we consider all occurrences of subtrees $C^kT'$
  where $T'$ has a unique branch $T''$. Every such subtree is replaced by
  the subtree $T''$, linked to the rest by a special edge ``$C^k_*$''.
\end{definition}

As an example, the outline graph of the tree from
Figure~\ref{fig:optimal-tree-181} is shown in Figure~\ref{fig:optimal-tree-181-outline}.

\begin{figure}[htbp]
  \centering
  \includegraphics{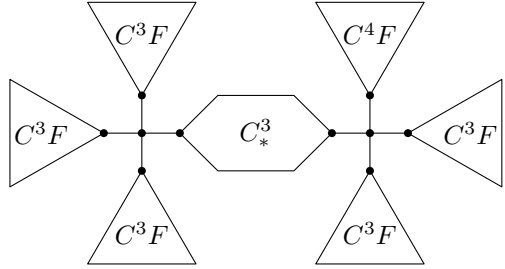}
  \caption{Outline of the unique optimal tree of order $181$.}
  \label{fig:optimal-tree-181-outline}
\end{figure}

%

We are now able to state our main theorem fully describing optimal trees.

\begin{theorem}\label{theorem:global-structure}
  Let $n \ge 4$ and $n\notin\{6,10,13,20,34\}$. Then there is a
  unique optimal tree $T^*_n$ of order $n$.
  \begin{enumerate}
  \item If $n\equiv 1\pmod 7$, then $T^*_n=C^{(n-1)/7}L$.
  \item If $n\equiv 2\pmod 7$, then $T^*_n$ is shown in
    Figure~\ref{fig:optimal-trees-2}, where
    \begin{align*}
      k_0&=\max\left\{0,\floor{\frac{n-37}{35}}\right\},&
      k_j&=
      \begin{cases}
        \floor{\frac{n-2+7j}{35}}&\text{ if $n\ge 37$,}\\
        \floor{\frac{n-9+7j}{35}}&\text{ if $n\le 30$}
      \end{cases}
    \end{align*}
    for $j\in\{1,2,3,4\}$.
  \item If $n\equiv 3\pmod 7$, then $T^*_n$ is shown in
    Figure~\ref{fig:optimal-trees-3}, where
    \begin{equation*}
      k_j=\floor{\frac{n-17+7j}{28}}
    \end{equation*}
    for $j\in\{0,1,2,3\}$.
  \item If $n\equiv 4\pmod 7$, then $T^*_n=C^{(n-4)/7}F$.
  \item If $n\equiv 5\pmod 7$, then $T^*_n$ is shown in
    Figure~\ref{fig:optimal-trees-5}, where
    \begin{equation*}
      k_j=\floor{\frac{n-5+7j}{21}}
    \end{equation*}
    for $j\in\{0,1,2\}$.
  \item If $n\equiv 6\pmod 7$, then $T^*_n$ is shown in Figure~\ref{fig:optimal-trees-6}, where
    \begin{equation*}
      k_j=\floor{\frac{n-27+7j}{49}}
    \end{equation*}
    for $0\le j\le 6$.
  \item If $n\equiv 0\pmod 7$, then $T^*_n$ is shown in
    Figure~\ref{fig:optimal-trees-7}, where
    \begin{equation*}
      k=\frac{n-7}7.
    \end{equation*}
  \end{enumerate}
  If $n\in\{1,2,3,10,13,20\}$, there is also a unique optimal tree $T^*_n$ of
  order $n$. For $n\in\{1,2,3\}$, there is only one tree of order $n$. 
  For $n\in\{10,13,20\}$, $T_n^*$ is shown in
  Figure~\ref{fig:optimal-trees-special-values}.

  For $n\in\{6,34\}$, there are two non-isomorphic optimal trees $T^*_{n,1}$
  and $T^*_{n,2}$ of order $n$. For $n=6$, $T^*_{6,1}$
  (the star of order $6$) and $T^*_{6,2}$ are
  shown in Figure~\ref{fig:optimal-trees-special-values}. 

  For $n=34$, both
  $T_{34,1}^*$ and $T_{34,2}^*$ have the shape as in
  Figure~\ref{fig:optimal-trees-6}. We have
  $(k_0,k_1,k_2,k_3,k_4,k_5,\allowbreak k_6)=(0,0,0,0,0,0,1)$ for $T_{34,1}^*$ (this
  corresponds to the general case $n\equiv 6\pmod 7$ as described above) and
  $(k_0,k_1,k_2,k_3,k_4,k_5,k_6)=(1,0,0,0,0,0,0)$ for $T_{34,2}^*$.
\end{theorem}

\begin{figure}[htbp]
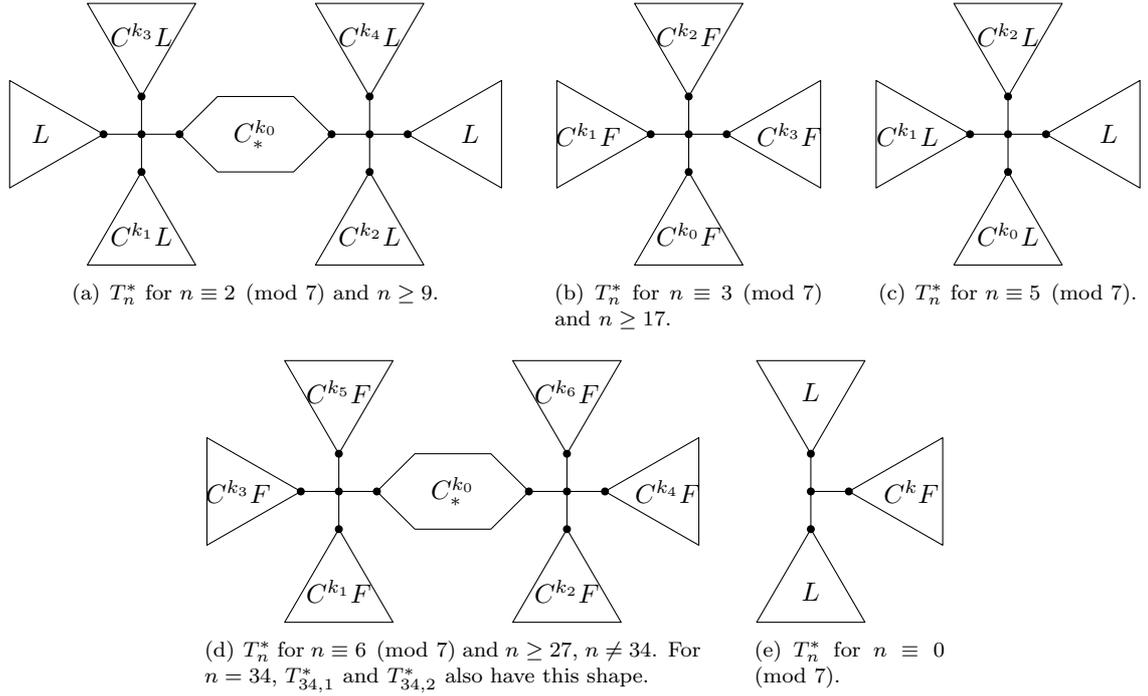

  \centering
  \subfigure[$T_n^*$ for $n\equiv 2\pmod 7$ and $n\ge 9$.]{\label{fig:optimal-trees-2}\includegraphics{optimal-outline.2}}\hfill
  \subfigure[$T_n^*$ for $n\equiv 3\pmod 7$ and $n\ge 17$.]{\label{fig:optimal-trees-3}\includegraphics{optimal-outline.3}}\hfill
  \subfigure[$T_n^*$ for $n\equiv 5\pmod 7$.]{\label{fig:optimal-trees-5}\includegraphics{optimal-outline.5}}\hfill
  \subfigure[$T_n^*$ for $n\equiv 6\pmod 7$ and $n\ge 27$, $n\neq 34$. For
  $n=34$, $T_{34,1}^*$ and $T_{34,2}^*$ also have this shape.]{\label{fig:optimal-trees-6}\includegraphics{optimal-outline.6}}\qquad
  \subfigure[$T_n^*$ for $n\equiv 0\pmod 7$.]{\label{fig:optimal-trees-7}\includegraphics{optimal-outline.7}}
  \caption{Optimal trees.}
  \label{fig:optimal-trees}
\end{figure}

\begin{figure}[htbp]
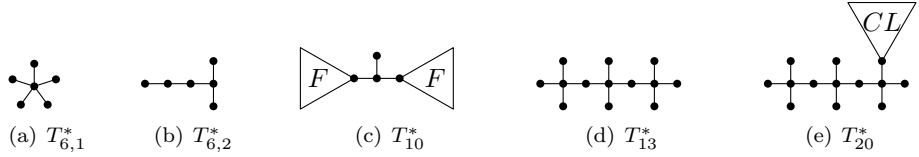

  \centering
  \subfigure[$T_{6,1}^*$]{\includegraphics{optimal_tree_6_1.1}\quad}\qquad
  \subfigure[$T_{6,2}^*$]{\includegraphics{optimal_tree_6_2.1}\quad}\qquad
  \subfigure[$T_{10}^*$]{\includegraphics{optimal_tree_10_1.1}\quad}\qquad
  \subfigure[$T_{13}^*$]{\includegraphics{optimal_tree_13_1.1}\quad}\qquad
  \subfigure[$T_{20}^*$]{\includegraphics{optimal_tree_20_1.1}}
  \caption{Optimal trees for $n\in\{6,10,13,20\}$.}
  \label{fig:optimal-trees-special-values}
\end{figure}

\begin{remark}
The quasi-periodicity of length $7$ is somewhat reminiscent of the situation encountered for dominating sets \cite{brod2006australasian,brod2008recurrence}, even though there are certain differences.
\end{remark}

\section{The upper bound: preliminaries}\label{sec:pre}

\subsection{The bipartition condition}
A tree may always be seen as a bipartite graph. In the case of an optimal tree,
however, the bipartition of the vertices corresponds to a specific behaviour in
terms of maximum matchings, as will be shown in this section. This will also
allow us to somewhat decompose the problem.

We start with a few definitions.

\begin{definition}
  Let $T$ be a forest. The \emph{matching number} $\mu(T)$ is the maximum
  cardinality of a matching of $T$. Hence a matching of $T$ is a maximum matching if it has cardinality $\mu(T)$.
  Denoting the empty graph by $\emptyset$, it is convenient to set $\mu(\emptyset)=0$ and $m(\emptyset)=1$.
\end{definition}

\begin{definition}
  A forest $T$ is called an \emph{optimal forest} if it maximises $m(T)$ over all forests
  of the same order.
\end{definition}

We now define the type of a vertex. These types will later be seen to
correspond to the bipartition of the set of vertices of optimal trees.

\begin{definition}
  Let $T$ be a forest. A vertex $v$ is said to be of type $A$ if $T$ admits a
  maximum matching that does not cover $v$. Otherwise, $v$ is said to be of type $B$.
\end{definition}

A first step towards the main result on the bipartition holds for all trees: there are no
edges between vertices of type $A$:

\begin{lemma}\label{lemma:no-adjacent-vertices-of-type-A}
  Let $T$ be a tree, $s\in V(T)$ of type $A$, and $t$ a neighbour of
  $s$ in $T$. Then $t$ is of type $B$. Denoting the connected components of
  $T-st$ by $T_s$ and $T_t$ with $s\in T_s$ and $t\in T_t$, cf.\ Figure~\ref{fig:adjacent-vertices}, we have
  \begin{align*}
    \mu(T_s-s)&=\mu(T_s),&\mu(T_t-t)&=\mu(T_t)-1,\\
    \mu(T)&=\mu(T_s)+\mu(T_t),&m(T)&=m(T_s)m(T_t)+m(T_s-s)m(T_t-t).
  \end{align*}
\end{lemma}
\begin{figure}[htbp]
  \centering
  \includegraphics{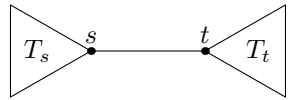}
  \caption{Decomposition of $T$ for Lemma~\ref{lemma:no-adjacent-vertices-of-type-A} and Proposition~\ref{proposition:full-bipartite}.}
  \label{fig:adjacent-vertices}
\end{figure}
\begin{proof}
  We first note that 
  \begin{equation*}
    \mu(T-v)\le \mu(T)\le \mu(T-v)+1
  \end{equation*}
  holds for any vertex $v$ of any tree $T$, 
  as any maximum matching of $T-v$ is a matching of $T$ and any maximum
  matching of $T$ minus possibly the edge covering $v$ is a matching of $T-v$.

  Any maximum matching $M$ of the tree $T$ either contains the edge $st$
  or it does not contain the edge $st$. In the first case, $M$ decomposes into a
  maximum matching of $T_s-s$, a maximum matching of $T_t-t$ and the edge $st$,
  which implies that $\mu(T)=\mu(T_s-s)+\mu(T_t-t)+1$. In the second case, $M$
  decomposes into a maximum matching of $T_s$ and a maximum matching of $T_t$,
  which implies that $\mu(T)=\mu(T_s)+\mu(T_t)$. We conclude that
  \begin{equation}\label{eq:bipartite-maximum-formula}
    \mu(T)=\max\{\mu(T_s-s)+\mu(T_t-t)+1,\mu(T_s)+\mu(T_t)\}.
  \end{equation}

  As $s$ is of type $A$, there is a maximum matching $M$ of $T$ not covering
  $s$, hence we have $\mu(T)=\mu(T_s-s)+\mu(T_t)$. In view of
  \eqref{eq:bipartite-maximum-formula}, this implies that $\mu(T_t)\ge
  \mu(T_t-t)+1$, i.e., $\mu(T_t-t)=\mu(T_t)-1$, and $\mu(T_s-s)\ge \mu(T_s)$,
  i.e., $\mu(T_s-s)=\mu(T_s)$. 

  In this case, we also have
  $\mu(T-t)=\mu(T_t-t)+\mu(T_s)<\mu(T_t)+\mu(T_s)=\mu(T)$, i.e., $t$ is of type
  $B$. Finally, $m(T_s)m(T_t)$ counts the number of maximum matchings of $T$
  not containing $st$ and $m(T_s-s)m(T_t-t)$ counts the number of maximum
  matchings of $T$ containing $st$, their sum is therefore $m(T)$.
\end{proof}

We now show that in almost all cases, optimal forests are trees, so we may
restrict our attention to trees afterwards. Nevertheless, at one point, we will
also use this result as a technical tool when considering trees.

\begin{lemma}\label{lemma:optimal-forest-connected}
  Let $T$ be an optimal forest of order at least $3$. Then $T$ is connected,
  i.e., $T$ is a tree.
\end{lemma}
\begin{proof}
  Let $T_1$ and $T_2$ be connected components of $T$. For simplicity, we may
  assume that these are the only connected components of $T$; otherwise, we use
  the following argument inductively.

  As $T$ is optimal, each of its connected components has to be optimal.

  If both $T_1$ and $T_2$ are of order $1$,
  then they both only admit the empty matching, inserting an edge between these
  two vertices does not alter the number of maximum cardinality matchings. 

  Next, we note that for $n\ge 3$, the star $S_n$ on $n$ vertices satisfies
  $m(S_n)=n-1>1$. Thus an optimal forest of order at least $3$ does not admit a
  perfect matching, as perfect matchings of trees are unique (see Theorem~\ref{thm:lower}).
  This implies that an optimal forest of order at   least $3$ has a vertex of type $A$.
  As the unique vertex of a tree of order $1$ is also of type $A$, we conclude that
  all optimal trees except the tree of order $2$ have a vertex of type $A$.

  As any neighbour of any vertex of type $A$ is of type $B$ by
  Lemma~\ref{lemma:no-adjacent-vertices-of-type-A} and the vertices of the tree of order $2$
  also are of type $B$, we conclude that every optimal tree of order at least
  $2$ has a vertex of type $B$.

  If $T_1$ and $T_2$ are both of order $2$, then there is no vertex of type
  $A$, thus $T$ is not optimal.

  So we may now assume that $v\in T_1$ is of type $A$ and $w\in T_2$ is of type
  $B$. If we insert the edge $vw$, we obtain a new graph $T'=T+vw$. As in
  Lemma~\ref{lemma:no-adjacent-vertices-of-type-A}, we obtain
  \begin{equation*}
    \mu(T')=\max\{\mu(T_1-v)+\mu(T_2-w)+1, \mu(T_1)+\mu(T_2) \}.
  \end{equation*}
  As $v$ is of type $A$ (with respect to $T_1$) and $w$ is of type $B$ (with
  respect to $T_2$), we have $\mu(T_1-v)=\mu(T_1)$ and
  $\mu(T_2-w)+1=\mu(T_2)$. This implies that $\mu(T')=\mu(T_1-v)+\mu(T_2)$,
  i.e., $v$ is of type $A$ with respect to $T'$ and
  Lemma~\ref{lemma:no-adjacent-vertices-of-type-A} can be applied to yield
  \begin{equation*}
    m(T')=m(T_1)m(T_2)+m(T_1-v)m(T_2-w)>m(T_1)m(T_2)=m(T),
  \end{equation*}
  contradiction.

  Thus the only disconnected optimal forest is the forest consisting of exactly two isolated vertices.
\end{proof}

We can now formalise what we will call the bipartition condition.

\begin{definition}
  Let $T$ be a tree. We say that $T$ fulfils the \emph{bipartition condition} if the two classes
  in $T$'s unique bipartition contain precisely the vertices of type $A$ and $B$ respectively.
\end{definition}

It turns out that indeed almost all optimal trees satisfy this condition.

\begin{proposition}\label{proposition:full-bipartite}
  Let $T$ be an optimal tree of order at least $3$. Then $T$ fulfils the
  bipartition condition.

  Let $st$ be an edge of $T$ where $s$ is of type $A$ and $t$ is of type $B$.
  The connected components of $T-st$ are denoted by $T_s$ and $T_t$ with $s\in
  T_{s}$ and $t\in T_t$. Then $s$ and $t$ are of types $A$ and $B$ with respect
  to the trees $T_s$ and $T_t$, respectively. Furthermore, 
  \begin{equation}\label{eq:full-bipartite-m-formula}
    m(T)=m(T_s)m(T_t)+m(T_s-s)m(T_t-t).
  \end{equation}
\end{proposition}
\begin{proof}
  Assume that $s$ and $t$ are two adjacent vertices of type $B$. 

  If we have $\mu(T_s-s)+\mu(T_t-t)+1=\mu(T_s)+\mu(T_t)$, then (w.l.o.g.)
  $\mu(T_s-s)=\mu(T_s)$ and $\mu(T_t-t)=\mu(T_t)-1$. In this case, we obtain
  $\mu(T)=\mu(T_s-s)+\mu(T_t)$, i.e., $s$ is of type $A$. Contradiction.

  Next, we consider the case that $\mu(T_s-s)+\mu(T_t-t)+1<\mu(T_s)+\mu(T_t)=\mu(T)$,
  i.e., the case that $st$ is not contained in any maximum matching of
  $T$. Deleting the edge $st$ resulting in a forest $T'=T-st$ does
  not alter the number of maximum matchings, i.e., $m(T)=m(T')$. By
  Lemma~\ref{lemma:optimal-forest-connected}, $T'$ and therefore $T$  are not
  optimal, contradiction.

  Finally, we consider the case
  $\mu(T_s-s)+\mu(T_t-t)+1>\mu(T_s)+\mu(T_t)=\mu(T)$, i.e., the case that $st$
  is contained in every maximum matching of $T$. Deleting all edges incident to
  $s$ or $t$ leads to a disconnected forest of the same order and the same number
  of maximum matchings. Contradiction.

  Thus exactly one of $s$ and $t$, say $s$, is of type $A$ by
  Lemma~\ref{lemma:no-adjacent-vertices-of-type-A} and the remaining assertions
  of this proposition are restatements of the results of
  Lemma~\ref{lemma:no-adjacent-vertices-of-type-A}.
\end{proof}

\subsection{Rooted Trees}
For many of our arguments, we will designate a vertex of a tree as the root and
recursively consider subtrees. To this end, we collect a few definitions as
well as some recursive formul\ae{} for the number of maximum matchings.

We assume that all rooted trees are non-empty. A rooted tree with underlying
tree $T$ and root $r$ will be denoted by the pair $(T,r)$; frequently, we will
simply write $T$ if the root is clear from the context. An important operation that we
will frequently apply is to choose another vertex $s\in V(T)$ as the new root. We will usually
denote the resulting rooted tree by a new symbol $(T',s)$ (and thus abbreviated
to $T'$) although the underlying unrooted trees $T$ and $T'$ are identical.

As usual, the branches of a rooted tree $(T,r)$ of the shape as
Figure~\ref{fig:rooted-tree-with-branches} are the rooted trees $(T_1,r_1)$, \ldots, $(T_k,r_k)$.
\begin{figure}[htbp]
  \centering
  \includegraphics{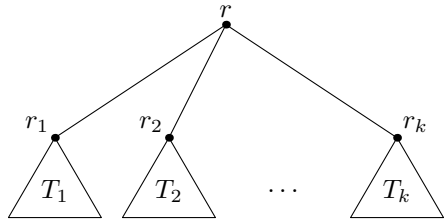}
  \caption{Rooted tree with branches.}
  \label{fig:rooted-tree-with-branches}
\end{figure}

A \emph{rooted subtree} $(T',v)$
of an unrooted tree $T$ is a connected component of $T-vw$ for some edge $vw$ of $T$ such
that $v\in T'$. Note that this definition forces
$T'$ to be a proper subtree of $T$.

A rooted subtree $(T',v)$ of a rooted tree $(T,r)$ is the
connected component of $T-vw$ containing $v$, where $w$ has to be the parent of
$v$, i.e., $T'$ is the subgraph induced by all the successors of $v$. We will
also write $T'=T(v)$ in this case. 

Let $T$ be a tree and $v$ be a vertex of $T$ with neighbours $r_1$, \ldots, 
$r_k$. The connected components of $T-v$ are denoted by $T_1$, \ldots, $T_k$ such that $r_j\in T_j$ for all $j$. Then the
rooted trees $(T_1,r_1)$, \ldots, $(T_k,r_k)$ are said to be the \emph{rooted
  connected components of $T-v$} (and usually, the roots $r_j$ will  not be mentioned).

\begin{definition}
  Let $(T,r)$ be a rooted tree. 
  \begin{enumerate}
  \item We define $m_1(T)$ to be the number of maximum
    matchings of $T$ covering the root $r$.
  \item We define $m_0(T)$ to be $m(T-r)$, the number of maximum matchings of $T-r$.
  \item The \emph{type} of $T$ is defined to be the type of the root as a
    vertex of the unrooted tree, i.e., $(T,r)$ is of type $A$ if
    $\mu(T-r)=\mu(T)$ and of type $B$ if $\mu(T-r)=\mu(T)-1$. We sometimes write
    $\type(T)=A$ and $\type(T)=B$, respectively.
  \end{enumerate}
\end{definition}

Thus $(T,r)$ is of type $A$ if and only if it admits a maximum matching not
covering the root $r$.

We have
\begin{equation}\label{eq:special-values-L}
 \mu(L)=0,  \qquad m(L)=1, \qquad m_0(L)=1, \qquad m_1(L)=0
\end{equation}
for the rooted tree $L$ of order $1$, which implies that it is a rooted tree of type
$A$. 

\begin{definition}
    We define the \emph{bipartition condition for rooted trees} recursively as follows:
    a rooted tree of order $1$ (rooted at its only vertex) is said to satisfy the bipartition
    condition. If $(T,r)$ is a rooted tree with branches $(T_1,r_1)$, \ldots,
    $(T_k,r_k)$, then the rooted tree $(T,r)$ is said to fulfil the bipartition condition if
    all branches $(T_j,r_j)$ fulfil the bipartition condition and the type of $(T,r)$ is not
    equal to the type of any of the branches $(T_j,r_j)$.
\end{definition}

\begin{remark}
  Let $T$ be an optimal tree of order at least $3$ and $(S,r)$ be a rooted
  subtree of $T$. Then the type of $r$ as vertex of $T$ coincides with the type
  of $S$ and $S$ fulfils the bipartition condition for rooted trees by Proposition~\ref{proposition:full-bipartite}.
\end{remark}

The main goal behind the definition of the two different types is to provide a recursive method to compute
$m(T)$. Note first that for a rooted tree $(T,r)$, we have
\begin{equation*}
  m(T)=\begin{cases}
    m_0(T)+m_1(T),&\text{ if $(T,r)$ is of type $A$},\\
    m_1(T),&\text{ if $(T,r)$ is of type $B$}.
    \end{cases}
\end{equation*}
We now give recursive formul\ae{} for these quantities in terms of the
branches of a rooted tree. Here, for technical reasons, we do not assume the bipartition condition for
rooted trees, but a weaker version only, and derive the bipartition condition for
rooted trees. 

\begin{lemma}\label{lemma:recursive-formulae-rooted-tree}
  Let $(T,r)$ be a rooted tree and $(T_1,r_1)$, \ldots, $(T_k,r_k)$
  its branches.  We assume
  that $T_1$, \ldots, $T_k$ are of the same type. Then $T$ is of the other type
  and we have
  \begin{align}
    m_0(T)&=\prod_{j=1}^k m(T_j)\label{equation:m_0-formula},\\
    m_1(T)&=m_0(T)\cdot \sum_{j=1}^k\frac{m_0(T_j)}{m(T_j)}.\label{equation:m_1-formula}
  \end{align}
\end{lemma}
\begin{proof}
  If $(T,r)$ is of order $1$, then there are no branches, and the product in
  \eqref{equation:m_0-formula} and the sum in \eqref{equation:m_1-formula} are
  empty, which coincides with the values for $L$ given in
  \eqref{eq:special-values-L}. Thus we may focus on the case that the order
  of $(T,r)$ is at least $2$.

  As $T-r$ consists of the connected components $T_1$, \ldots, $T_k$, we
  clearly have $\mu(T-r)=\sum_{i=1}^k \mu(T_i)$ and
  \eqref{equation:m_0-formula}. Furthermore,
  \begin{equation*}
    \mu(T)=\max\Bigl(\{\mu(T-r)\} \cup
    \Bigl\{1+\mu(T_j-r_j)+\sum_{i\neq j}\mu(T_i): j\in\{1,\ldots,k\}\Bigr\}\Bigr),
  \end{equation*}
  as a maximum matching
  either does not cover $r$ or contains the edge $rr_j$ for some $j$.

  If all branches are of type $B$, i.e.,  $\mu(T_j-r_j)=\mu(T_j)-1$ for all
  $j$, then $1+\mu(T_j-r_j)+\sum_{i\neq j}\mu(T_i)=\sum_i\mu(T_i)=\mu(T-r)$ for
  all $j$. This implies that $\mu(T)=\mu(T-r)$, $T$ is of type $A$ and each of the
  edges $rr_j$ can be used in a maximum matching.

  If all branches are of type $A$, i.e., $\mu(T_j-r_j)=\mu(T_j)$ for all $j$,
  then $1+\mu(T_j-r_j)+\sum_{i\neq j}\mu(T_i)=1+\sum_i\mu(T_i)=1+\mu(T-r)$ for
  all $j$. This implies that $\mu(T)=\mu(T-r)+1$, $T$ is of type $B$ and again,
  each of the edges $rr_j$ can be used in a maximum matching.

  There are $m(T_1)\ldots
  m(T_{j-1})m_0(T_j)m(T_{j+1})\ldots m(T_k)$ maximum matchings of $T$ containing the
  edge $rr_j$. Summing over all $j$ yields \eqref{equation:m_1-formula}.
\end{proof}

If $T_1$, \ldots, $T_k$ are rooted trees of type $A$, then the rooted tree with
branches $T_1$, \ldots, $T_k$ is also denoted by $\calB(T_1,\ldots,
T_k)$. It is of type $B$ by Lemma~\ref{lemma:recursive-formulae-rooted-tree}.

Similarly, if $T_1$, \ldots, $T_k$ are rooted trees of type $B$, then the rooted tree with
branches $T_1$, \ldots, $T_k$ is also denoted by $\calA(T_1,\ldots,T_k)$. 
It is of type $A$ by Lemma~\ref{lemma:recursive-formulae-rooted-tree}. If
$k=1$, we will omit the parentheses and simply write $\calA T_1$.

The crucial quantity in our investigation will be the following quotient:

\begin{definition}
  For a rooted tree $(T,r)$, we set $\rho(T)=m_0(T)/m(T)$. 
\end{definition}

We note that by definition, $\rho(T)>0$ for all rooted trees $(T,r)$. 

We now reformulate the recursive formul\ae{} for $m$ and $m_0$ to yield
recursive formul\ae{} for $\rho$.

\begin{lemma}\label{lemma:rho-formulae}
  Let $(T,r)$ be a rooted tree fulfilling the bipartition condition with branches $(T_1,r_1)$, \ldots,
  $(T_k,r_k)$. Then
  \begin{equation*}
    \rho(T)=\begin{cases}
      \dfrac{1}{ 1+\sum\limits_{j=1}^k \rho(T_j)},&\text{ if $(T,r)$ is of type $A$},\\
      \dfrac{1}{ \sum\limits_{j=1}^k \rho(T_j)},&\text{ if $(T,r)$ is of type $B$}.
      \end{cases}
  \end{equation*}
\end{lemma}
\begin{proof}
  This is a simple consequence of \eqref{equation:m_0-formula} and
  \eqref{equation:m_1-formula}.
\end{proof}

\subsection{\texorpdfstring{$\alpha$}{\textalpha}-optimality}

It turns out that a rooted subtree of an optimal tree no longer needs to be
optimal. Instead, we introduce the auxiliary notion of $\alpha$-optimality.

\begin{definition}
  Let $\alpha$ be a non-negative real number. A rooted tree $(T,r)$ is said to
  be $\alpha$-optimal if it fulfils the bipartition condition and if 
  \begin{multline}\label{eq:alpha-optimality}
    m(T)+\alpha m_0(T)=
    \max\{ m(T')+\alpha m_0(T') :(T',r') \text{ is a rooted tree}\\ \text{fulfilling the bipartition condition with }|T|=|T'|\text{ and }\type(T)=\type(T') \}.
  \end{multline}
\end{definition}

Note that $0$-optimality is just ordinary optimality. This definition is motivated by the fact that any rooted subtree of an
optimal tree is indeed $\alpha$-optimal for an appropriate value of $\alpha$:

\begin{proposition}\label{proposition:subtrees-of-optimal-trees-are-alpha-optimal}
  Let $T$ be an optimal tree, $st$ an edge of $T$ and $T_s$ and $T_t$ the
  connected components of $T-st$, with $s\in T_s$ and $t\in T_t$. Then 
  $(T_s,s)$ is a $\rho(T_t)$-optimal tree and $(T_t,t)$ is a
  $\rho(T_s)$-optimal tree.
\end{proposition}
\begin{proof}
  If the order of $T$ is $\leq 2$, the statement holds trivially.

  Reformulating \eqref{eq:full-bipartite-m-formula} in terms of the function
  $\rho$ shows that
  \begin{align*}
    m(T)&=m(T_t)(m(T_s)+\rho(T_t)m_0(T_s)).
  \end{align*}
  If $T_s$ was not $\rho(T_t)$-optimal, we could replace it by a
  $\rho(T_t)$-optimal tree and this would increase $m(T)$, contradiction. The
  same argument applies to $T_t$.
\end{proof}

We note the fact that $\rho(T)\le 1$ holds for all rooted trees of type $A$ by Lemma~\ref{lemma:rho-formulae}, where
equality holds if and only if $T=L$. Thus, by
Proposition~\ref{proposition:subtrees-of-optimal-trees-are-alpha-optimal}, we
may restrict ourselves to the investigation of $\alpha$-optimal trees of type $A$
with $\alpha\in[0,\infty)$ as well as $\alpha$-optimal trees of type $B$ with
$\alpha\in[0,1]$.

A few rooted trees will be considered repeatedly in our proofs. These are shown
in Figure~\ref{figure:important-rooted-trees}. One could indeed show that these
trees are $\alpha$-optimal for some $\alpha>0$, but we do not need this
information. On the other hand, we will later need to know that some rooted trees are
not $\alpha$-optimal for some ranges of $\alpha$. We list these trees (together
with a replacement $T'$) in Table~\ref{tab:replacements-alpha-optimal} in the appendix. 
Similarly, we list a few non-optimal trees in
Table~\ref{tab:replacements-optimal}, where $T_n^*$ is given in Theorem~\ref{theorem:global-structure}.
We will simply refer to the entries of these two tables by
\rref{replacement-full:4-claw-is-bad} to
\rref{replacement:outline-ends-LCFT3T4}. These tables can be verified using a
Sage~\cite{Stein-others:2010:sage-mathem} program available in \cite{Heuberger-Wagner:max-card-matching-sage}.

\let\incgra=\includegraphics
\newcommand{\alphaoptimal}[6]{\parbox[b]{0.19\linewidth}{\noindent
    \incgra{#1}\par\noindent #2: $(#4,#5)$\iffalse\\\noindent $\alpha#3$\fi}}
\newcommand{\alphaoptimalb}[6]{\parbox[b]{0.32\linewidth}{\noindent
    \incgra{#1}\par\noindent #2: $(#4,#5)$\iffalse\\\noindent $\alpha#3$\fi}}
\begin{figure}[htbp]
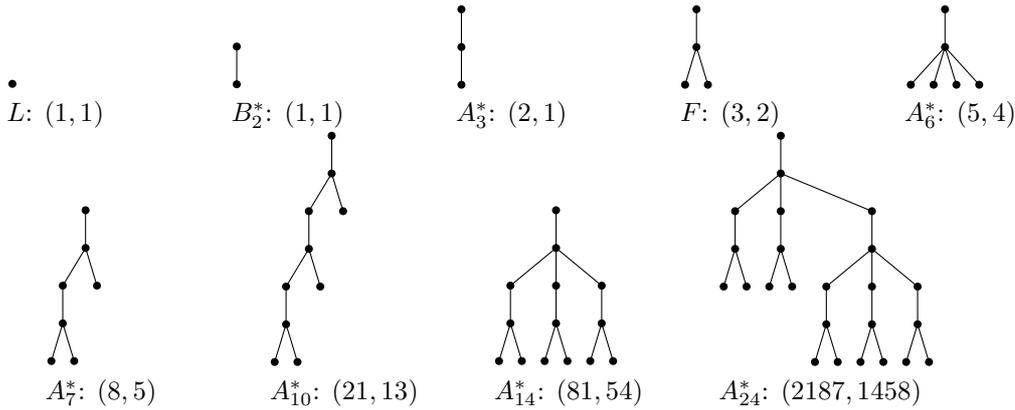

  \centering
\alphaoptimal{alpha-optimal-tree-type_A-d_1-alphamin_0_1-alphamax_Infinity_1-m_1-m0_1-rho_1_1.1}{$L$}{\in(0,\infty )}{1}{1}{1}
\alphaoptimal{alpha-optimal-tree-type_B-d_2-alphamin_0_1-alphamax_1_1-m_1-m0_1-rho_1_1.1}{$B_{2}^*$}{\in(0,1]}{1}{1}{1}
\alphaoptimal{alpha-optimal-tree-type_A-d_3-alphamin_0_1-alphamax_Infinity_1-m_2-m0_1-rho_1_2.1}{$A_{3}^*$}{\in(0,\infty )}{2}{1}{\frac{1}{2}}
\alphaoptimal{alpha-optimal-tree-type_A-d_4-alphamin_0_1-alphamax_Infinity_1-m_3-m0_2-rho_2_3.1}{$F$}{\in(0,\infty )}{3}{2}{\frac{2}{3}}
\alphaoptimal{alpha-optimal-tree-type_A-d_6-alphamin_0_1-alphamax_Infinity_1-m_5-m0_4-rho_4_5.1}{$A_{6}^*$}{\in(0,\infty )}{5}{4}{\frac{4}{5}}
\alphaoptimal{alpha-optimal-tree-type_A-d_7-alphamin_0_1-alphamax_Infinity_1-m_8-m0_5-rho_5_8.1}{$A_7^*$}{}{8}{5}{}
\alphaoptimal{alpha-optimal-tree-type_A-d_10-alphamin_0_1-alphamax_1_1-m_21-m0_13-rho_13_21.1}{$A_{10}^*$}{}{21}{13}{}
\alphaoptimal{AAR-14-1.1}{$A_{14}^*$}{\in(0,2]}{81}{54}{\frac{2}{3}}
\alphaoptimalb{AAR-24-1.1}{$A_{24}^*$}{\in(0,\frac{28}{51}]}{2187}{1458}{\frac{2}{3}}
\caption{Some important rooted trees. All trees are given with the
    pair $(m(T),m_0(T))$.}
  \label{figure:important-rooted-trees}
\end{figure}

\subsection{Exchanging Subtrees}

In order to derive information on the structure of optimal trees, we will
compare optimal trees with trees where some rooted subtrees have been
exchanged. In order to estimate the effect of such exchange operations, we need
an extension of our recursive formul\ae{} \eqref{equation:m_0-formula} and
\eqref{equation:m_1-formula} to finer decompositions of a tree. These
extensions will be formulated in terms of continuants and continued fractions. 

We therefore fix some notations and definitions in the context of continuants
and continued fractions. We follow Graham, Knuth and
Patashnik~\cite{Graham-Knuth-Patashnik:1994}, Section~6.7.

\begin{definition}[{\cite[(6.127)]{Graham-Knuth-Patashnik:1994}}]
  The \emph{continuant polynomial} $\CFP_n(x_1,\ldots,x_n)$ has $n$ parameters,
  and it is defined by the following recurrence:
  \begin{equation}\label{eq:continuant-recursion-right}
    \CFP_n(x_1,\ldots,x_n)=\CFP_{n-1}(x_1,\ldots,x_{n-1})x_n+\CFP_{n-2}(x_1,\ldots,x_{n-2})
  \end{equation}
  for $n\ge 2$ and $\CFP_0()=1$, $\CFP_1(x_1)=x_1$.
\end{definition}
We will omit the index $n$ in $\CFP_n$ whenever it is clear from the
context. 

We need the following additional properties of continuants:
\begin{lemma}
  We have
  \begin{align}
    \CFP(x_1,\ldots,x_n)&=\CFP(x_n,\ldots, x_1)\label{eq:continuant-symmetry}\\
    \CFP_n(x_1,\ldots,x_n)&=x_1\CFP_{n-1}(x_2,\ldots,x_n)+\CFP_{n-2}(x_3,\ldots,x_n),\label{eq:continuant-recursion-left}
  \end{align}
\end{lemma}
\begin{proof}
  The symmetry relation~\eqref{eq:continuant-symmetry} is
  \cite[(6.131)]{Graham-Knuth-Patashnik:1994}, the
  recursion~\eqref{eq:continuant-recursion-left} is a consequence of the
  symmetry relation~\eqref{eq:continuant-symmetry} and the defining
  recursion~\eqref{eq:continuant-recursion-right}, cf.\ \cite[(6.132)]{Graham-Knuth-Patashnik:1994}.
\end{proof}

The following lemma shows how continuants can be used to determine $m(T)$.
We use the Iversonian notation $[\mathit{expr}]=1$ if
$\mathit{expr}$ is true and $[\mathit{expr}]=0$ otherwise, cf.\ Knuth~\cite{Knuth:1992:two-notes-notat}.

\begin{figure}[htbp]
  \centering
  \includegraphics{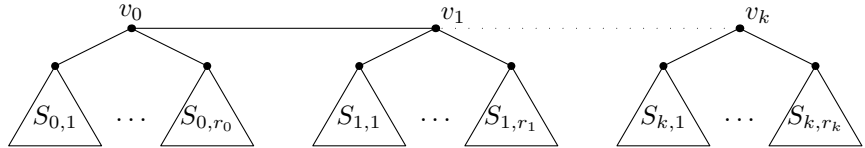}
  \caption{Shape of $T$ for the exchange lemma.}
  \label{fig:shape-exchange-lemma}
\end{figure}

\begin{lemma}\label{lemma:m-m-0-via-continuants}
  Let $T$ be a tree fulfilling the bipartition condition of the shape given
  in Figure~\ref{fig:shape-exchange-lemma}
  for some $k\ge 0$, integers $r_i\ge 0$ for $0\le i\le k$, and rooted trees
  $S_{i,j}$, $0\le i\le k$, $1\le j\le r_i$.

  Then 
  \begin{equation*}
    m(T)=\CFP(\rho_0,\rho_{1},\ldots,\rho_{k-1},\rho_k)\prod_{h=0}^k\prod_{j=1}^{r_h}m(S_{h,j}),\\
  \end{equation*}
  where 
  \begin{equation*}
    \rho_i=[\type v_i=A]+\sum_{j=1}^{r_i}\rho(S_{i,j}).
  \end{equation*}

\end{lemma}
\begin{proof}
  We set
  \begin{equation*}
    M_i=\prod_{h=i}^k \prod_{j=1}^{r_h}m(S_{h,j})
  \end{equation*}
  and consider $v_0$ as root of $T$. We claim that
  \begin{align*}
    m(T(v_i))&=M_i\CFP(\rho_i,\rho_{i+1},\ldots,\rho_{k-1},\rho_k),\\
    m_0(T(v_i))&=M_i\CFP(\rho_{i+1},\ldots,\rho_{k-1},\rho_k)
  \end{align*}
  holds for $0\le i\le k$. This can be shown by reverse induction on $i$ using only the
  recursive formul\ae{} \eqref{equation:m_0-formula},
  \eqref{equation:m_1-formula} and \eqref{eq:continuant-recursion-left}.
\end{proof}

We now turn to continued fractions.

\begin{definition}
  We set
  \begin{equation*}
    \CF(x_0,x_1,\ldots,x_n)=
    x_0+\cfrac1{x_1+\cfrac1{x_2+\cfrac1{\ddots+\cfrac1{x_n}}}}.
  \end{equation*}
  As usual, for a sequence $(x_k)_{k\ge 0}$, the infinite continued fraction
  $\CF(x_0,x_1,\ldots)$ is defined as the limit
  $\lim_{k\to\infty}\CF(x_0,x_1,\ldots,x_k)$.
\end{definition}

The connection between continuants and continued fractions is stated
in the following result.
\begin{lemma}[{\cite[(6.136)]{Graham-Knuth-Patashnik:1994}}]\label{lemma:continuant-continued-fraction}We have
  \begin{equation*}
  \CF(x_0,x_1,\ldots,x_n)=\frac{\CFP(x_0,x_1,\ldots,x_n)}{\CFP(x_1,\ldots,x_n)}.
\end{equation*}
\end{lemma}

We are now able to formulate our main exchange lemma. It comes in several
flavours: First, the most general version is stated, which might be cumbersome
to use. Next, in a mostly symmetric case, we get a neat formulation, which will
be frequently used. Finally, we give two estimates for the asymmetric case,
which are not best possible, but sufficient for our purposes. For these
estimates, we make some assumptions on the occurring values of $\rho$ which
will be fulfilled in the applications later on.

\begin{lemma}\label{lemma:exchange}
  Let $T$ be an optimal tree of the shape given
  in Figure~\ref{fig:shape-exchange-lemma}
  for some even $k\ge 2$, integers $r_i\ge 0$ for $0\le i\le k$, and rooted trees
  $S_{i,j}$, $0\le i\le k$, $1\le j\le r_i$. We set
  \begin{equation*}
    \rho_i=[\type v_i=A]+\sum_{j=1}^{r_i}\rho(S_{i,j}).
  \end{equation*}

  Let $0\le s_0\le r_0$ and $0\le s_k\le r_k$ and set
  \begin{align*}
    x&:=\sum_{j=1}^{s_0}\rho(S_{0,j}),&
    y&:=\sum_{j=1}^{s_k}\rho(S_{k,j}),\\
    a&:=[\type v_0=A]+\sum_{j=s_0+1}^{r_0}\rho(S_{0,j}),&
    b&:=[\type v_k=A]+\sum_{j=s_k+1}^{r_k}\rho(S_{k,j})
  \end{align*}
  so that $\rho_0=x+a$ and $\rho_k=y+b$. Assume that $y+a>0$,
  $x+b>0$ and $\rho_1$, $\rho_{k-1} > 0$.

  \begin{enumerate}
  \item\label{item:exchange-general} If $x>y$, then
    \begin{equation}\label{eq:exchange-lemma-full-form}
      \CF(a,\rho_1,\rho_2,\ldots,\rho_{k-2},\rho_{k-1})\le\CF(b,\rho_{k-1},\rho_{k-2},\ldots,\rho_2,\rho_1).
    \end{equation}
  \item\label{item:exchange-symmetric} If $x>y$ and $(\rho_1,\ldots,\rho_{k-1})=(\rho_{k-1},\ldots,\rho_1)$,
    then $a\le b$.
  \item\label{item:exchange-B} If $x>y$, $\rho_j=1$ for odd $j$ and $\ell\le \rho_j\le u$ for all even $j$
    with $2\le j\le k-2$ and for fixed
    $0< \ell\le u$, then
    \begin{equation*}
      a< b+U_0(\ell,u),
    \end{equation*}
    where 
    \begin{equation*}
      U_0(\ell,u)=
        \frac1{\CF(1,u,1,u,1)}-\frac1{\CF(1,\ell,1,\ell,\ldots)}.
    \end{equation*}
    In particular, we have
    \begin{align*}
      U_0(1,2)&<0.1153,&
      U_0(2,3)&<0.0597,&
      U_0(3,4)&<0.0373.
    \end{align*}
  \item\label{item:exchange-A} If $x>y$, $\rho_j=1$ for even $j$ with $2\le j\le k-2$ and
    $\ell\le \rho_j\le u$ for all odd $j$ and for fixed $0\le \ell\le u$, then
    \begin{equation*}
      a< b+U_1(\ell,u),
    \end{equation*}
    where
    \begin{equation*}
      U_1(\ell,u)=\frac1{\CF(\ell,1,\ell)}-\frac1{\CF(u,1,u,1,\ldots)}.
    \end{equation*}
    In particular, we have
    \begin{align*}
      U_1(1,2)&<0.3007,&
      U_1(2,3)&<0.1113,&
      U_1(3,4)&<0.0596.
    \end{align*}
  \end{enumerate}
\end{lemma}
\begin{proof}
  \begin{enumerate}
  \item Set $M=\prod_{h=0}^k\prod_{j=1}^{r_h}m(S_{h,j})$ and
    let $T'$ be the tree arising from $T$ by exchanging $S_{0,1}$, \ldots,
    $S_{0,s_0}$ against $S_{k,1}$, \ldots,
    $S_{k,s_k}$. As $k$ is even and $a+y>0$ and $b+x>0$, the types of all $v_j$ are the same in
    $T$ and $T'$. As $T$ is an optimal tree, we have
    \begin{align*}
      0&\le
      \frac{m(T)-m(T')}{M}\\
      &=\CFP(x+a,\rho_1,\ldots,\rho_{k-1},y+b)-\CFP(y+a,\rho_1,\ldots,\rho_{k-1},x+b)\\
      &=\bigl((x+a)(y+b)-(y+a)(x+b)\bigr)\CFP(\rho_1,\ldots,\rho_{k-1})\\
      &\qquad+\bigl((x+a)-(y+a)\bigr)\CFP(\rho_1,\ldots,\rho_{k-2})\\
      &\qquad+\bigl((y+b)-(x+b)\bigr)\CFP(\rho_2,\ldots,\rho_{k-1})\\
      &=(x-y)\CFP(\rho_1,\ldots,\rho_{k-1})\left(b-a+\frac{\CFP(\rho_1,\ldots,\rho_{k-2})}{\CFP(\rho_1,\ldots,\rho_{k-2},\rho_{k-1})}-\frac{\CFP(\rho_2,\ldots,\rho_{k-1})}{\CFP(\rho_1,\rho_2,\ldots,\rho_{k-1})}\right)\\
      &=(x-y)\CFP(\rho_1,\ldots,\rho_{k-1})\left(b-a+\frac1{\CF(\rho_{k-1},\ldots,\rho_1)}-\frac1{\CF(\rho_1,\ldots,\rho_{k-1})}\right)\\
      &=(x-y)\CFP(\rho_1,\ldots,\rho_{k-1})\left(\CF(b,\rho_{k-1},\ldots,\rho_1)-\CF(a,\rho_1,\ldots,\rho_{k-1})\right)
    \end{align*}
    by Lemma~\ref{lemma:m-m-0-via-continuants},
    \eqref{eq:continuant-recursion-right},
    \eqref{eq:continuant-recursion-left}, \eqref{eq:continuant-symmetry},
    Lemma~\ref{lemma:continuant-continued-fraction} and the obvious recursion
    formula for continued fractions. The result follows upon division by the
    positive quantity $(x-y)\CFP(\rho_1,\ldots,\rho_{k-1})$.
  \item The symmetry implies that
    $\CF(\rho_{k-1},\ldots,\rho_1)=\CF(\rho_1,\ldots,\rho_{k-1})$
    and the result follows from \eqref{eq:exchange-lemma-full-form}.
  \item If $k\le 4$, then the assertion follows from
    \liref{lemma:exchange}{item:exchange-symmetric}. So we may assume $k\ge 6$.
    By \eqref{eq:exchange-lemma-full-form} we have
    \begin{align*}
      a+\frac1{\CF(1,\ell,1,\ell,\ldots)}&=
      \CF(a,1,\ell,1,\ell,\ldots)<
      \CF(a,1,\ell,\ldots,\ell,1)\\&\le
      \CF(a,1,\rho_2,\ldots,\rho_{k-2},1)\\&\le
      \CF(b,1,\rho_{k-2},1,\rho_{k-4},1,\ldots,\rho_2,1)\\&\le
      \CF(b,1,u,1,u,1,\ldots,u,1)\\&\le
      \CF(b,1,u,1,u,1)=b+\frac1{\CF(1,u,1,u,1)},
    \end{align*}
    as decreasing the entries
    at even-numbered indices of a continued fraction $\CF(x_0,x_1,\ldots)$ decreases the continued
    fraction, and increasing entries at odd-numbered indices also decreases the
    continued fraction. 
  \item If $k\le 2$, then the assertion follows from
    \liref{lemma:exchange}{item:exchange-symmetric}. So we may assume $k\ge 4$.
    By \eqref{eq:exchange-lemma-full-form} we have
    \begin{align*}
      a+\frac1{\CF(u,1,u,1,\ldots)}&=
      \CF(a,u,1,u,1,\ldots)<
      \CF(a,u,1,\ldots,1,u)\\&\le
      \CF(a,\rho_1,1,\ldots,1,\rho_{k-1})\le
      \CF(b,\rho_{k-1},1,\ldots,1,\rho_1)\\&\le
      \CF(b,\ell,1,\ell)=
      b+\frac1{\CF(\ell,1,\ell)}.
    \end{align*}
  \end{enumerate}
\end{proof}

This exchange lemma will be used repeatedly in the following to deduce information about the structure of optimal trees. To simplify explanations, we will call the vertices $v_0$ and $v_k$ in Figure~\ref{fig:shape-exchange-lemma} \emph{pivotal vertices}. 

\section{The upper bound: local structure}\label{sec:local}

We have now gathered enough auxiliary tools to start with the proof of Theorem~\ref{theorem:global-structure} and thus
Theorem~\ref{thm:upper}. To abbreviate some statements, we introduce the following definitions.

\begin{definition}
  Let $T$ be a tree. We say that it fulfils the \emph{local
    conditions} (LC), if all of the following conditions are
  fulfilled:
  \begin{enumerate}
  \item[(LC1)] $T$ fulfils the bipartition condition,
  \item[(LC2)] each vertex of type $A$ has degree $1$ or $2$,
  \item[(LC3)] each vertex of type $B$ has degree at least $3$,
  \item[(LC4)] each vertex of degree $3$ is adjacent to at least two leaves,
  \item[(LC5)] each vertex has degree at most $4$,
  \item[(LC6)] no vertex is adjacent to $3$ leaves.
  \end{enumerate}
\end{definition}
By Proposition~\ref{proposition:full-bipartite}, an optimal tree of order $\ge
3$ fulfils LC1.

The following theorem will be shown step by step in Sections~\ref{section:structure},
\ref{section:light-trees} and \ref{section:heavy-trees}:
\begin{theorem}\label{theorem:local-structure-again}
  Let $\calS=\{T_{2}^*, T_{3}^*, T_5^*, T_{6,1}^*, T_{6,2}^*, T_8^*, T_9^*,
  T_{10}^*, T_{12}^*, T_{13}^*, T_{16}^*, T_{20}^*\}$ and $T$ be an optimal tree with
  $T\notin\calS$.

  Then $T$ fulfils the local conditions LC1--LC6.
\end{theorem}

We note that it is debatable whether LC6 shall be considered to be part of the \emph{local
  structure} as $T_8^*$, $T_9^*$, $T_{12}^*$, $T_{16}^*$ fulfil LC1--LC5 and
are contained in the generic cases described in
Theorem~\ref{theorem:global-structure}. So these trees may simply be seen as degenerated cases of the generic cases
even though LC6 is violated. On the
other hand, $T_{13}^*$ and $T_{20}^*$ fulfil LC1--LC5, but not LC6, and these
two trees are not contained in one of the generic families of
Theorem~\ref{theorem:global-structure}. Since the overall proof is simpler
when excluding the trees in $\calS$ at this stage, this is the route we
proceed on.

\subsection{Vertices of type \texorpdfstring{$A$}{A} and
  estimates for vertices of type \texorpdfstring{$B$}{B}}\label{section:structure}

We first aim to show that almost all optimal trees fulfil LC2
and LC3. As a first step, we will show that almost all rooted subtrees of optimal
trees contain a $k$-claw for $k\in\{2,3,4\}$, i.e., a rooted subtree with $k$ branches all of which are single vertices,
see Figure~\ref{fig:bad-ends-1}. In a second step, the existence of $k$-claws
will provide us with bounds for $\rho(S)$ for rooted subtrees $S$ of optimal
trees. These bounds will be quite weak, but sufficient for using our general
exchange lemma (Lemma~\ref{lemma:exchange}) to give a useful technical result
on decompositions of optimal trees along a path. This almost immediately yields
LC2. We then characterise all optimal trees containing a $B_{2}^*$ or a $4$-claw
as a rooted subtree (there are only very few), such that from the
end of this subsection, we can work exclusively with $2$- and $3$-claws.

\begin{lemma}\label{lemma:good-ends}
  Let $S$ be a rooted subtree of an optimal tree $T\neq T_{6,2}^*$. Then $S$ is isomorphic to
  $L$, $A_{3}^*$, $B_{2}^*$ or it contains a $k$-claw for some
  $k\ge 2$, i.e., a rooted subtree as in Figure~\ref{fig:bad-ends-1}.
\end{lemma}
\begin{proof}
  If $|S|\le 2$, then $S\in\{L,B_{2}^*\}$ and there is nothing to show.
  We assume that $|S|>2$ and that $S$ does
  not contain a $k$-claw for any $k\ge 2$. 

  Let $v_1$ be a leaf of $S$ of maximum height. Then $v_1$ is of type $A$. If its
  parent $v_2$ (which is of type $B$) has other branches, they
  have to be leaves by the choice of $v_1$ and we found a $k$-claw for $k\ge
  2$, contradiction. Thus $v_2$ has only one branch, $v_1$. 

  The parent of $v_2$ is called $v_3$. It has to be of type $A$. So all
  branches of $v_3$ are of type $B$, thus they cannot be leaves. By
  construction, all branches of $v_3$ are isomorphic to
  $B_2^*$, cf.\ Figure~\ref{fig:shape-of-S-v-3-a}. 

  \begin{figure}[htbp]
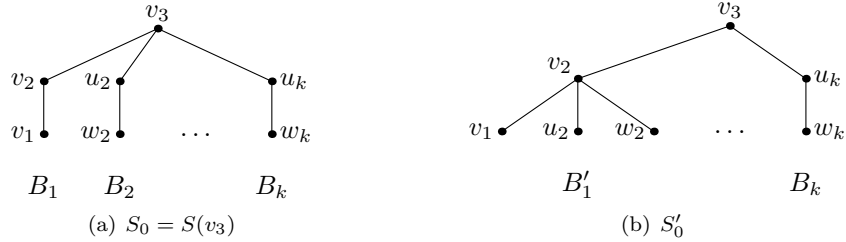

    \centering
    \phantom{.}\hfill\subfigure[$S_0=S(v_3)$]{\label{fig:shape-of-S-v-3-a}\includegraphics{max-card-matching.3}}\hfill
    \subfigure[$S_0'$]{\label{fig:shape-of-S-v-3-b}
\includegraphics{max-card-matching.4}}\hfill\phantom{.}
    \caption{Shape of $S_0=S(v_3)$ and $S_0'$ in the proof of Lemma~\ref{lemma:good-ends}.}
    \label{fig:shape-of-S-v-3}
  \end{figure}
  
  Denote the branches of the rooted tree $S_0:=S(v_3)$ by
  $B_1$, $B_2$, \ldots, $B_k$ and assume that $k\ge 2$. Then we have
  $m_0(B_j)=1$ and $m(B_j)=1$. Thus $m_0(S_0)=1$ and $m(S_0)=k+1$. If we remove
  $B_2$ and add the two vertices as children of $v_2$, cf.\ Figure~\ref{fig:shape-of-S-v-3-b}, the resulting branch
  $B_1'$ has $m_0(B_1')=1$ and $m(B_1')=3$. The modified tree $S_0'$ has 
  $m_0(S_0')=3$ and $m(S_0')=3(1+k-2+1/3)=3k-2\ge k+2>m(S_0)$, contradiction to
  Proposition~\ref{proposition:subtrees-of-optimal-trees-are-alpha-optimal}. Thus
  $v_3$ has only one child.

  If $S=S_0$, then $S=A_{3}^*$ and there is nothing to show. Otherwise,
  the parent of $v_3$ is called $v_4$. Then $T$ has the shape shown in
  Figure~\ref{fig:shape-of-S-v-4} for some $k\ge 0$ and rooted trees $A_0$,
  \ldots, $A_k$ of type $A$.
  
  \begin{figure}[htbp]
    \centering
    \includegraphics{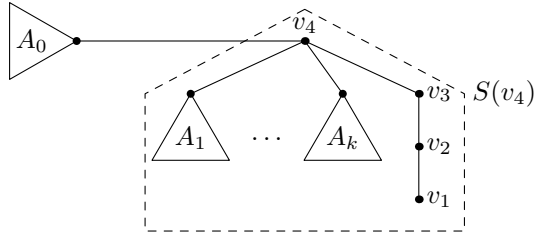}
    
    \caption{Shape of $S(v_4)$ in the proof of Lemma~\ref{lemma:good-ends}.}
    \label{fig:shape-of-S-v-4}
  \end{figure}
  Each of the $A_j$, $j\in\{1,\ldots,k\}$ is either a leaf (with $\rho(A_j)=1$) or an $A_{3}^*$
  with $\rho(A_j)=1/2$. As $\rho(A_0)>0$,
  \liref{lemma:exchange}{item:exchange-symmetric} (with $v_4$ and $v_2$ as pivotal vertices) yields
  $\rho(A_1)+\cdots+\rho(A_k)\le 1$, i.e., either $k\le 1$ or $k=2$ and both $A_1$
  and $A_2$ are isomorphic to $A_{3}^*$. Thus $S(v_4)$ is one of the
  trees in Figure~\ref{fig:good-ends-final}.
  \begin{figure}[htbp]
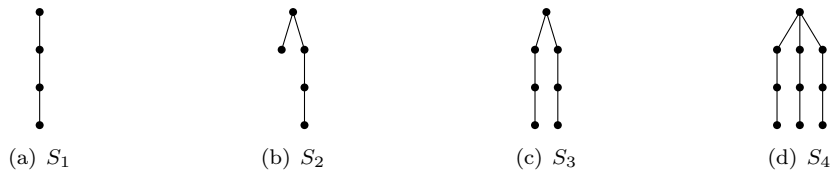

    \centering
    \phantom{A}\hfill
    \subfigure[$S_1$\label{fig:good-ends-final-1}]{\parbox[t]{2cm}{\centering\includegraphics{good-ends-replacement-0a.1}}}\hfill
    \subfigure[$S_2$\label{fig:good-ends-final-2}]{\parbox[t]{2cm}{\centering\includegraphics{good-ends-replacement-1a.1}}}\hfill
    \subfigure[$S_3$\label{fig:good-ends-final-3}]{\parbox[t]{2cm}{\centering\includegraphics{good-ends-replacement-2a.1}}}\hfill
    \subfigure[$S_4$\label{fig:good-ends-final-4}]{\parbox[t]{2cm}{\centering\includegraphics{good-ends-replacement-3a.1}}}\hfill\phantom{A}
    \caption{$S(v_4)$ in the proof of Lemma~\ref{lemma:good-ends}.}
    \label{fig:good-ends-final}
  \end{figure}
  The trees $S_1$, $S_3$, $S_4$ are not $\alpha$-optimal for any $\alpha\in[0,1]$,
  cf.\ \rref{replacement:good-ends-replacements}, contradiction to
  Proposition~\ref{proposition:subtrees-of-optimal-trees-are-alpha-optimal}. The
  tree $S_2$ is not $\alpha$-optimal for
  $\alpha<1$, cf.\ \rref{replacement:good-ends-replacements}, 
  so we must have $\rho(A_0)=1$ and therefore $A_0=L$. Thus we must have
  $T=T_{6,2}^*$, which has been excluded.
\end{proof}
Knowing now that almost every rooted subtree of an optimal tree contains a $k$-claw with $k\ge 2$, we show that no
$k$-claws with $k\ge 5$ occur.

\begin{lemma}\label{lemma:bad-ends}
  For $k\ge 5$, a $k$-claw does not occur as rooted subtree of an optimal tree.
\end{lemma}
\begin{proof}Let $(T,b)$ be a $k$-claw, cf. Figure~\ref{fig:bad-ends-1}, and
  $(T',b)$ be the rooted tree in Figure~\ref{fig:bad-ends-2} of the same order
  and type. 
  \begin{figure}[htbp]
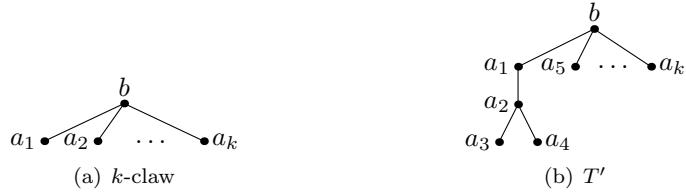

    \centering
    \rule{0pt}{0pt}\hfill
    \subfigure[$k$-claw\label{fig:bad-ends-1}]{\includegraphics{lower-rho-estimate-0.1}}\hfill
    \subfigure[$T'$\label{fig:bad-ends-2}]{\includegraphics{lower-rho-estimate-B0.1}}
    \hfill\rule{0pt}{0pt}
    \caption{$k$-claw and tree $T'$ for the proof of Lemma~\ref{lemma:bad-ends}}
    \label{fig:bad-ends}
  \end{figure}

  We have
  \begin{align*}
    m_0(T)&=1^k=1,&m(T)&=m_1(T)=1(1+\cdots+1)=k,\\
    m_0(T')&=3\cdot 1^{k-4}=3, &m(T')&=m_1(T')=3\left(\frac23+(k-4)\right)=3k-10.
  \end{align*}

  For $k\ge 5$, we have $m(T')\ge k=m(T)$ and $m_0(T')=3>1=m_0(T)$, thus $T$
  cannot be $\alpha$-optimal for any $\alpha>0$, so it is not a subtree of an
  optimal tree by Proposition~\ref{proposition:subtrees-of-optimal-trees-are-alpha-optimal}.
\end{proof}

We are now able to prove lower bounds for $\rho(S)$ for rooted subtrees $S$ of
optimal trees. The key idea is the following: Changing the root of $S$ to
another root can only alter $m_0(S)$, but $m(S)$ remains unchanged. Changing
the root of $S$ cannot increase $\rho(S)$, since this would increase $m(S) +
\alpha m_0(S)$, 
contradiction to the $\alpha$-optimality of $S$. The new roots used for
comparison will be leaves or roots of $k$-claws. 

We start with lower bounds for rooted subtrees of type $A$.
 
\begin{lemma}\label{lemma:weak-rho-estimate-A}
  Let $S$ be a rooted subtree of type $A$ of an optimal tree $T$.
  Then $\rho(S)\ge 1/2$ with equality if and only if $S=A_{3}^*$.
\end{lemma}
\begin{proof}
  Let $T$ consist of the rooted
  subtrees $(S,s)$ and $(T_t,t)$ of types $A$ and $B$, respectively,  and of
  the edge $st$.

  For $|S|\le 3$, we have
  $S\in\{L,A_{3}^*\}$ and there is nothing to
  show, so we assume $|S|>3$.

  By Lemmata~\ref{lemma:good-ends} and \ref{lemma:bad-ends}, $S$ contains
  an $\ell$-claw for some $2\le \ell\le 4$. 
  We switch the root of $S$ to a leaf of the $\ell$-claw, obtaining a new rooted
  tree $(S',a_1)$ shown in Figure~\ref{fig:rho-estimate-A-T-prime}.
  \begin{figure}[htbp]
    \centering
    \includegraphics{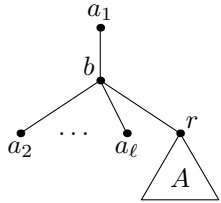}
    \caption{Tree $S'$ in the proof of Lemma~\ref{lemma:weak-rho-estimate-A}}
    \label{fig:rho-estimate-A-T-prime}
  \end{figure}
  The rooted tree $(A,r)$ arises from the rooted connected component $T_r$ of
  $T-br$ by removing the rooted subtree $T_t$ of type $B$. Removing a rooted
  subtree of type $B$ from a rooted tree of type $A$ fulfilling the
  bipartition condition for rooted trees yields a rooted tree of type $A$, so
  $A$ is of type $A$. 

  The $\rho(T_t)$-optimality of $S$ together with 
  $m(S)=m(S')$ implies that
  \begin{equation*}
    \rho(S)\ge
    \rho(S')=\frac1{1+\frac1{\ell-1+\rho(A)}}>\frac1{1+\frac1{1}}=\frac12.
  \end{equation*}
\end{proof}

We are now able to prove a lower bound for $\rho(S)$ for rooted subtrees $S$ of type $B$.

\begin{lemma}\label{lemma:rho-B-bound}
  Let $(S,r)$ be a rooted subtree of type $B$ of an optimal tree $T$ with
  $|T|\ge 3$. If $S$ contains a $k$-claw for some $k\ge 1$, then $\rho(B)\ge\frac1{1+k}$.
\end{lemma}
\begin{proof}
  By Proposition~\ref{proposition:full-bipartite}, $T$ and $S$ fulfil the
  bipartition condition and the bipartition condition for rooted trees, respectively.

  If $S$ has only one branch, say $S=\calB(S')$, then $S'$ is of type $A$ by
  the bipartition condition for rooted trees and $\rho(S')\le 1$. Thus
  $\rho(S)=1/\rho(S')\ge 1$, as required. So we assume that $S$ has more than one branch. 

  If $S$ is the $k$-claw, then $\rho(S) = 1/k>1/(k+1)$.

  Let $T$ consist of the rooted
  subtrees $(S,r)$ and $(T_t,t)$ of types $B$ and $A$, respectively,  and of the edge $rt$.
  
  We change the root of $S$ to the root of the $k$-claw,
  which results in a rooted tree $(S',s)$ with $m(S)=m(S')$ shown in Figure~\ref{fig:rho-B-bound}. 
  \begin{figure}[htbp]
    \centering
    \includegraphics{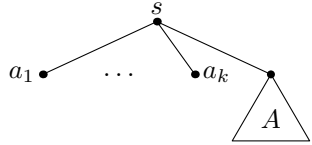}
    \caption{Tree $S'$ in the proof of Lemma~\ref{lemma:rho-B-bound}}
    \label{fig:rho-B-bound}
  \end{figure}
  Here $A$ arises from a rooted subtree of $T$ of type $A$ by removing the
  rooted subtree $(T_t,t)$. Since it was assumed that $S$ has more than one branch, we conclude that
  $A$ is still of type $A$ by Lemma~\ref{lemma:recursive-formulae-rooted-tree}
  and that $\rho(A)\le 1$. As $S$ is $\rho(T_t)$-optimal by
  Proposition~\ref{proposition:subtrees-of-optimal-trees-are-alpha-optimal}, this yields
  \begin{equation*}
    \rho(S)\ge\rho(S')=\frac1{k+\rho(A)}\ge\frac1{k+1}.
  \end{equation*}
\end{proof}

Next, we give a preliminary upper bound for $\rho(S)$ for rooted subtrees of type $A$:

\begin{lemma}\label{lemma:A-worse-upper-bound-rho}
  Let $S$ be a rooted subtree of type $A$ of an optimal tree $T$.
  Then $S=L$ or $\rho(S)\le 5/6$.
\end{lemma}
\begin{proof}
  Assume that $S\neq L$. Let $S=\calA(B_1,\ldots, B_\ell)$ for suitable
  branches $B_1$, \ldots, $B_\ell$ for some $\ell\ge 1$.
  By
  Lemmata~\ref{lemma:rho-B-bound}, \ref{lemma:good-ends} and \ref{lemma:bad-ends}, we have
  $\rho(B_1)\ge 1/5$. which implies
  \begin{equation*}
    \rho(S)=\frac1{1+\rho (B_1)+\cdots+\rho(B_\ell)}\le\frac1{1+\rho (B_1)} \le \frac56.
  \end{equation*}
\end{proof}

We have now collected the necessary (weak) bounds for $\rho(S)$ for rooted
subtrees of optimal trees. These suffice for the following path decomposition
lemma, using the exchange lemma (Lemma~\ref{lemma:exchange}) to derive
bounds for $\rho_i$ along a path (as in Figure~\ref{fig:shape-exchange-lemma})
when the two ends of the path are roots of claws.

\begin{lemma}\label{lemma:path}
  Let $T$ be an optimal tree of the shape as in
  Figure~\ref{fig:shape-exchange-lemma} for some even $k\ge 2$ with
  $S_{0,1}=\cdots=S_{0,r_0}=L$, $S_{k,1}=\cdots=S_{k,r_k}=L$, i.e., $v_0$ and
  $v_k$ are the roots of an $r_0$-claw and an $r_k$-claw, respectively. We assume
  that $r_0\ge r_k\ge
  1$ and $r_0\ge 2$ and set
  \begin{equation*}
    \rho_i=[\type v_i=A]+\sum_{j=1}^{r_i}\rho(S_{i,j}).
  \end{equation*}
  Then $v_i$ is of type $A$ and $(\rho_i,r_i)=(1,0)$ for odd $i$ and $r_0-1\le \rho_i\le
  r_k$ for even $i$ with $0<i<k$.
  In particular, we have $r_0\le r_k+1$.
\end{lemma}
\begin{proof}By the bipartition condition and the fact that $v_0$ and $v_k$
  are the roots of an $r_0$-claw and an $r_k$-claw, respectively, we conclude
  that $v_i$ is of type $A$ for odd $i$ and of type $B$ for even $i$. We define
  the rooted trees $(R,v_{0})$ and $(L,v_{k})$ as two rooted versions of $T$, so that the notations
  $R(v_i)$ and $L(v_i)$ are defined.

  We prove the lemma by induction on $i$, where we first only prove that
  \begin{equation}
    \label{eq:path:induction}
    \text{$v_i$ is of type $A$ and $(\rho_i,r_i)=(1,0)$ for odd $i$ and $r_0-1\le \rho_i\le
  r_0$ for even $i$ with $0<i<k$,}
  \end{equation}
  i.e., we relax the upper bound for $\rho_i$ in the case of even $i$.

  We first consider the case of odd
  $i$, i.e., $v_i$ is of type $A$. If $\rho_i>1$, we conclude that
  $\rho(R(v_{i+1}))< 0+U_1(r_0-1,r_0)$ from
  \liref{lemma:exchange}{item:exchange-A} (with $v_i$ and $S_{0,1}=L$ as pivotal vertices). By Lemma~\ref{lemma:rho-B-bound}, we have $\rho(R(v_{i+1}))\ge
  1/(r_k+1)\ge 1/(r_0+1)$. As $U_1(d-1,d)<1/(d+1)$ for $d\in\{2,3,4\}$, this
  is a contradiction. So $\rho_i=1$ and therefore $r_i=0$.

  Next, we consider the case of even $i$, i.e., $v_i$ is of type $B$. If
  $\rho_i>\rho_0=r_0$, then \liref{lemma:exchange}{item:exchange-B} (now with $v_i$ and $v_0$ as pivotal vertices) yields
  $\rho(R(v_{i+1}))\le U_0(r_0-1,r_0)\le 0.1153$, a contradiction to
  Lemma~\ref{lemma:weak-rho-estimate-A}. Thus we have $\rho_i\le
  \rho_0=r_0$. 

  For the lower bound on $\rho_i$, we assume that
  $\rho_i<r_0-1=\rho(S_{0,2})+\cdots+\rho(S_{0,r_0})$. Then
  \liref{lemma:exchange}{item:exchange-B} implies $1=\rho(S_{0,1})\le
  \rho(R(v_{i+1}))+U_0(r_0-1,r_0)<\rho(R(v_{i+1}))+0.1153$. As
  $\rho(R(v_{i+1}))\le 5/6$ by Lemma~\ref{lemma:A-worse-upper-bound-rho}, this
  is a contradiction. This concludes the proof of \eqref{eq:path:induction}.

  Finally, $r_i\le r_k$ is again a consequence of
  \liref{lemma:exchange}{item:exchange-B}, as $\rho(L(v_{i-1}))\ge 1/2>0.1153$.

  If $k>2$, then $r_0-1\le r_2\le r_k$. If $k=2$, then
  \liref{lemma:exchange}{item:exchange-symmetric} and $\rho(S_{0,r_0})>0$
  imply that $r_0-1=\rho(S_{0,1})+\cdots+\rho(S_{0,r_0-1})\le
  \rho(S_{0,1})+\cdots+\rho(S_{0,r_k})=r_k$, as required.
\end{proof}

Combining the description of rooted subtrees without any $k$-claw with $k\ge 2$
with the path decomposition lemma (Lemma~\ref{lemma:path}) shows that $B_2^*$ is forbidden in almost all
optimal trees.

\begin{lemma}\label{lemma:strange-special-cases}
  Suppose that $B_{2}^*$ is a rooted subtree of an optimal tree $T$. 
  Then $T\in\{T_{3}^*,T_{6,2}^*\}$.
\end{lemma}
\begin{proof}
  Assume first that $T$ does not contain an $\ell$-claw for any $\ell\ge 2$. As
  $T$ contains a $B_2^*$ as a rooted subtree, we have $|T|\ge 3$.  Consider a leaf
  $s$ of $T$. Then by
  Proposition~\ref{proposition:full-bipartite}, the rooted subtree $T-s$ of $T$
  is of
  type $B$ and therefore equals $B_2^*$ by Lemma~\ref{lemma:good-ends}. We
  conclude that $T=T_3^*$.

  So we may now assume that $T$ contains an  $\ell$-claw with $\ell\ge 2$. Thus $T$ can be decomposed
  as in Lemma~\ref{lemma:path} with $r_0\ge \ell\ge 2$ and $r_k=1$ for some $k\ge
  2$. By Lemma~\ref{lemma:path}, we have $r_0\le 1+1=2$, so $r_0=2$. If $k=2$,
  then we have $T=T_{6,2}^*$.

  So we may assume that $k>2$. By
  Lemma~\ref{lemma:path} again, we have $\rho_{k-2}=1$, which by
  Lemmata~\ref{lemma:weak-rho-estimate-A} and
  \ref{lemma:A-worse-upper-bound-rho} implies that $r_{k-2}=1$ with
  $S_{k-2,1}=L$ or $r_{k-2}=2$ with $S_{k-2,1}=S_{k-2,2}=A_3^*$. Thus
  $R(v_{k-2})$ is a rooted subtree of $T$ of type $B$ and order at least $5$
  containing no $\ell$-claw for any $\ell\ge 2$, contradiction to Lemma~\ref{lemma:good-ends}.
\end{proof}

\begin{remark}
Having excluded $B_2^*$, we can also exclude the presence of $A_3^*$ as a rooted subtree in the following, which will be important in many arguments.
\end{remark}

As a direct consequence of the path decomposition lemma (together with the
information that it can always be applied as $B_2^*$ has now been excluded), we have shown LC2.
\begin{proposition}\label{proposition:degree-bound-A}
  Let $T\neq T_2^*$ be an optimal tree. Then $T$ fulfils LC1--LC2.
\end{proposition}
\begin{proof}
  Let $v$ be a vertex of type $A$ in $T$ of degree at least $2$. Then $T$ may
  be represented as in Lemma~\ref{lemma:path} with $v=v_i$ for some odd $i$: choose a longest path that contains $v$. 
  The ends of this path are leaves, their unique neighbours are the pivotal vertices $v_0$ and $v_k$. All but one of the neighbours of $v_0$ have to be leaves by the choice of the path, and there has to be more than one such neighbour in view of Lemma~\ref{lemma:strange-special-cases}. The same applies to $v_k$. Hence $r_i=0$ by Lemma~\ref{lemma:path}, i.e., $\deg v=2$.
\end{proof}

We conclude this subsection by excluding
$4$-claws in almost all cases. To do so, we will use a direct substitution for those
cases which are allowed by the path decomposition lemma.

\begin{lemma}\label{lemma:4-claw-is-bad}
  Let $T$ be an optimal tree containing a $4$-claw as a rooted subtree. Then $T=T_{6,1}^*$.
\end{lemma}
\begin{proof}
  We denote the root of the $4$-claw by $w$. The neighbour of $w$ which is not
  contained in the $4$-claw is denoted by $v$. By
  Proposition~\ref{proposition:degree-bound-A}, $T$ fulfils LC1--LC2. 
  As $w$ is of type $B$, $v$ is of type $A$ and $\deg v\le 2$. If $v$ is a
  leaf, then $T=T_{6,1}^*$ and we are done. So we assume the contrary and
  denote the neighbour of $v$ different from $w$ by $u$. Then
  $T$ is of the shape shown in Figure~\ref{fig:4-claw-is-bad-1}, 
  where $A_1$, \ldots, $A_\ell$ denote some trees of type $A$ with
  $\rho(A_1)\ge \rho(A_2)\ge\cdots\ge\rho(A_\ell)$. If 
  $A_1=\cdots=A_\ell=L$, then $\ell\in\{3,4\}$ by Lemmata~\ref{lemma:bad-ends}
  and \ref{lemma:path}.
  Both cases do not lead to an optimal tree, cf.\
  \rref{replacement-full:4-claw-is-bad}.

  So we may assume that $\rho(A_\ell)<1$, whence \liref{lemma:exchange}{item:exchange-symmetric}, with $u$ and $w$ as pivotal vertices, yields $\rho(A_1)+\cdots+\rho(A_{\ell-1})\ge 3$. Thus we have $\ell=4$ and
  $A_1=A_2=A_3=L$ or $\ell\ge 5$.
  \begin{figure}[htbp]
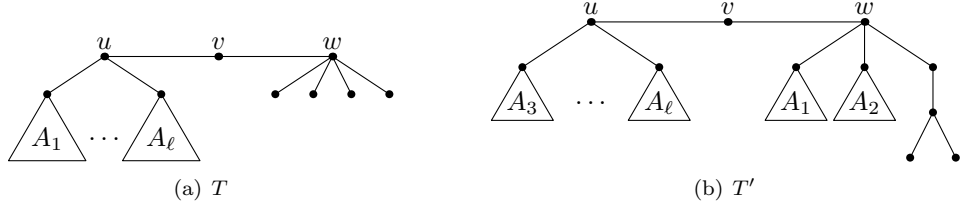

    \centering
    \rule{0pt}{0pt}\hfill\subfigure[$T$\label{fig:4-claw-is-bad-1}]{\includegraphics{4claw-1.1}}
    \hfill\subfigure[$T'$\label{fig:4-claw-is-bad-2}]{\includegraphics{4claw-3.1}}
    \hfill\rule{0pt}{0pt}
    \caption{Trees considered in Lemma~\ref{lemma:4-claw-is-bad}.}
  \end{figure}

  We consider the tree $T'$ shown in Figure~\ref{fig:4-claw-is-bad-2} of the same order as $T$.
  Using the abbreviations $a=\rho(A_3)+\cdots+\rho(A_\ell)$, 
  $b=\rho(A_1)+\rho(A_2)$ and $M=m(A_1)\ldots m(A_{\ell})$ as well as the optimality of $T$,
  Lemma~\ref{lemma:m-m-0-via-continuants} yields
  \begin{align*}
    M(5a+5b+4)&=M\CFP(a+b,1,4)=m(T)\\&\ge m(T')=3M\CFP\left(a,1,b+\frac23\right)=M(3ab+5a+3b+2),
  \end{align*}
  which implies
  \begin{equation*}
    2\ge b(3a-2).
  \end{equation*}
  As $b=\rho(A_1)+\rho(A_2) \geq 1$ and $a \geq 3/2$ (for $\ell \ge 5$ as well as for $\ell = 4$ and $\rho(A_3) = 1$)
  by Lemma~\ref{lemma:weak-rho-estimate-A}, this is a contradiction.
\end{proof}

\subsection{Lower degree bounds for vertices of type \texorpdfstring{$B$}{B}}
\label{section:light-trees}
We now want to show that almost all optimal trees fulfil LC1--LC4.
In order to facilitate the discussion, we introduce the notion of a ``light''
vertex. 

\begin{definition}
  Let $T$ be a tree fulfilling the bipartition condition and $v$ a vertex of
  type $B$ in $T$. Then $v$ is said to be a \emph{light vertex} if it has
  degree $\le 3$ and is adjacent to at most one leaf.
\end{definition}

The sum of the $\rho$-values of the rooted connected components of $T-v$ is quite small for
a light vertex $v$. The exchange lemma then forbids vertices whose
rooted connected components have a high sum of $\rho$-values.

The path decomposition lemma can be used to
derive a description of light vertices:

\begin{lemma}\label{lemma:light-vertex-properties}
  If $T\notin\{ T_{2}^*, T_{3}^*, T_{6,2}^*\}$ is an optimal tree, then $T$
  fulfils LC1--LC3. If $v$ is a light vertex of $T$, then $\deg v=3$ and
  $v$ is adjacent to exactly one leaf.
\end{lemma}
\begin{proof}
  If $T$ has no light vertex, then $T$ fulfils LC1--LC4. So we assume that $v$
  is a light vertex.
  Denote the rooted connected components of $T-v$ by $T_0$, \ldots,
  $T_{k-1}$ with $|T_0|\ge |T_1|\ge \cdots \ge |T_{k-1}|$.  As $v$ is
  a light vertex, we have $k\le 3$. If $|T_0|=1$, i.e., $T_0=\cdots=T_{k-1}=L$,
  we have $k\le 2$ and
  $T\in\{T_{2}^*, T_{3}^*\}$, which have been excluded.  If $|T_1|=1$, then $v$ is the root of a $(k-1)$-claw. As $v$
  is light, we have $k=2$, contradiction to
  Lemma~\ref{lemma:strange-special-cases}.

  So both $T_0$ and $T_1$ contain an $\ell_0$-claw and an $\ell_1$-claw, respectively, for
  some $\ell_0\ge 2$, $\ell_1\ge 2$ by Lemma~\ref{lemma:good-ends} and
  Lemma~\ref{lemma:strange-special-cases}. By Lemma~\ref{lemma:path}, we obtain 
  $\rho(T_2)+\cdots+\rho(T_{k-1})\ge 1$. As $k\le 3$ by assumption, we have 
  $\rho(T_2)+\cdots+\rho(T_{k-1})=\rho(T_2)\le 1$, thus $T_2$ is a leaf, as required.
\end{proof}

We now describe vertices of type $B$ when a light vertex is present.

\begin{lemma}\label{lemma:light-tree-vertex-classes}
  Let $T\notin\{ T_{2}^*, T_{3}^*, T_{6,2}^*\}$ be an optimal tree, $v$ be a light vertex and $w$ be a vertex of type
  $B$ of $T$. Then either $\deg w=4$ and $w$ is not adjacent to any leaf or $\deg w=3$ and
  $w$ is adjacent to one or two leaves.
\end{lemma}
\begin{proof}
  If $w$ is light, then there is nothing to show by
  Lemma~\ref{lemma:light-vertex-properties}. Otherwise, either $\deg w\ge 4$ or
  $\deg w=3$ and $w$ is adjacent to two leaves, as required.
  So we now assume that $\deg w\ge 4$.

  Denote the rooted connected components of $T-w$ by $T_0$, \ldots,
  $T_{\ell-1}$ with $v\in T_0$ and $\rho(T_1)\ge \cdots\ge \rho(T_{\ell-1})$. The rooted connected component of $T-v$ which does
  not contain $w$ and is not a leaf is denoted by $S$. By
  Lemma~\ref{lemma:weak-rho-estimate-A}, we have
  $\rho(T_{\ell-2})+\rho(T_{\ell-1})>\rho(L)$ (since $A_3^*$ has been excluded). Now we make use of a combination of Lemma~\ref{lemma:path} and \liref{lemma:exchange}{item:exchange-B}---the following argument will be used several times, so we only explain it in detail here: the path between $v$ and $w$ can be extended to a longest path ending in an $r_0$- and an $r_k$-claw, with $2 \le r_0$, $r_k \le 3$ by Lemmata~\ref{lemma:strange-special-cases}, \ref{lemma:bad-ends} and \ref{lemma:4-claw-is-bad}. Application of Lemma~\ref{lemma:path} now shows that the vertices on the path between $v$ and $w$ satisfy the necessary conditions to make \liref{lemma:exchange}{item:exchange-B} applicable (with $v$ and $w$ as pivotal vertices), which yields
  \begin{equation}\label{eq:light-tree-vertex-classes-estimate}
    \rho(T_1)+\cdots+\rho(T_{\ell-3})\le \rho(S)+0.1153.
  \end{equation}
  From Lemmata~\ref{lemma:weak-rho-estimate-A} and
  \ref{lemma:A-worse-upper-bound-rho}, we obtain
  \begin{equation*}
    \frac{\ell-3}{2}< \rho(T_1)+\cdots+\rho(T_{\ell-3})\le \rho(S)+0.1153\le \frac56+0.1153<1,
  \end{equation*}
  which yields $\ell<5$ and  $\deg w=\ell= 4$. Furthermore,
  \eqref{eq:light-tree-vertex-classes-estimate} together with
  Lemma~\ref{lemma:A-worse-upper-bound-rho} yields $\rho(T_1)\le
  \frac56+0.1153<1$, i.e., $1>\rho(T_1)\ge\rho(T_2)\ge\rho(T_3)$, so $w$ is not
  adjacent to any leaf.
\end{proof}

Light vertices correspond to low values of $\rho(A)$ for rooted subtrees of
type $A$. This correspondence is described in the following two lemmata.

\begin{lemma}\label{lemma:rho-A-2-3}
  Let $T\notin\{ T_{2}^*, T_{3}^*, T_{6,2}^*\}$ be an optimal tree and $A$ be a rooted subtree of $T$
  of type $A$ that contains no light vertex of $T$. Then
  \begin{equation*}
    \rho(A)\ge \frac23
  \end{equation*}
  with equality for $A\in\{F, A_{14}^*,A_{24}^*\}$.
\end{lemma}
\begin{proof}
  We prove the result by induction on the order of $A$. If $A$ has order
  $1$, then $\rho(A)=1$.

  Otherwise, $A=\calA\calB(A_1,\ldots,A_k)$ for suitable rooted trees $A_1$, \ldots,
  $A_k$ of type $A$ by Proposition~\ref{proposition:degree-bound-A}. As $A$ does not contain any light vertex, we have
  $A=F$ with $\rho(A)=2/3$ or $k\ge 3$. 

  We now turn to the case $k\ge 3$. We have
  $\rho(A_j)\ge 2/3$ by the induction hypothesis and therefore
  \begin{equation*}
    \rho(A)=\frac{1}{1+\frac1{\rho(A_1)+\cdots+\rho(A_k)}}\ge \frac{1}{1+\frac1{3\cdot\frac23}}=\frac23.
  \end{equation*}
  Equality holds for $k=3$ and $\rho(A_1)=\rho(A_2)=\rho(A_3)=2/3$. By the induction
  hypothesis, we conclude that $A_1$, $A_2$, $A_3\in\{F,
  A_{14}^*,A_{24}^*\}$. We have
  $A_{14}^*=\calA\calB(F,F,F)$ and
  $A_{24}^*=\calA\calB(F,F,A_{14}^*)$. Next, the two trees
  $\calA\calB(F,F,A_{24}^*)$ and 
  $\calA\calB(F,A_{14}^*,A_{14}^*)$ are not $\alpha$-optimal for
  $\alpha>50/2473$, cf.\ \rref{replacement:rho-A-2-3}, whence they do
  not occur as rooted subtrees of optimal trees by
  Proposition~\ref{proposition:subtrees-of-optimal-trees-are-alpha-optimal} and
  Lemma~\ref{lemma:rho-B-bound}. A further six cases have to be checked, but none
  of these is an $\alpha$-optimal tree for any $\alpha\ge0$, cf.\ again \rref{replacement:rho-A-2-3}.
\end{proof}

\begin{lemma}\label{lemma:upper-bound-rho-A-light}
  Let $T\notin\{ T_{2}^*, T_{3}^*, T_{6,2}^*\}$ be an optimal tree containing a light vertex and $A$ be
  a rooted subtree of $T$ of type $A$. Then  $A$ is a leaf or
  \begin{equation*}
    \rho(A)\le \frac23,
  \end{equation*}
  where equality holds if and only if $A$ does not contain a light vertex.
\end{lemma}
\begin{proof}
  We prove the assertion by induction on the order of $A$. For $|A|>1$, we have
  $A=\calA\calB(A_1,\ldots,A_k)$ for suitable rooted trees $A_1$, \ldots,
  $A_k$. By Lemma~\ref{lemma:light-tree-vertex-classes}, we have $k\in\{2,3\}$.

  If $k=2$, then $\rho(A_1)+\rho(A_2)\le 1+1\le 2$ where equality holds if and
  only if both $A_1$ and $A_2$ are leaves, i.e., $A$ does not contain a light vertex.

  If $k=3$, then none of
  $A_1$, $A_2$, $A_3$ is a leaf by Lemma~\ref{lemma:light-tree-vertex-classes},
  so the induction hypothesis yields $\rho(A_1)+\rho(A_2)+\rho(A_3)\le 2$ with
  equality if and only if none of $A_1$, $A_2$, $A_3$ contains a light
  vertex. In both cases, we get
  \begin{equation*}
     \rho(A)=\frac{1}{1+\frac1{\rho(A_1)+\cdots+\rho(A_k)}}\le \frac{1}{1+\frac1{2}}=\frac23.
  \end{equation*}
\end{proof}

We are now ready to prove local condition LC4.

\begin{proposition}\label{proposition:optimal-tree-with-light-vertex-new}
  Let $T$ be an optimal tree containing a light vertex. Then
  $T\in\{T_{2}^*, T_{3}^*, T_{6,2}^*,\allowbreak T_{10}^*\}$. In other words, all optimal trees
  except $T_{2}^*, T_{3}^*, T_{6,2}^*, T_{10}^*$ fulfil LC1--LC4.
\end{proposition}
\begin{proof}
  We assume that $T\notin\{T_2^*,T_3^*, T_{6,2}^*\}$.
  As the tree in Figure~\ref{fig:optimal-tree-with-light-vertex-not-optimal-new}
  \begin{figure}[htbp]
    \centering
    \includegraphics{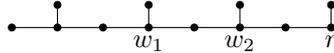}
    \caption{Subtree of $T$ in Proposition~\ref{proposition:optimal-tree-with-light-vertex-new}.}
    \label{fig:optimal-tree-with-light-vertex-not-optimal-new}
  \end{figure}
  with root $r$ is not $\alpha$-optimal for $\alpha<7/6$, cf.\
  \rref{replacement:light}, and any vertex of degree $3$ is adjacent to at least
  one leaf by Lemma~\ref{lemma:light-vertex-properties},
  \begin{equation}
    \label{eq:subtree-max-degree-3-light}
    \parbox{0.8\linewidth}{rooted subtrees $S$ of type $A$ of $T$ containing a light vertex
      and no vertex of degree $4$ are isomorphic to $A_7^*=\calA\calB(L,F)$ or $A_{10}^*=\calA\calB(L,A_7^*)$.}
  \end{equation}
  By \rref{replacement:light-2-3-replacements},
  Lemma~\ref{lemma:upper-bound-rho-A-light},
  Proposition~\ref{proposition:subtrees-of-optimal-trees-are-alpha-optimal} and
  the fact that $T$ contains a light vertex, we see that
  \begin{equation}
    \label{eq:B-L-A-14-or-24}
    \text{neither $\calB(L,A_{14}^*)$ nor $\calB(L,A_{24}^*)$ occurs as a
      rooted subtree of $T$.}
  \end{equation}
  
  Let now $S$ be a rooted subtree of type $A$ of $T$ containing a vertex of
  degree $4$ and a light vertex. We choose $S$ in such a way that its order is
  minimal among all rooted subtrees with these properties. We write
  $S=\calA\calB(S_1,\ldots, S_d)$ for some $d\in\{2,3\}$ (by
  Lemma~\ref{lemma:light-tree-vertex-classes}).

  We first consider the case $d=2$. By
  Lemma~\ref{lemma:light-tree-vertex-classes}, we have $S_1=L$. As $S$ contains
  a vertex of degree $4$, so does $S_2$. By minimality of $S$, $S_2$ does not
  contain a light vertex. By Lemmata~\ref{lemma:rho-A-2-3} and
  \ref{lemma:upper-bound-rho-A-light}, we have $\rho(S_2)=2/3$ and therefore
  $S_2\in \{A_{14}^*, A_{24}^*\}$, contradiction to \eqref{eq:B-L-A-14-or-24}.

  Thus we are left with the case $d=3$. By the minimality of $S$,
  each of the $S_j$ either contains a light vertex and does not contain a vertex of degree $4$, whence
  $S_j\in\{A_7^*, A_{10}^*\}$ by \eqref{eq:subtree-max-degree-3-light}, or does
  not contain a light vertex, whence $S_j\in\{F, A_{14}^*, A_{24}^*\}$ by Lemmata~\ref{lemma:rho-A-2-3} and
  \ref{lemma:upper-bound-rho-A-light}. 

  We first consider the case that one of $S_1$, $S_2$, $S_3$, say $S_1$, is an
  $A_{10}^*$. Then $T$ can be decomposed as in
  Figure~\ref{fig:light-degree-4-decomposition} for some rooted tree $S_0$ of
  \begin{figure}[htbp]
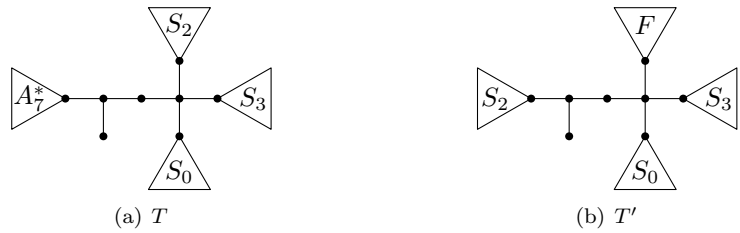

    \centering
    \rule{0cm}{0cm}\hfill\subfigure[\label{fig:light-degree-4-decomposition} $T$]{\includegraphics{light-exchange.5}}\hfill
    \subfigure[\label{fig:light-degree-4-T-prime} $T'$]{\includegraphics{light-exchange.7}}\hfill\rule{0cm}{0cm}
    \caption{Decomposition of $T$ and $T'$ in Proposition~\ref{proposition:optimal-tree-with-light-vertex-new}.}
    
  \end{figure}
  type $A$. As $\rho(S_0)+\rho(S_3)>\rho(L)=1$,
  \liref{lemma:exchange}{item:exchange-symmetric} yields $\rho(S_2)\le
  \rho(A_7^*)<2/3$. Analogously, we get $\rho(S_3)<2/3$. Thus we have
  $\{S_2,S_3\}\subseteq\{A_7^*,A_{10}^*\}$. By
  \rref{replacement:light-2-3-replacements-A-10}, $\calB(S_1,S_2,S_3)$ is not
  $\rho(S_0)$-optimal, contradiction to
  Proposition~\ref{proposition:subtrees-of-optimal-trees-are-alpha-optimal}.

  So we are left with the case that
  $\{S_1,S_2,S_3\}\subseteq\{F,A_7^*,A_{14}^*, A_{24}^*\}$. As $S$ contains a
  light vertex, we may assume that $S_1=A_7^*$. If $S_2\in\{A_{14}^*,
  A_{24}^*\}$, we note that switching $S_2$ and the $F$ of
  $S_1=A_7^*=\calA\calB(L,F)$ yields the tree $T'$ shown in
  Figure~\ref{fig:light-degree-4-T-prime}  with $m(T)=m(T')$ by
  Lemma~\ref{lemma:m-m-0-via-continuants}. But $T'$ is not optimal by
  \eqref{eq:B-L-A-14-or-24}. We conclude that
  $\{S_2,S_3\}\subseteq\{F,A_7^*\}$. By
  \rref{replacement:light-2-3-replacements-A-7},
  Lemma~\ref{lemma:upper-bound-rho-A-light} and
  Proposition~\ref{proposition:subtrees-of-optimal-trees-are-alpha-optimal},
  the only remaining case is $(S_0,S_1,S_2,S_3)=(L,A_7^*,F,F)$. This case is
  ruled out by \rref{replacement-full:light-no-degree-4-new}.

  So there is no rooted subtree $S$ of type $A$ containing both a light vertex
  and a vertex of degree $4$. By Lemmata~\ref{lemma:light-tree-vertex-classes}
  and \ref{lemma:good-ends}, $T$ contains a $2$-claw. Removing this $2$-claw
  from $T$ yields a rooted tree $S$ of type $A$ of $T$ containing a light
  vertex. We conclude that $S$ does not contain a vertex of
  degree $4$, thus $S\in\{A_7^*, A_{10}^*\}$ by
  \eqref{eq:subtree-max-degree-3-light}. The case $S=A_7^*$ yields
  $T=T_{10}^*$, the case $S=A_{10}^*$ is impossible in view of
  \rref{replacement:light} (cf.\ Figure~\ref{fig:optimal-tree-with-light-vertex-not-optimal-new}).
\end{proof}

\subsection{Upper degree bounds for vertices of type \texorpdfstring{$B$}{B}}
\label{section:heavy-trees}
We now conclude the proof which shows that almost all optimal trees fulfil LC1--LC6.

If an optimal tree contains a $2$-claw, the upper degree bound LC5 for degrees
of type $B$ is a consequence of the exchange lemma together with the improved lower
bound for $\rho(S)$ for rooted subtrees $S$ of type $A$ obtained by the
exclusion of light vertices.

\begin{lemma}\label{lemma:2-claw-implies-LC5}
  Let $T\notin\{ T_{2}^*, T_{3}^*, T_{6,2}^*, T_{10}^*\}$ be an optimal
  tree. If $T$ contains a $2$-claw as a rooted subtree, then $T$ fulfils LC1--LC5.
\end{lemma}
\begin{proof}
  Assume that there is a vertex $w$ of type $B$ of degree $k\ge 5$. We denote the rooted connected components of $T-w$ by
  $T_0$, $T_1$, \ldots, $T_{k-1}$, where a $2$-claw is contained in $T_0$. Now we combine Lemma~\ref{lemma:path} and
  \liref{lemma:exchange}{item:exchange-B} as before: since $\rho(T_{k-1})+\rho(T_{k-2})>
  1=\rho(L)$ (where $L$ is one of the leaves of the $2$-claw contained in
  $T_0$), we have
  \begin{equation*}
    \rho(T_1)+\rho(T_2)+\cdots+\rho(T_{k-3})< \rho(L)+0.1153.
  \end{equation*}
  As
  $T$ contains no light vertex by
  Proposition~\ref{proposition:optimal-tree-with-light-vertex-new},
  Lemma~\ref{lemma:rho-A-2-3} yields
  \begin{equation*}
     \frac23(k-3)<1.1153.
  \end{equation*}
  We conclude that $k<5$, as required.
\end{proof}

We are now left with optimal trees containing $3$-claws. The arguments are
somewhat similar as in the case of light vertices, except that we now have to deal with
``heavy'' vertices.

\begin{lemma}\label{lemma:3-claws-and-degree-4}
  Let $T$ be an optimal tree containing a $3$-claw and containing a vertex $v$ of
  degree $4$ which is adjacent to at most $2$ leaves. 

  Then $v$ is adjacent to exactly $2$ leaves. Furthermore, each vertex of degree $k>4$ is
  adjacent to at most $2$ leaves.
\end{lemma}
\begin{proof}
  Denote the rooted connected components of $T-v$ by $T_1$, \ldots,
  $T_4$, where we assume that $T_1$ contains a $3$-claw and $T_4$ is not
  a leaf. By Lemma~\ref{lemma:path}, we have $\rho(T_2)+\rho(T_3)\ge 2$, i.e.,
  $T_2$ and $T_3$ are indeed leaves.

  We have therefore shown that every vertex of degree $4$ is adjacent to $2$ or
  $3$ leaves.

  Assume that $w$ is a vertex of degree $k\ge 5$ in $T$. Still denoting the
  rooted connected components of $T-v$ by $T_1$, $T_2=L$, $T_3=L$, $T_4$, we may now assume
  that $w$ is contained in $T_4$ ($T_4$ also contains a $3$-claw, since otherwise Lemma~\ref{lemma:2-claw-implies-LC5}
  would apply, so we can interchange the roles of
  $T_1$ and $T_4$ if necessary). The rooted connected components of $T-w$ are
  denoted by $S_0$, $S_1$, \ldots, $S_{k-1}$ with the assumption that $v$ is
  contained in $S_0$ and $\rho(S_1)\ge \rho(S_2)\ge \cdots\ge\rho(S_{k-1})$.  
  Combining Lemma~\ref{lemma:path} and \liref{lemma:exchange}{item:exchange-B} again (with $v$ and $w$ as pivotal vertices and $\rho(S_{k-1})>0.1153$), we get
  \begin{equation*}
    \rho(S_1)+\rho(S_2)+\rho(S_3)\le \rho(S_1)+\rho(S_2)+\cdots+\rho(S_{k-2})\le
    \rho(T_1)+\rho(T_2)+\rho(T_3)<3,
  \end{equation*}
  which implies that $\rho(S_3)<1$ and therefore $w$ is adjacent to at most $2$ leaves.
\end{proof}

The lower bound in the following lemma is the same as in Lemma~\ref{lemma:rho-A-2-3}; but instead
of considering a rooted subtree of an optimal tree, we only assume LC1--LC4.

\begin{lemma}\label{lemma:rho-A-2-3-weak-local-optimality-condition}
  Let $T$ be a tree fulfilling LC1--LC4 and $A$ be a
  rooted subtree of type $A$. Then
  \begin{equation*}
    \rho(A)\ge 2/3.
  \end{equation*}
\end{lemma}
\begin{proof}
  Analogous to Lemma~\ref{lemma:rho-A-2-3}.
\end{proof}

We will exclude the occurrence of three or more $3$-claws in an optimal tree
by substituting two $3$-claws by  $2$-claws and use the additional vertices
in order to create a light vertex. To make this work, we have to analyse the
effects of these substitutions. As the intermediate steps do not necessarily lead to
optimal trees, we can only use the above bound.

\begin{lemma}\label{label:3-claw-replacements}
  Let $T$ be a tree fulfilling LC1--LC4.
  \begin{enumerate}
  \item If one $3$-claw in $T$ is replaced by a $2$-claw, we have $m(T')/m(T)\ge
    8/11$ and $|T'|=|T|-1$ for the resulting tree $T'$. Furthermore $T'$
    fulfils LC1--LC4.
  \item If one $3$-claw in $T$ is replaced by a $\calB(L,F)$, we have $m(T')/m(T)\ge
    21/11$ and $|T'|=|T|+2$ for the resulting tree $T'$.
  \end{enumerate}
\end{lemma}
\begin{proof}
  Let $T$ consist of a $3$-claw, a rooted tree $S$ of type $A$ and the edge
  connecting the root of the $3$-claw and the root of $S$. We have $\rho(S)\ge
  2/3$ by Lemma~\ref{lemma:rho-A-2-3-weak-local-optimality-condition}. 
  \begin{enumerate}
  \item By Lemma~\ref{lemma:no-adjacent-vertices-of-type-A}, we have 
    \begin{equation*}
      \frac{m(T')}{m(T)}=\frac{1m_0(S)+2m(S)}{1m_0(S)+3m(S)}=\frac{\rho(S)+2}{\rho(S)+3}=
      1-\frac{1}{\rho(S)+3}\ge 1-\frac1{\frac23+3}=\frac8{11}.
    \end{equation*}
  \item Analogous.
  \end{enumerate}
\end{proof}

We now deal with optimal trees containing a $3$-claw. The restrictions are now
so strict that we can discuss all cases.

\begin{proposition}\label{proposition:heavy-trees}
  Let $T$ be an optimal tree which contains a $3$-claw. Then $T\in\{T_5^*,\allowbreak
  T_8^{*},\allowbreak T_{9}^*,\allowbreak T_{12}^*,\allowbreak T_{13}^*,\allowbreak T_{16}^*,\allowbreak T_{20}^*\}$.
\end{proposition}
\begin{proof}
  Assume that $T$ contains $3$ rooted subtrees isomorphic to a
  $3$-claw. Replacing two of them by a $2$-claw and the third by a 
  $\calB(L,F)$ yields a tree $T'$ with $|T|=|T'|$ and 
  \begin{equation*}
    \frac{m(T')}{m(T)}\ge \frac8{11}\cdot\frac8{11}\cdot\frac{21}{11}>1
  \end{equation*}
  by Lemma~\ref{label:3-claw-replacements}, contradiction.

  We conclude that $T$ contains at most $2$ rooted subtrees isomorphic to a
  $3$-claw. 

  We now assume that $T$ has a vertex $v$ of degree $k\ge 5$. By
  Lemma~\ref{lemma:2-claw-implies-LC5}, $T$ does not contain a $2$-claw and thus no vertices of degree $3$ by
  Proposition~\ref{proposition:optimal-tree-with-light-vertex-new}. Thus every rooted subtree of $T-v$ either
  contains a $3$-claw or is a leaf by Lemmata~\ref{lemma:good-ends}, \ref{lemma:bad-ends},
  \ref{lemma:strange-special-cases} and \ref{lemma:4-claw-is-bad}. As there are at most $2$ subtrees
  isomorphic to a $3$-claw in $T$, there are at least $k-2$ leaves. 
  By Lemmata~\ref{lemma:bad-ends} and \ref{lemma:4-claw-is-bad},
  we conclude that $v$ is adjacent to exactly $k-2$ leaves. 
  Denote the rooted connected components of $T-v$ by $T_1$, $T_2$, $L$, \ldots,
  $L$. By Lemma~\ref{lemma:path}, we have 
  \begin{equation*}
    k-2=\rho(L)+\cdots+\rho(L)\le 3,
  \end{equation*}
  i.e., $k=5$. By
  Lemma~\ref{lemma:3-claws-and-degree-4}, we conclude that all vertices of
  degree $4$ in $T$ are adjacent to $3$ leaves, i.e., they are the root of a
  $3$-claw.
 
  Thus $T$ is of the shape given in Figure~\ref{fig:heavy-shape-4}
  \begin{figure}[htbp]
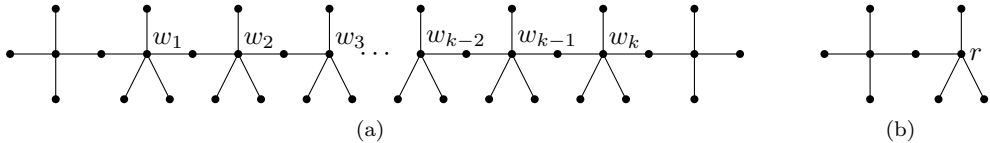

    \centering
    \rule{0cm}{0cm}\hfill
    \subfigure[\label{fig:heavy-shape-4}]{\includegraphics{heavy-shape.4}}\hfill
    \subfigure[\label{fig:heavy-shape-5}]{\includegraphics{heavy-shape.5}}\hfill\rule{0cm}{0cm}
    \caption{Shape of $T$ in Proposition~\ref{proposition:heavy-trees}.}
  \end{figure}
  for some $k\ge 1$ (as there is no vertex of degree $\ge 5$ for $k=0$). As the
  tree in Figure~\ref{fig:heavy-shape-5}
  with root $r$ is not $\alpha$-optimal for $\alpha>1/2$, cf.\ \rref{replacement:heavy}, we conclude from
  Proposition~\ref{proposition:subtrees-of-optimal-trees-are-alpha-optimal},
  Lemma~\ref{lemma:weak-rho-estimate-A} and Lemma~\ref{lemma:strange-special-cases} that this tree does not occur as a subtree
  of an optimal tree. Thus $T$ has no vertex of degree $\ge 5$, i.e., $T$
  fulfils LC1--LC5.

  So by Lemma~\ref{lemma:3-claws-and-degree-4}, every vertex of $T$ of type $B$
  is adjacent to $2$ or $3$ leaves and has degree $3$ or $4$. We conclude that
  $T=T_{5}^*$ or it is a caterpillar tree of the shape given in Figure~\ref{fig:heavy-shape-2}
  \begin{figure}[htbp]
    \centering
    \includegraphics{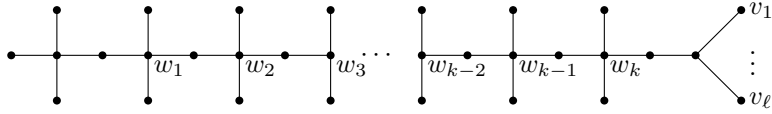}
    \caption{Shape of $T$ in Proposition~\ref{proposition:heavy-trees}.}
    \label{fig:heavy-shape-2}
  \end{figure}
  for some $k\ge 0$ and some $\ell\in\{2,3\}$. We note that 
  the tree in Figure~\ref{fig:heavy-shape-3}
  \begin{figure}[htbp]
    \centering
    \includegraphics{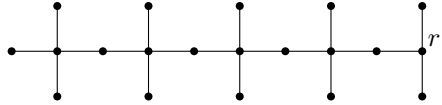}
    \caption{Subtree of $T$ in Proposition~\ref{proposition:heavy-trees}.}
    \label{fig:heavy-shape-3}
  \end{figure}
  with root $r$ is not $\alpha$-optimal for $\alpha> 2/17$, cf.\ \rref{replacement:heavy}, thus this tree
  does not occur as a rooted subtree of an optimal tree by
  Proposition~\ref{proposition:subtrees-of-optimal-trees-are-alpha-optimal} and
  Lemma~\ref{lemma:weak-rho-estimate-A}. This implies that $k\le 2$ or
  $(k,\ell)=(3,2)$. For $(k,\ell)=(2,3)$, the resulting tree is not optimal,
  cf.\ \rref{replacement-full:heavy-trees}. The remaining cases correspond to
  $T\in\{T_8^*,T_9^*,T_{12}^*,T_{13}^*, T_{16}^*, T_{20}^*\}$.
\end{proof}

We are now able to prove Theorem~\ref{theorem:local-structure-again}.
\begin{proof}[Proof of Theorem~\ref{theorem:local-structure-again}]
  As $T\notin\calS$,  $T$ has no
  light vertex by Proposition~\ref{proposition:optimal-tree-with-light-vertex-new}
  and fulfils LC1--LC4. By Proposition~\ref{proposition:heavy-trees}, $T$
  contains no $3$-claw as a rooted subtree, so $T$ fulfils LC1--LC4 and LC6.
  By Lemmata~\ref{lemma:good-ends}, \ref{lemma:bad-ends}, \ref{lemma:strange-special-cases},
  \ref{lemma:4-claw-is-bad} and LC6, $T$ contains a $2$-claw and fulfils LC5 by Lemma~\ref{lemma:2-claw-implies-LC5}.
\end{proof}

We conclude this section with refined bounds on $\rho$ for subtrees of type $A$ of optimal trees.
The bounds only depend on the LC.

\begin{lemma}\label{lemma:A-good-upper-bound-rho}
  Let $T$ be a tree fulfilling the LC and let $A$ be a rooted subtree of $T$ of order $>1$ and type $A$.
  Then
  \begin{equation*}
    \frac23\le \rho(A)< \sqrt{3}-1 \le 0.7321.
  \end{equation*}
  If no branch of the unique branch of $A$ is a leaf, then
  \begin{equation*}
    \rho(A)< 0.688.
  \end{equation*}
\end{lemma}
\begin{proof}
  We prove the result by induction on the order of $A$. 
  By the LC, we have $A=\calA\calB(T_1,\ldots,T_r)$ for some $r\in\{2,3\}$ and
  rooted trees $T_1$, \ldots, $T_r$. If $r=2$, then both $T_1$ and $T_2$ are leaves by the LC and
  $A=F$ with $\rho(F)=2/3$. 

  We now consider the case $r=3$. 
  By the LC, there are at most two leaves among $T_1$, $T_2$, $T_3$. Thus
  $2\le \rho(T_1)+\rho(T_2)+\rho(T_3)< 2+(\sqrt{3}-1)$ by the induction
  hypothesis. We obtain
  \begin{equation*}
    \frac23=\frac1{1+\frac1{2}}\le\rho(A)=\frac1{1+\frac1{\rho(T_1)+\rho(T_2)+\rho(T_3)}}<
    \frac1{1+\frac1{1+\sqrt{3}}}=\frac{1+\sqrt{3}}{2+\sqrt{3}}= \sqrt{3}-1.
  \end{equation*}

  If none of $T_1$, $T_2$, $T_3$ is a leaf, we use the upper bound
  $\rho(T_1)+\rho(T_2)+\rho(T_3)< 3(\sqrt3-1)$ to obtain
  \begin{equation*}
      \rho(A)=\frac1{1+\frac1{\rho(T_1)+\rho(T_2)+\rho(T_3)}}<\frac1{1+\frac1{3(\sqrt{3}-1)}}=\frac{21-3\sqrt
      3}{23}<0.688.
    \end{equation*}
\end{proof}

If there is a vertex of type $B$ and degree $4$ which is adjacent to two leaves (e.g., in a
$CL$) in an optimal tree, this has consequences to every vertex of type $B$, as the following
lemma shows.

\begin{lemma}\label{lemma:two-leaves-force-leave-everywhere}
  Let $T\notin\calS$ be an optimal tree.  If there is a vertex $v$
  of degree $4$ of $T$ which is adjacent to two leaves, then every vertex $w$
  of $T$ of type $B$ is adjacent to at least one leaf.
\end{lemma}
\begin{proof}
  We assume the contrary and denote the rooted connected components of $T-w$ by $S_0$, $S_1$, $S_2$, $S_3$
  where $v$ is contained in $S_0$ and the rooted connected components of $T-v$ by
  $T_0$, $T_1$, $L$, $L$ with $w$ contained in $T_0$. As $1=\rho(L)>\rho(S_3)$,
  Lemma~\ref{lemma:path} and \liref{lemma:exchange}{item:exchange-B} yield
  \begin{equation*}
    \frac53\le 1+\rho(T_1)\le \rho(S_1)+\rho(S_2)+0.1153<2(\sqrt3-1)+0.1153<1.58,
  \end{equation*}
  a contradiction.
\end{proof}

\section{The upper bound: global structure}\label{sec:global}

\subsection{Outline Graph}
Now we start with the discussion of the global structure of optimal trees. Let us first collect
a few results on the outline graph of an optimal tree.
\begin{lemma}\label{lemma:outline-properties}
  Let $T\notin\calS$ be an optimal tree of order $n$ and $T'$ its outline graph
  as defined in Definition~\ref{definition:outline-graph}. Then $T'$ has the
  following properties.
  \begin{enumerate}
  \item The leaves of $T'$ correspond to rooted subtrees of type $A$ of $T$,
    the non-leaves of $T'$ correspond to vertices of type $B$ of $T$.
  \item If there is a vertex of degree $3$ in $T'$, then $n\equiv 0\pmod{7}$,
    $T$ has the shape as given in Figure~\ref{fig:optimal-trees-7} and 
    \begin{equation*}
      k=\frac{n-7}7.
    \end{equation*}
  \item There is no vertex $v$ in $T'$ which is adjacent to an $L$ and an $F$.
  \item If $T'$ is of order $1$, then $n\equiv 1\pmod 7$ and $T=C^{(n-1)/7}L$ or $n\equiv 4\pmod 7$ and $T=C^{(n-4)/7}F$.
  \end{enumerate}
\end{lemma}

\begin{proof}
  \begin{enumerate}
  \item By construction, all special leaves of $T'$ correspond to rooted
    subtrees of type $A$ of $T$. A leaf in $T$ is either contained in some
    larger special leaf or is eventually seen as an $L$ in $T'$. All non-leaves
    of $T$ of type $A$ are either contained in some larger special edge or
    special leaf or they are transformed into a $C^0_*$. Thus all non-leaves of
    $T'$ have to correspond to vertices of type $B$ of $T$.
  \item We denote the vertex of degree $3$ by $v$. By
    Theorem~\ref{theorem:local-structure-again}, two of the neighbours of $v$ are
    leaves. Thus we have one of the situations in Figure~\ref{fig:outline-degree-3}
    \begin{figure}[htbp]
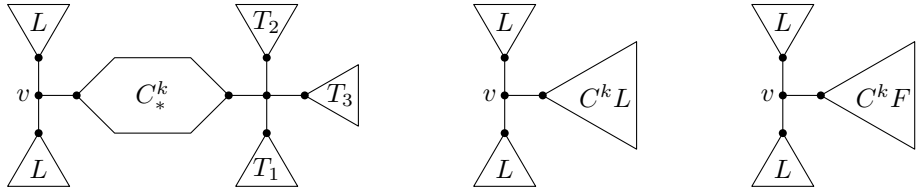

      \rule{0pt}{0pt}\hfill
      \includegraphics{outline-degree-3.1}\hfill
      \includegraphics{outline-degree-3.2}\hfill
      \includegraphics{outline-degree-3.3}\hfill
      \rule{0pt}{0pt}\hfill
      \caption{Possible Cases in Lemma~\ref{lemma:outline-properties}.}
      \label{fig:outline-degree-3}
    \end{figure}
    for appropriate $k\ge 0$ and rooted trees $T_1$, $T_2$, $T_3$.
    The first case is a contradiction to the construction of the outline, as a
    subtree $C^kF$ would have been contracted earlier than the $C^k_*$. The
    second case is also not a correct outline, as this graph is isomorphic to a
    $C^kF$ (use the $L$ in the present $C^kL$ as the new root).

    So we are left with the third case. In this case, we have $n=|T|=3+7k+4$,
    which immediately implies $n\equiv 0\pmod{7}$ and $k=(n-7)/7$. And this is
    exactly the situation in Figure~\ref{fig:optimal-trees-7}.
  \item We assume that there is a vertex $v$ in $T'$ which is adjacent to an $L$
    and an $F$. This could mean one of the situations in
    Figure~\ref{fig:cases-lemma-outline-properties}, where $S_j\in\{L,F\}$ for $j\in\{0,1,2\}$.
    \begin{figure}[htbp]
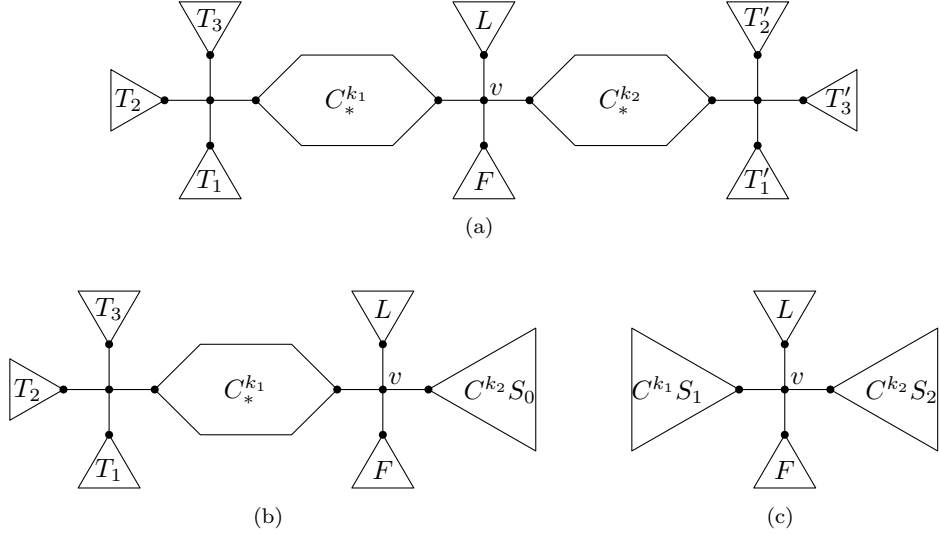

      \centering
      \subfigure[\label{fig:lemma-outline-properties-case-a}]{\includegraphics{outline-LF.1}}\\
      \phantom{A}\hfill\subfigure[\label{fig:lemma-outline-properties-case-b}]{\includegraphics{outline-LF.2}}\hfill
      \subfigure[\label{fig:lemma-outline-properties-case-c}]{\includegraphics{outline-LF.3}}\hfill\phantom{A}
      \caption{Possible cases in Lemma~\ref{lemma:outline-properties}}
      \label{fig:cases-lemma-outline-properties}
    \end{figure}
    \begin{enumerate}
    \item In the situation in Figure~\ref{fig:lemma-outline-properties-case-a},
      a $C^{k_1+k_2+1}_*$ would have been used in the outline of $T$ instead of $v$, $L$, $F$,
      $C^{k_1}_*$ and $C^{k_2}_*$.
    \item In the situation in Figure~\ref{fig:lemma-outline-properties-case-b},
      a $C^{k_1+k_2+1}S_0$ would have been used in the outline of $T$ instead of  $v$, $L$, $F$,
      $C^{k_1}_*$ and $C^{k_2}S_0$.
    \item We consider the situation in
      Figure~\ref{fig:lemma-outline-properties-case-c}.

      If $S_1=S_2=F$, the graph is isomorphic to the graph in
      Figure~\ref{fig:optimal-trees-7}, i.e., a $C^{k_1+k_2+1}F$ would have
      been combined with a vertex of degree $3$ and two leaves in the outline of $T$ .

      If $S_1=L$, the graph is isomorphic to $C^{k_1+k_2+1}S_2$, and this
      would have been taken in the outline of $T$ .
    \end{enumerate}
  \item If $T'$ is of order $1$, then the unique vertex of $T'$ must be a $C^kF$
    or a $C^kL$ for a suitable $k$. In the first case, we have $n=7k+4$, in the
    second $n=7k+1$.
  \end{enumerate}
\end{proof}

\subsection{Chains}
In the global structure of optimal trees, chains as introduced in
Definition~\ref{definition:chain} occur prominently. This subsection is devoted to
the computation of the relevant parameters and to some further necessary
optimality criteria in relation to chains. Recall that the growth constant $\lambda \approx 10.1097722286464$ in Theorem~\ref{thm:upper} is defined as the larger root of the polynomial $x^2-11x+9$. In the following, the other root
of this polynomial is denoted by $\lambdabar\approx 0.890227771353556$.

\begin{lemma}\label{lemma:chain-parameter}
  \begin{enumerate}
  \item Let $T$ be a rooted tree of type $A$. Then
    \begin{align*}
      \begin{pmatrix}
        m(CT)\\m_0(CT)
      \end{pmatrix}&=M
      \begin{pmatrix}
        m(T)\\m_0(T)
      \end{pmatrix},&
      \rho(CT)&=\sigma(\rho(T)),
    \end{align*}
    with
    \begin{align*}
      M&=
      \begin{pmatrix}
        8&3\\
        5&3
      \end{pmatrix},&
      \sigma:\mathbb{R}^+\to\mathbb{R}^+; x\mapsto 1-\frac3{8+3x}.
    \end{align*}
  \item We have
    \begin{equation*}
    \begin{aligned}
      m(C^kL)&=G_{k+1},&m_0(C^kL)&=G_{k+1}-3G_k,\\
      m(C^kF)&=3G_{k+1}-3G_k,&m_0(C^kF)&=2G_{k+1}-G_k,
    \end{aligned}\qquad\text{ with }\qquad G_k=\frac{\lambda^k-\lambdabar{}^k}{\lambda-\lambdabar}.
  \end{equation*}
\item Let $\rholimit=(\lambda-8)/3\approx 0.7032574$ and $x>0$. If $x<\rholimit$, then the sequence
  $\sigma^k(x)$ is strictly increasing, if $x>\rholimit$, then the sequence
  $\sigma^k(x)$ is strictly decreasing. In both cases,
  $\lim_{k\to\infty}\sigma^k(x)=\rholimit$.

  In particular, $\rho(C^kF)$ is strictly increasing and $\rho(C^kL)$ is
  strictly decreasing.
  \end{enumerate}
\end{lemma}
\begin{proof}
  \begin{enumerate}
  \item This is a straightforward consequence of the recursive formul\ae{}
    \eqref{equation:m_0-formula} and
    \eqref{equation:m_1-formula} for
    $m$, $m_0$ and $\rho$.
  \item The eigenvalues of $M$ are $\lambda$ and $\lambdabar$. Thus the
    sequences $m(C^kL)$, $m_0(C^kF)$, $m(C^kF)$, $m_0(C^kF)$ for $k\ge 0$ are
    elements of the linear space spanned by $\lambda^k$ and
    $\lambdabar{}^k$. Another basis of this linear space is given by $G_{k+1}$
    and $G_{k}$. It therefore suffices to check the formul\ae{} for $k=0$ and $k=1$.
  \item It is easily checked that $\rholimit$ is the unique positive fixed point
    of $\sigma$. The assertions on $\sigma^k(x)$ are easy consequences of the
    definition of $\sigma$. Finally, the assertions on $\rho(C^kF)$ and
    $\rho(C^kL)$ follow from $\rho(F)=2/3$ and $\rho(L)=1$.
  \end{enumerate}
\end{proof}

Next we show that an $L$ and a $C^kF$ never occur as neighbours of the
same vertex in the outline of an optimal tree.

\begin{lemma}\label{lemma:outline-ends-LCFT3T4}
  Let $T\notin\calS$ be an optimal tree, $v$ a vertex of degree
  $4$ and $T_1$, $T_2$, $T_3$, $T_4$ the rooted connected components of $T-v$. We
  assume that $T_1=L$ and $T_2=C^kF$ for some $k\ge 0$. Then $F\in\{T_2,T_3,T_4\}$. In
  particular, $v$ is not in the outline of $T$.
\end{lemma}
\begin{proof}
  We assume $k>0$. Without loss of generality, we also assume that
  $\rho(T_3)\le\rho(T_4)$. The tree $T$ has the shape shown in Figure~\ref{fig:outline-ends.5}.
  If $\rho(T_3)>\rho(F) = 2/3$, we have $\rho(L)+\rho(T_4)\le \rho(L)+\rho(F)$ by
  \liref{lemma:exchange}{item:exchange-symmetric}, which implies that $\rho(T_4)=2/3$ and
  $\rho(T_3)=2/3$ by Lemma~\ref{lemma:A-good-upper-bound-rho}, contradiction.

  \begin{figure}[htbp]
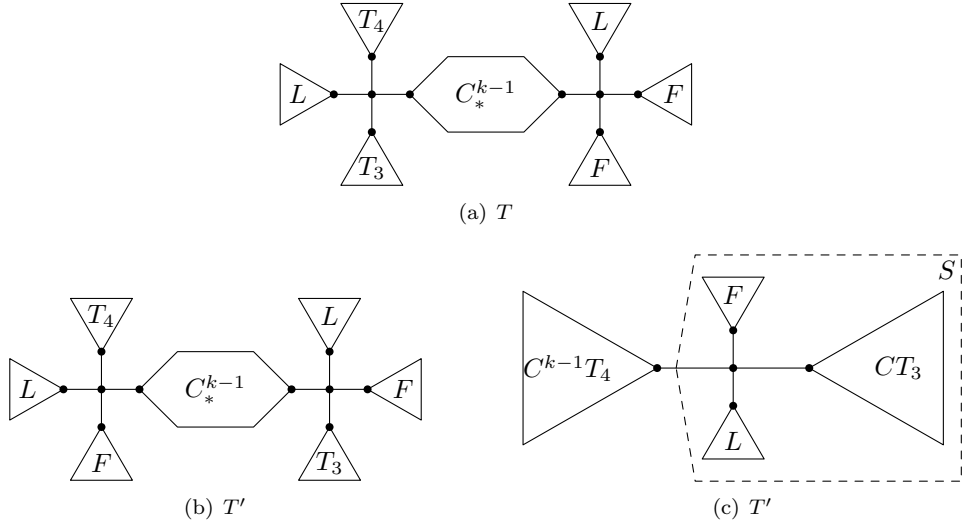

    \centering
    \subfigure[$T$\label{fig:outline-ends.5}]{\includegraphics{outline-ends.5}}\\
    \rule{0pt}{0pt}\hfill\subfigure[$T'$\label{fig:outline-ends.6}]{\includegraphics{outline-ends.6}}
    \hfill\subfigure[$T'$\label{fig:outline-ends.7}]{\includegraphics{outline-ends.7}}\hfill\rule{0pt}{0pt}
    \caption{Trees $T$ and $T'$ considered in Lemma~\ref{lemma:outline-ends-LCFT3T4}.}
  \end{figure}

  Therefore, we have $\rho(T_3)=2/3$, thus $T_3\in\{F, A_{14}^*,
  A_{24}^*\}$ by
  Lemma~\ref{lemma:rho-A-2-3}. In the case
  $T_3=F$, Lemma~\ref{lemma:outline-properties} yields
  the result. Thus we are left with
  $T_3\in\{A_{14}^*, A_{24}^*\}$.

  We consider the tree $T'$ where $F$ and $T_3$ have been exchanged, cf. Figure~\ref{fig:outline-ends.6}.
  From Lemma~\ref{lemma:m-m-0-via-continuants} we conclude that $m(T')=m(T)$, 
  i.e., $T'$ is also an optimal tree. We rewrite $T'$ as in
  Figure~\ref{fig:outline-ends.7}. For $T_3\in\{A_{14}^*, A_{24}^*\}$, the
  rooted tree $S$ is not $\alpha$-optimal for any $\alpha>1/3$, cf.\ \rref{replacement:outline-ends-LCFT3T4}.
  As
   $\rho(C^{k-1}T_4)\ge 2/3$ by
   Lemma~\ref{lemma:A-good-upper-bound-rho}, this is a
   contradiction to Proposition~\ref{proposition:subtrees-of-optimal-trees-are-alpha-optimal}
   and the optimality of $T'$.
\end{proof}

We now prove a necessary optimality condition involving one chain element. 
\begin{lemma}\label{lemma:simple-chain-exchange}
  Let $T\notin\calS$ be an optimal tree, $v$ a vertex of type
  $B$ and degree $4$ of $T$ and $CS_1$, $S_2$, $S_3$, $S_4$ the rooted connected
  components of $T-v$ for some rooted trees $S_1$, $S_2$, $S_3$, $S_4$.
  \begin{enumerate}
  \item If neither $S_3$ nor $S_4$ is a leaf, then 
    \begin{equation*}
      \rho(S_1)\le \rho(S_2).
    \end{equation*}
  \item If $S_3$ is a leaf and $\rho(S_4)>2/3$, then 
    \begin{equation*}
      \rho(S_1)\ge \rho(S_2).
    \end{equation*}
  \end{enumerate}
\end{lemma}
\begin{proof}
  If $\rho(S_1)>\rho(S_2)$, then $5/3=\rho(L)+\rho(F)\le \rho(S_3)+\rho(S_4)$
  by \liref{lemma:exchange}{item:exchange-symmetric},
  i.e., $S_3$  or $S_4$ is a leaf by Lemma~\ref{lemma:A-good-upper-bound-rho}.

  If $\rho(S_1)<\rho(S_2)$, then $5/3\ge \rho(S_3)+\rho(S_4)$ by \liref{lemma:exchange}{item:exchange-symmetric}, i.e., either
  both $S_3$ and $S_4$ are non-leaves or $S_3$, say, is a leaf and
  $\rho(S_4)\le 2/3$.

  The contrapositions are the statements of the lemma.

\end{proof}
The following lemma lists some consequences of this result. 

\begin{lemma}\label{lemma:outline-ends}
  Let $T\notin\calS$ be an optimal tree, $v$ a vertex of
  degree $4$ and $T_0$, $T_1$, $T_2$, $T_3$ the rooted connected components of $T-v$.
  \begin{enumerate}
  \item\label{item:outline-ends-CTCTCTCT-T-small} Let $\ell\le 4$ and let $T_j=C^{k_{j}}S_{j}$ for some
    $k_j\ge 0$ and some rooted tree $S_j$ with $\rho(S_j)<7/10$ for
    $j\in\{0,\ldots,\ell-1\}$. Further assume that $T_j$ is not a leaf for
    $j\in\{\ell,\ldots,3\}$ and that $k_0\le \cdots \le k_{\ell-1}$.

    Then $k_{\ell-1}\le k_0+1$ and, if $k_j<k_{j+1}$ for some $0\le j<\ell-1$, then $\rho(S_j)\ge \rho(S_{j+1})$.
  \item\label{item:outline-ends-LCT1CT2-T-large} If $T_0=L$, $T_1=C^{k_1}S_1$,
    $T_2=C^{k_2}S_2$ with $k_1$, $k_2\ge 0$,
    $\rho(C^{\ell_j+1}L)<\rho(S_j)\le\rho(C^{\ell_j}L)$ with $\ell_j\in\{0,1\}$ for
    $j\in\{1,2\}$ and finally $\rho(T_3)>2/3$, then we have $k_2+\ell_2\le k_1+\ell_1+1$.

    Furthermore, if $\sigma^{-\ell_1}(\rho(S_1))>\sigma^{-\ell_2}(\rho(S_2))$, then
    $k_2\le \max\{0,k_1+\ell_1-\ell_2\}$.
  \end{enumerate}
  
\end{lemma}
\begin{proof}
  \begin{enumerate}
  \item Assume that $k_{\ell-1}>k_0+1$. Then
    \begin{equation*}
      \rho(C^{k_{\ell-1}-1}F)\le \rho(C^{k_{\ell-1}-1}S_{\ell-1})\le \rho(C^{k_0}S_0)<\rho(C^{k_0+1}F)
    \end{equation*}
    by the monotonicity of $\sigma$,
    Lemma~\ref{lemma:A-good-upper-bound-rho},
    Lemma~\ref{lemma:simple-chain-exchange} and the fact that
    $\sigma(2/3)=7/10$. We conclude that $k_{\ell-1}-1<k_0+1$, i.e., $k_{\ell-1}\le k_0+1$,
    contradiction. Thus $k_{\ell-1}\le k_0+1$.

    If $k_j<k_{j+1}$, then $k_{j+1}=k_j+1$. From
    Lemma~\ref{lemma:simple-chain-exchange} we see that
    \begin{equation*}
      \rho(C^{k_j}S_{j+1})=\rho(C^{k_{j+1}-1}S_{j+1})\le \rho(C^{k_j}S_j),
    \end{equation*}
    which yields $\rho(S_{j+1})\le\rho(S_j)$ in view of the monotonicity of $\sigma$.
  \item Assume that $k_2>0$. By Lemma~\ref{lemma:simple-chain-exchange}, we
    have
    \begin{equation*}
      \rho(C^{k_2+\ell_2-1}L)\ge\rho(C^{k_2-1}S_2)\ge \rho(C^{k_1}S_1)>\rho(C^{k_1+\ell_1+1}L),
    \end{equation*}
    which yields $k_2+\ell_2-1<k_1+\ell_1+1$ by
    Lemma~\ref{lemma:chain-parameter} and therefore $k_2+\ell_2\le
    k_1+\ell_1+1$. 
    If $k_2=0$, then the same inequality holds trivially as
    $\ell_2$ has been assumed to be at most $1$.

    Now we turn to the second assertion and assume $k_2>0$ and
    $\sigma^{-\ell_1}(\rho(S_1))>\sigma^{-\ell_2}(\rho(S_2))$. We have
    \begin{equation*}
      \sigma^{k_2-1}(\rho(S_2))=\rho(C^{k_2-1}S_2)\ge
      \rho(C^{k_1}S_1)=\sigma^{k_1+\ell_1}(\sigma^{-\ell_1}(\rho(S_1)))>
      \sigma^{k_1+\ell_1-\ell_2}(\rho(S_2))
    \end{equation*}
    by Lemma~\ref{lemma:simple-chain-exchange},
    which yields $k_2-1<k_1+\ell_1-\ell_2$, i.e., $k_2\le k_1+\ell_1-\ell_2$,
    as required. For $k_2=0$, there is nothing to show.
  \end{enumerate}
\end{proof}

\subsection{Switching Forks and Leaves}
So far, we mainly compared optimal trees to trees where some subtrees have
been switched between two positions. It turns out that more invasive operations
are needed in order to obtain information on the global structure of optimal
trees. 

The basic idea is the following: If $7$ forks are replaced by $7$ leaves,
the order of the tree is reduced by $7\cdot 3=21$. As a chain element requires
$7$ vertices, these $21$ ``free'' vertices can be used to introduce $3$ chain
elements. If all this is done at the ``right'' positions, then $m(T)$
increases. In some circumstances, however, the inverse operation may be
beneficial.

Before we state the main lemma regarding such exchange operations, we collect two technical details
concerning the floor function in the following lemma.

\begin{lemma}\label{lemma:floor}
  \begin{enumerate}
  \item For real $x$ and positive integers $d$, the identity
    \begin{equation*}
      \floor{dx}=\sum_{j=0}^{d-1}\floor{x+\frac jd}
    \end{equation*}
    holds.
  \item Let $d\ge 0$ and $k_0$, \ldots, $k_{d-1}$ be integers with
    \begin{equation*}
      k_0\le k_1\le \cdots \le k_{d-1}\le k_0+1.
    \end{equation*}
    Then
    \begin{equation*}
      k_j=\floor{\frac{k+j}d} \text{\qquad with\qquad}
      k=k_0+k_1+\cdots+k_{d-1}
    \end{equation*}
    for $j\in\{0,\ldots,d-1\}$.
  \end{enumerate}
\end{lemma}
\begin{proof}
  \begin{enumerate}
  \item Cf.\ Graham, Knuth and Patashnik~\cite[(3.26)]{Graham-Knuth-Patashnik:1994}.
  \item Choose $1\le r\le d$ such that $k_0=\ldots=k_{r-1}<k_r\le \cdots\le
    k_{d-1}=k_0+1$. Then $k=dk_0+(d-r)$ and
    \begin{equation*}
      \floor{\frac{k+j}d}=\floor{k_0+\frac{d+j-r}{d}}=k_0+[d+j-r\ge
      d]=k_0+[j\ge r]=k_j
    \end{equation*}
    for $j\in\{0,\ldots,d-1\}$.
  \end{enumerate}
\end{proof}

Now we are able to provide the required exchange operations. Their
consequences will be exploited afterwards. The proof relies on
similar ideas as the proof of Lemma~\ref{label:3-claw-replacements}.

\begin{lemma}\label{lemma:switching-7forks-to-7leaves}
  Let $T$ be a tree fulfilling the LC and
  let $T'$ be a tree that is obtained from $T$ by replacing one rooted subtree $S$ by $S'$,
  where $S$ and $S'$ will be specified below.
  \begin{enumerate}
    \item\label{item:switching-7forks-to-7leaves:F1} If $S=\calB(C^{k_0}F, C^{k_1}F, C^{k_2}F)$ with
      $k_j=\floor{(k+j)/3}$ for some $k\ge 0$ and
      $S'=\calB(C^{\lfloor k/2\rfloor+1}L,\allowbreak C^{\lfloor(k+1)/2\rfloor+1}L, L)$, then $m(T')/m(T)\ge 5.211$ and $|T'|-|T|=5$.
    \item\label{item:switching-7forks-to-7leaves:F2} If
      $S=\calB(C^{k_0}F, C^{k_1}F, C^{k_2}F)$ with
      $k_j=\floor{(k+j)/3}$ for some $k\ge 0$ and
      $S'=\calB(C^{\lfloor k/2\rfloor+1}L,\allowbreak C^{\lfloor(k+1)/2\rfloor}L, L)$, then $m(T')/m(T)\ge 0.5154$ and $|T'|-|T|=-2$.
    \item\label{item:switching-7forks-to-7leaves:F3} If $S=\calB(C^{k_0}F, C^{k_1}F, C^{k_2}F)$ with
      $k_j=\floor{(k+j)/3}$ for some $k\ge 0$ and $S'=\calB(C^{k}F, F, L)$, then $m(T')/m(T)\ge 0.3726$ and $|T'|-|T|=-3$.
    \item\label{item:switching-7forks-to-7leaves:F4} If \[S=\calB(C^{k+\lfloor
        (j+1)/3\rfloor}F,C^{k+\lfloor
        (j+2)/3\rfloor}F, C^k\calA\calB(C^{k+\lfloor
        (i+1)/4\rfloor}F,C^{k+\lfloor
        (i+2)/4\rfloor}F,C^{k+\lfloor
        (i+3)/4\rfloor}F) )\]
      and $S'=\calB(L,F,C^{2k+i+j}\calA\calB(C^{2k+1}L,C^{2k+1}L,L))$ for some
      $k\ge 0$, $i\in\{0,1,2,3\}$, $j\in\{0,1,2\}$, then 
      $m(T')/m(T)\ge 1.943$ and $|T'|-|T|=2$.
    \item\label{item:switching-7leaves-to-7forks:F5} If $S=\calB(L,C^{\lfloor
        k/2\rfloor}L,C^{\lfloor (k+1)/2\rfloor}L)$ and $S'=\calB(L,F,C^{k-1}F)$
      for some $k\ge 1$, then $m(T')/m(T)\ge 0.722$ and $|T'|-|T|=-1$.
    \item\label{item:switching-7leaves-to-7forks:F6} If $S=\calB(L,C^{\lfloor
        k/2\rfloor}L,C^{\lfloor (k+1)/2\rfloor}L)$ and $S'=\calB(C^{\lfloor (k-1)/3\rfloor}F,C^{\lfloor k/3\rfloor}F,C^{\lfloor (k+1)/3\rfloor}F)$
      for some $k\ge 1$, then $m(T')/m(T)\ge 27/14$ and $|T'|-|T|=2$. If $k\ge
      2$, then $m(T')/m(T)\ge 1.9302$.
    \item\label{item:switching-7leaves-to-7forks:F7} If
$S=\calB(C^{k+1+t}L,L,C^k\calA\calB(C^{k+1+\lfloor
        (s+1)/3\rfloor}L,C^{k+1+\lfloor (s+2)/3\rfloor}L,L))$
      and 
      $S'=\calB(F,L,\allowbreak C^{k-1+\lfloor(s+t+2)/4\rfloor}\calA\calB(C^{k+\lfloor(s+t+3)/4\rfloor}F,C^{k+\lfloor(s+t+4)/4\rfloor}F,C^{k+\lfloor(s+t+5)/4\rfloor}F))$
      for some $k\ge 1$, $s\in\{0,1,2\}$, $t\in\{0,1\}$, then $m(T')/m(T)\ge
      0.5181$ and $|T'|-|T|=-2$.
    \item\label{item:switching-7leaves-to-7forks:F8} If  \[S=\calB(C^{t}L,L,\calA\calB(C^{\lfloor
        s/2\rfloor}L,C^{\lfloor (s+1)/2\rfloor}L,L))\]
      and
      $S'=\calB(C^{\lfloor(s+t-1)/3\rfloor}F,C^{\lfloor(s+t)/3\rfloor}F,C^{\lfloor(s+t+1)/3\rfloor}F)$
      for some $s\in\{1,2,3,4\}$ and $t\in\{0,1,2\}$, then $m(T')/m(T)\ge
      0.516$ and $|T'|-|T|=-2$.
    \item\label{item:switching-7leaves-to-7forks:F10} If  $S=\calB(C^{t}L,L,\calA\calB(C^{\lfloor
        s/2\rfloor}L,C^{\lfloor (s+1)/2\rfloor}L,L))$
      and
      $S'=\calB(L,F,C^{s+t}F)$
      for some $s\in\{1,2,3,4\}$ and $t\in\{0,1,2\}$, then $m(T')/m(T)\ge 1.95$ and $|T'|-|T|=2$.
  \end{enumerate}
\end{lemma}
\begin{proof}
  \begin{enumerate}
  \item Let $T$ consist of some rooted subtree $T_1$, $S$ and the edge between
    the roots of $T_1$ and $S$. Then from Lemma~\ref{lemma:no-adjacent-vertices-of-type-A} we
    obtain
    \begin{align*}
      \frac{m(T')}{m(T)}&=\frac{m_0(S')}{m_0(S)}\cdot\frac{\rho(T_1)+\rho(C^{\lfloor k/2\rfloor+1}L)+\rho(C^{\lfloor(k+1)/2\rfloor+1}L)+1}{\rho(T_1)+\rho(C^{k_0}F)+\rho(C^{k_1}F)+\rho(C^{k_2}F)}\\
      &=\frac{m_0(S')}{m_0(S)}\left(1+\frac{\rho(C^{\lfloor k/2\rfloor+1}L)+\rho(C^{\lfloor(k+1)/2\rfloor+1}L)+1-\rho(C^{k_0}F)-\rho(C^{k_1}F)-\rho(C^{k_2}F)}{\rho(T_1)+\rho(C^{k_0}F)+\rho(C^{k_1}F)+\rho(C^{k_2}F)}\right)\\
      &\ge \frac{m_0(S')}{m_0(S)}\left(1+\frac{\rho(C^{\lfloor k/2\rfloor+1}L)+\rho(C^{\lfloor(k+1)/2\rfloor+1}L)+1-\rho(C^{k_0}F)-\rho(C^{k_1}F)-\rho(C^{k_2}F)}{\sqrt{3}-1+\rho(C^{k_0}F)+\rho(C^{k_1}F)+\rho(C^{k_2}F)}\right)
    \end{align*}
    where 
    \begin{equation*}
      \rho(C^{\lfloor k/2\rfloor+1}L)+\rho(C^{\lfloor(k+1)/2\rfloor+1}L)+1>3\rholimit>\rho(C^{k_0}F)+\rho(C^{k_1}F)+\rho(C^{k_2}F),
    \end{equation*}
    cf.\ Lemma~\ref{lemma:chain-parameter}, and Lemma~\ref{lemma:A-good-upper-bound-rho} have been used. From
    Lemma~\ref{lemma:chain-parameter} and Lemma~\ref{lemma:floor}, we get
    \begin{align*}
      \frac{m(T')}{m(T)}&\ge
      \frac{\lambda^4(\lambda-\lambdabar)}{(3\lambda-3)^3}(1+O(q^{k/3}))\frac{\sqrt{3}-1+2\rholimit(1+O(q^{k/2}))+1}{\sqrt{3}-1+3\rholimit(1+O(q^{k/3}))}\\
      &\ge 5.21101232
    \end{align*}
    with $q=\lambdabar/\lambda\approx 0.088$, where we replaced the explicit formul\ae{} obtained from
    Lemma~\ref{lemma:chain-parameter} by asymptotic expansions for ease of
    presentation; the actual computations leading to the given constant have
    been performed exactly---in this particular case, it even turned out that
    the whole expression was strictly decreasing in $k$. The explicit branch
    $L$ of $S'$ has been taken into account exactly instead of using
    Lemma~\ref{lemma:chain-parameter}.

    We have $|S|=1+7(k_0+k_1+k_2)+3\cdot 4=13+7k$ and $|S'|=1+7(2+\lfloor
    k/2\rfloor+\lfloor(k+1)/2\rfloor)+3=18+7k$ by Lemma~\ref{lemma:floor}.
  \item Analogous.
  \item Analogous.
  \item Analogous.
  \item Analogous, but the lower bound $\rho(T_1)\ge 2/3$
    (Lemma~\ref{lemma:A-good-upper-bound-rho}) has to be used, 
    as $\rho(L)+\rho(F)+\rho(C^{k-1}F)<1+2\rholimit<\rho(L)+\rho(C^{\lfloor
        k/2\rfloor}L)+\rho(C^{\lfloor (k+1)/2\rfloor}L)$.
  \item Analogous.
  \item Analogous.
  \item Analogous, but simpler, as this is a finite case and no limits have to
    be considered.
  \item Analogous.
  
  \end{enumerate}
\end{proof}

\begin{remark}
  The precise proof of Lemma~\ref{lemma:switching-7forks-to-7leaves} has been
  carried out using Sage~\cite{Stein-others:2010:sage-mathem}. The program is
  available in \cite{Heuberger-Wagner:max-card-matching-sage}.
\end{remark}

\subsection{\texorpdfstring{$CL$}{CL}-free Optimal Trees}\label{section:A-8-1-free}
Throughout this subsection, we assume that $T\notin\calS$ is an optimal tree which is $CL$-free, i.e., it does not contain a $CL$ as a rooted subtree. Obviously, such a tree does not contain any $C^kL$ as rooted subtree for $k\ge 1$. We will
describe all optimal trees with this property.

\begin{lemma}\label{lemma:A-8-1-free-possible-end-configuration}
  Let $T\notin\calS$ be a $CL$-free optimal tree and $v$ be
  a vertex of degree $4$ in
  the outline graph of $T$ which is adjacent to at least three ``special
  leaves'' $T_0$, $T_1$, $T_2$ with $|T_0|\le |T_1|\le |T_2|$. Then there is a
  $k\ge 0$ such that $T_j=C^{k_j}F$ with $k_j=\lfloor(k+j)/3\rfloor$ for $j\in\{0,1,2\}$.
\end{lemma}
\begin{proof}
  As $T$ is $CL$-free, $T_j\neq C^{k_j}L$ for any $k_j>0$ and
  $j\in\{0,1,2\}$. So for every $j$ we have $T_j=L$ or $T_j=C^{k_j}F$ for some $k_j$.
  As $T$ is optimal, the case $T_0=T_1=T_2=L$ is excluded by
  LC6. The cases $T_0=T_1=L$, $T_2=C^{k_2}F$ and
  $T_0=L$, $T_1=C^{k_1}F$, $T_2=C^{k_2}F$ are excluded by
  Lemma~\ref{lemma:outline-ends-LCFT3T4}. Thus
  $T_j=C^{k_j}F$ for some $k_j\ge 0$ for all $j$. As $v$ is in the outline of
  $T$, it is not adjacent to a leaf by Lemma~\ref{lemma:outline-ends-LCFT3T4}.
  Thus
  \liref{lemma:outline-ends}{item:outline-ends-CTCTCTCT-T-small}
  (with $\ell=3$) and Lemma~\ref{lemma:floor} prove the assertion.
\end{proof}

\begin{lemma}\label{lemma:A-8-1-free-3-ends}
  Let $T\notin\calS$ be a $CL$-free optimal tree. Then
  there are no three distinct vertices  in the outline of $T$ such
  that each of them is adjacent to three ``special leaves''.
\end{lemma}
\begin{proof}
  Assume that there are three distinct vertices $v_0$, $v_1$, $v_2$ in the outline of
  $T$ such that $v_i$ is adjacent to $C^{k_{i0}}F$, $C^{k_{i1}}F$,
  $C^{k_{i2}}F$ with $k_{ij}=\lfloor(k_i+j)/3\rfloor$ for some $k_i\ge 0$ and
  $j\in\{0,1,2\}$. By
  Lemma~\ref{lemma:A-8-1-free-possible-end-configuration} this is the only case
  to consider.

  Replacing the rooted subtree with root $v_0$ and branches $C^{k_{00}}F$,
  $C^{k_{01}}F$, $C^{k_{02}}F$ by a rooted subtree with root $v_0$ and branches $C^{\lfloor k_0/2\rfloor+1}L$, $C^{\lfloor(k_0+1)/2\rfloor+1}L$,
  $L$, cf.\ \liref{lemma:switching-7forks-to-7leaves}{item:switching-7forks-to-7leaves:F1}, yields a tree $T'$ which is not necessarily optimal, but fulfils the LC.

  Replacing the rooted subtree with root $v_1$ and branches $C^{k_{10}}F$,
  $C^{k_{11}}F$, $C^{k_{12}}F$ in $T'$ by a rooted subtree with root $v_1$ and branches $C^{\lfloor k_1/2\rfloor+1}L$, $C^{\lfloor(k_1+1)/2\rfloor}L$,
  $L$, cf.\ \liref{lemma:switching-7forks-to-7leaves}{item:switching-7forks-to-7leaves:F2}, yields a tree $T''$, which still fulfils the LC.

  Replacing the rooted subtree with root $v_2$ and branches $C^{k_{20}}F$,
  $C^{k_{21}}F$, $C^{k_{22}}F$ in $T''$ by a rooted subtree with root $v_2$ and branches $C^{k_2}F$, $F$,
  $L$, cf.\ \liref{lemma:switching-7forks-to-7leaves}{item:switching-7forks-to-7leaves:F3}, yields a tree $T'''$.

  \lirefthree{lemma:switching-7forks-to-7leaves}{item:switching-7forks-to-7leaves:F1}{item:switching-7forks-to-7leaves:F2}{item:switching-7forks-to-7leaves:F3} yields $|T'''|-|T|=5-2-3=0$ and
  \begin{equation*}
    \frac{m(T''')}{m(T)}=\frac{m(T''')}{m(T'')}\cdot\frac{m(T'')}{m(T')}\cdot\frac{m(T')}{m(T)}\ge
     5.211\cdot 0.5154\cdot 0.3726>1.0007,
  \end{equation*}
  thus $m(T''')>m(T)$, contradiction to the optimality of $T$.
\end{proof}

\begin{lemma}\label{lemma:caterpillar-A-8-1-free}
  Let $T\notin\calS$ be a $CL$-free optimal tree of order $n$ and $u$,
  $v$ two distinct vertices of degree $4$ in the outline graph of $T$ which are
  adjacent to three special leaves. 

  Then  $n\equiv 6\pmod 7$ and $T$ is of the shape given in
  Figure~\ref{fig:optimal-trees-6} where
  \begin{equation*}
    k_j=\floor{\frac{n-27+7j}{49}}
  \end{equation*}
  for $0\le j\le 6$ or $(k_0,k_1,k_2,k_3,k_4,k_5,k_6)=(1,0,0,0,0,0,0)$.
\end{lemma}
\begin{proof}
  By Lemma~\ref{lemma:outline-properties}, the outline of $T$ does not contain
  a vertex of degree~$3$. By Lemma~\ref{lemma:A-8-1-free-3-ends}, the outline graph of $T$ is a
  caterpillar tree. By Lemma~\ref{lemma:A-8-1-free-possible-end-configuration},
  it must be of the shape shown in Figure~\ref{fig:caterpillar-A-8-1-free}
  for some $s\ge 0$, some non-negative integers $k_0$, \ldots, $k_8$ and some
  special leaves $T_j$, $T_j'$ for $j\in\{1,\ldots,s\}$. As $T$ is
  $CL$-free, $T_j$ is a leaf or a $C^{\ell_j}F$ and $T_j'$ is a leaf or
  a $C^{\ell'_j}F$ for suitable $\ell_j$, $\ell_j'\ge 0$ and
  $j\in\{1,\ldots,s\}$.
  \begin{figure}[htbp]
    \centering
    \includegraphics{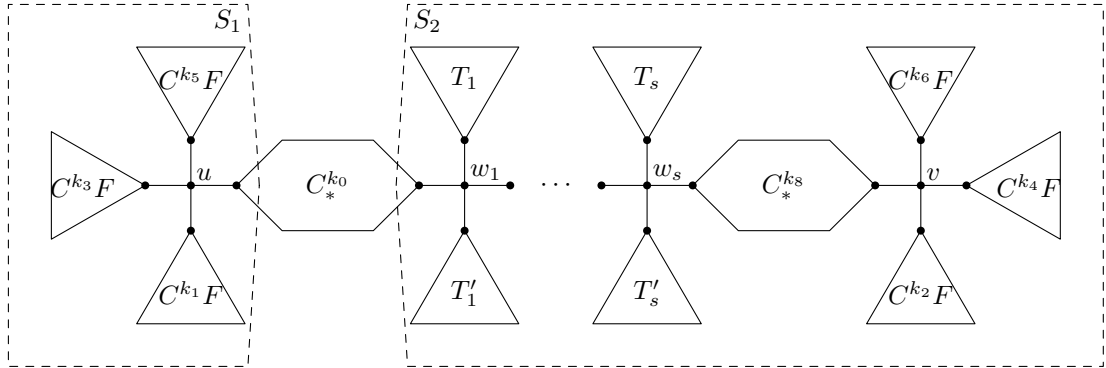}
    \caption{Decomposition of $T$ for Lemma~\ref{lemma:caterpillar-A-8-1-free}.}
    \label{fig:caterpillar-A-8-1-free}
  \end{figure}

  Assume that $s>0$.  By Lemma~\ref{lemma:two-leaves-force-leave-everywhere},
  it is impossible that both $T_1=L$ and $T_1'=L$, as $u$ is not adjacent to a
  leaf. By Lemma~\ref{lemma:outline-ends-LCFT3T4} we conclude that
  $T_1=C^{\ell_1}F$ and $T_2=C^{\ell_1'}F$ for some $\ell_1\ge 0$, $\ell_1'\ge
  0$. Lemma~\ref{lemma:A-good-upper-bound-rho} implies that $\rho(S_1)\le
  0.688$ and $\rho(S_2)\le 0.688$.

  W.l.o.g. we assume $k_1\le k_3\le k_5$. We claim that
  \begin{equation}\label{eq:caterpillar-C-L-free-chain-left}
     k_0\le k_1\le k_3\le k_5\le k_0+1.
  \end{equation}
  If $\rho(S_2)>2/3=\rho(F)$ or $k_0=0$, this follows from
  \liref{lemma:outline-ends}{item:outline-ends-CTCTCTCT-T-small}.
  So we consider the case that $\rho(S_2)=2/3$ and $k_0>0$. By
  Lemma~\ref{lemma:rho-A-2-3} and the shape of $T$ as shown in
  Figure~\ref{fig:caterpillar-A-8-1-free}, this implies that
  $S_2=A_{24}^*$. In particular, we have $T_1=T_1'=F$, i.e., $\ell_1 = \ell_1' = 0$. As $\ell_1<k_0$, we
  have $2/3=\rho(T_1)\ge\rho(S_1)$ and $k_0\le 1$ by
  \liref{lemma:outline-ends}{item:outline-ends-CTCTCTCT-T-small}.
  Thus we have $S_1=A_{14}^*$ by Lemma~\ref{lemma:rho-A-2-3} and
  $k_0=1$. Then $T=\calA\calB(F,A_{14}^*,A_{24}^*)$ which is not optimal by
  \rref{replacement:rho-A-2-3}. This concludes the proof of \eqref{eq:caterpillar-C-L-free-chain-left}.

  W.l.o.g we assume $\ell_1\le \ell_1'$. If $\rho(S_1)=2/3$, then
  $S_1=A_{14}^*$ and $k_1=k_3=k_5=0$, thus also $k_0=0$ by \eqref{eq:caterpillar-C-L-free-chain-left}. Thus $k_0\le \ell_1\le \ell_1'\le k_0+1$ in this case by \liref{lemma:outline-ends}{item:outline-ends-CTCTCTCT-T-small}. If
  $\rho(S_1)>2/3 = \rho(F)$, we get the same estimate $k_0\le \ell_1\le \ell_1'\le k_0+1$
  from \liref{lemma:outline-ends}{item:outline-ends-CTCTCTCT-T-small}.

  Replacing $S=\calB(T_1,T_1',C^{k_0}\calA\calB(C^{k_1}F,C^{k_3}F,C^{k_5}F))$ in
  $T$ by
  $S'$ as in \liref{lemma:switching-7forks-to-7leaves}{item:switching-7forks-to-7leaves:F4}  yields a tree $T'$ fulfilling
  the LC.

  Replacing $\calB(C^{k_2}F, C^{k_4}F, C^{k_6}F)$ in $T'$ as in
  \liref{lemma:switching-7forks-to-7leaves}{item:switching-7forks-to-7leaves:F2} yields a tree $T''$.

  By \lireftwo{lemma:switching-7forks-to-7leaves}{item:switching-7forks-to-7leaves:F2}{item:switching-7forks-to-7leaves:F4}, we have $|T''|=|T|$ and
  \begin{equation*}
    \frac{m(T'')}{m(T)}= \frac{m(T'')}{m(T')}\frac{m(T')}{m(T)}
    \ge 0.5154\cdot 1.943>1.001,
  \end{equation*}
  contradiction to the optimality of $m(T)$.

  Thus we have shown that $s=0$, i.e., $T$ is of the shape given in
  Figure~\ref{fig:optimal-trees-6}. We set $k=k_0+k_1+k_2+k_3+k_4+k_5+k_6$
  and have $n=7k+6\cdot4+3=27+7k$ and in particular $n\equiv 6\pmod 7$.

  We set $a=k_1+k_3+k_5$ and $b=k_2+k_4+k_6$. Without loss of generality we
  assume $a\le b$.
  By \liref{lemma:outline-ends}{item:outline-ends-CTCTCTCT-T-small} and Lemma~\ref{lemma:floor}, we have
  \begin{equation}\label{eq:caterpillar-A-8-1-free-prepare-balance}
    \begin{aligned}
      k_1&=\floor{\frac
        a3},&k_3&=\floor{\frac{a+1}3},&k_5&=\floor{\frac{a+2}3},&
      k_2&=\floor{\frac
        b3},&k_4&=\floor{\frac{b+1}3},&k_6&=\floor{\frac{b+2}3}.\\
    \end{aligned}
  \end{equation}
  If $b\ge a+2$, we obtain $k_6\ge\lfloor (a+4)/3\rfloor=k_3+1$, thus $\rho(C^{k_6}F) > \rho(C^{k_3}F)$, and
  \liref{lemma:exchange}{item:exchange-symmetric} and Lemma~\ref{lemma:chain-parameter} yield
  \begin{align*}
    \rho(C^{k_1}F)+\rho(C^{k_5}F)&<\rho(C^{\floor{(a+2)/3}}F)+\rho(C^{\floor{(a+3)/3}}F)\le \rho(C^{\floor{b/3}}F)+\rho(C^{\floor{(b+1)/3}}F)\\
    &=\rho(C^{k_2}F)+\rho(C^{k_4}F)\le \rho(C^{k_1}F)+\rho(C^{k_5}F),
  \end{align*}
  a contradiction. We conclude that $a\le b\le
  a+1$. From~\eqref{eq:caterpillar-A-8-1-free-prepare-balance} we immediately
  conclude that $k_1\le k_2\le k_3\le k_4\le k_5\le k_6\le k_1+1$ holds in both
  cases.

  From \liref{lemma:outline-ends}{item:outline-ends-CTCTCTCT-T-small} and
  Lemma~\ref{lemma:A-good-upper-bound-rho} we see that
  $k_0\le k_1+1$ and $k_6\le k_0+1$. 

  If $b=0$, we therefore obtain 
  \begin{equation*}
    (k_0,k_1,k_2,k_3,k_4,k_5,k_6)\in\{(0,0,0,0,0,0,0), (1,0,0,0,0,0,0)\},
  \end{equation*}
  where we have $T=T_{27}^*$ in the first case and $T=T_{34,2}^*$ in the second
  case.

  If $b>0$, then $\rho(S_2) > \frac23 = \rho(F)$ and thus $k_0\le
  k_1$ by \liref{lemma:outline-ends}{item:outline-ends-CTCTCTCT-T-small}. Thus we have
  \begin{equation*}
    k_0\le k_1\le k_2\le k_3\le k_4\le k_5\le k_6\le k_0+1,
  \end{equation*}
  and therefore $k_j=\floor{(k+j)/7}$ for $j\in\{0,1,2,3,4,5,6\}$ by Lemma~\ref{lemma:floor}. We have
  $k\ge 1$, the case $k=1$ corresponds to $T=T_{34,1}^*$. Indeed,
  $m(T_{34,1}^*)=m(T_{34,2}^*)$.
\end{proof}

\begin{proposition}\label{proposition:A-8-1-free-classification-optimal}
  Let $T\notin\calS$ be a $CL$-free optimal tree of order $n>1$. Then $n\equiv
  0\pmod 7$, $n\equiv 3\pmod 7$ or $n\equiv 6\pmod 7$  and $T$ has the shape described in
  Theorem~\ref{theorem:global-structure} for these congruence classes.
\end{proposition}
\begin{proof}
  Let $T'$ be the outline of $T$. If $T'$ has a vertex of degree $3$,
  Lemma~\ref{lemma:outline-properties} yields the required result.

  If $T'$ has at least two vertices of degree $4$, then there are at least two
  vertices of degree $4$ which are adjacent to at least $3$ special leaves. In
  this case, Lemma~\ref{lemma:caterpillar-A-8-1-free} yields the required
  result.

  We now consider the case that $T'$ has exactly one vertex $v$ of degree
  $4$. Its neighbours are special leaves $T_0$, $T_1$, $T_2$, $T_3$, where each
  $T_j$ is either an $L$ or a $C^{k_j}F$ for $0\le j\le 3$. We assume that
  $\rho(T_0)\le\rho(T_1)\le\rho(T_2)\le\rho(T_3)$. The case $T_0=T_1=T_2=T_3=L$
  corresponds to $T=T_5^*\in\calS$. Then by
  Lemma~\ref{lemma:outline-ends-LCFT3T4}, we cannot have a leaf, so
  $T_j=C^{k_j}F$ for $0\le j\le 3$. From
  \liref{lemma:outline-ends}{item:outline-ends-CTCTCTCT-T-small} and Lemma~\ref{lemma:floor}, we obtain
  that $k_j=\lfloor
  (k+j)/4\rfloor$ for $0\le j\le 3$ and $k=k_0+k_1+k_2+k_3$. We have
  \begin{equation*}
    n=|T|=1+4\cdot 4+7(k_0+k_1+k_2+k_3)=17+7k.
  \end{equation*}
  We conclude that $n\equiv 3\pmod 7$ and obtain
  \begin{equation*}
    k_j=\floor{\frac{n-17+7j}{28}}
  \end{equation*}
  and of course, $T$ has the shape given in Figure~\ref{fig:optimal-trees-3}.

  Finally we consider the case that $T'$ has order $1$. This case is covered by Lemma~\ref{lemma:outline-properties}. Both graphs mentioned in this lemma ($C^{(n-1)/7}L$ and $C^{(n-4)/7}F$) contain a $CL$ except for $T=T_1^*$.
\end{proof}

\subsection{Optimal Trees Containing \texorpdfstring{$CL$}{CL}}\label{section:A-8-1}
This final subsection is devoted to optimal trees $T\notin\calS$ containing a
$CL$ as a rooted subtree. By
Lemma~\ref{lemma:two-leaves-force-leave-everywhere}, every vertex of type $B$
of such a tree is adjacent to a leaf. By
Lemma~\ref{lemma:outline-ends-LCFT3T4}, $T$ does not contain any rooted
subtree of the shape $C^kF$ for $k>0$.

\begin{lemma}\label{lemma:A-8-1-possible-outline-ends}
  Let $T\notin\calS$ be an optimal tree containing $CL$ as a
  rooted subtree and $v$ a vertex of degree $4$ in the
  outline of $T$ which is adjacent to $3$ special leaves $T_0$, $T_1$, $L$ with
  $|T_0|\le |T_1|$.
  Then there is a $k\ge 1$ such that $T_j=C^{\lfloor(k+j)/2\rfloor}L$ for $j\in\{0,1\}$.
\end{lemma}
\begin{proof}
  Let $T_0$, $T_1$, $T_2$, $L$ denote the rooted connected components of $T-v$. None
  of them is a $C^kF$ for $k\ge 0$ by Lemma~\ref{lemma:outline-ends-LCFT3T4}
  and Lemma~\ref{lemma:two-leaves-force-leave-everywhere}. Thus $T_0=C^{k_0}L$ and
  $T_1=C^{k_1}L$ for suitable $k_0$, $k_1\ge 0$. As $A_{14}^*$ and
  $A_{24}^*$ cannot be rooted subtrees of $T$ by
  Lemma~\ref{lemma:two-leaves-force-leave-everywhere}, we have $\rho(T_2)>2/3$
  by Lemma~\ref{lemma:rho-A-2-3}. Thus \liref{lemma:outline-ends}{item:outline-ends-LCT1CT2-T-large} can be used to see that $k_0\le k_1\le
  k_0+1$. 
  With $k=k_0+k_1$, Lemma~\ref{lemma:floor} and LC6, the desired result follows.
\end{proof}

\begin{lemma}\label{lemma:A-8-1-3-ends}
  Let $T\notin\calS$ be an  optimal tree  containing $CL$ as
  a rooted subtree. Then
  there are no three distinct vertices  in the outline of $T$ such
  that each of them is adjacent to three special leaves.
\end{lemma}
\begin{proof}
  Assume that there are three distinct vertices $v_0$, $v_1$, $v_2$ in the outline of
  $T$ such that $v_i$ is adjacent to $C^{k_{i0}}L$, $C^{k_{i1}}L$,
  $L$ with $k_{ij}=\lfloor(k_i+j)/2\rfloor$ for some $k_i\ge 1$ and
  $j\in\{0,1\}$. By
  Lemma~\ref{lemma:A-8-1-possible-outline-ends} this is the only case
  to consider.

  Replacing the rooted subtree with root $v_0$ and branches $C^{k_{00}}L$,
  $C^{k_{01}}L$, $L$ by a rooted subtree with root $v_0$ and branches $L$, $F$,
  $C^{k_0-1}F$, cf.\ \liref{lemma:switching-7forks-to-7leaves}{item:switching-7leaves-to-7forks:F5}, yields a tree $T'$, which does not have to be optimal, but fulfils the LC.

  Replacing the rooted subtree with root $v_1$ and branches $C^{k_{10}}L$,
  $C^{k_{11}}L$, $L$ in $T'$ by a rooted subtree with root $v_1$ and branches $L$, $F$,
  $C^{k_1-1}F$, cf.\ \liref{lemma:switching-7forks-to-7leaves}{item:switching-7leaves-to-7forks:F5}, yields a tree $T''$, which still fulfils the LC.

  Replacing the rooted subtree with root $v_2$ and branches $C^{k_{20}}L$,
  $C^{k_{21}}L$, $L$ in $T''$ by a rooted subtree with root $v_2$ and branches 
  $C^{\lfloor (k_2-1)/3\rfloor}F,C^{\lfloor k_2/3\rfloor}F,C^{\lfloor (k_2+1)/3\rfloor}F$, cf.\ \liref{lemma:switching-7forks-to-7leaves}{item:switching-7leaves-to-7forks:F6},
  yields a tree $T'''$.

  \lireftwo{lemma:switching-7forks-to-7leaves}{item:switching-7leaves-to-7forks:F5}{item:switching-7leaves-to-7forks:F6} yields $|T'''|-|T|=-1-1+2=0$ and
  \begin{equation*}
    \frac{m(T''')}{m(T)}=\frac{m(T''')}{m(T'')}\cdot\frac{m(T'')}{m(T')}\cdot\frac{m(T')}{m(T)}\ge
     0.722\cdot 0.722\cdot\frac{27}{14}>1.005,
  \end{equation*}
  thus $m(T''')>m(T)$, contradiction to the optimality of $T$.
\end{proof}

Next we need better bounds on the $\rho$-values of subtrees of type $A$ which are visible in the outline
of an optimal tree. 

\begin{lemma}\label{lemma:rho-bounds-A_8_1-optimal}
  Let $T\notin\calS$ be an optimal tree  containing a $CL$,
  $v$ be a vertex of degree $4$ in the outline of $T$ and $L$, $T_1$, $T_2$,
  $T_3$ be the rooted connected components of $T-v$. We assume that $\rho(T_1)\ge\rho(T_2)$ and set
  $\ell=[|T_1|>1]$. We further set $S=\calA\calB(L,T_1,T_2)$.

Then 
  \begin{equation*}
    \rho(C^{1+\ell}L)<\rho(S)<\rho(C^{\ell}L).
  \end{equation*}
\end{lemma}
\begin{proof}
  We prove the lemma by induction on the order of $S$.

  By Lemma~\ref{lemma:A-good-upper-bound-rho} and LC6, we have
  $\rho(T_1)+\rho(T_2)\le 1-\ell+(1+\ell)(\sqrt{3}-1)$. 

  We write $T_j=C^{k_j}T_j'$ for suitable
  trees $T_j'$ and maximal $k_j\ge 0$ for $j\in\{1,2\}$. As the outline of
  $T$ does not contain a $C^kF$ by
  Lemma~\ref{lemma:outline-ends-LCFT3T4} and Lemma~\ref{lemma:two-leaves-force-leave-everywhere}, we
  conclude that either $T_j'$ is a leaf or we have $\rho(T_j')>\rho(C^2L)$ by
  the induction hypothesis. By Lemma~\ref{lemma:chain-parameter}, we have
  $\rho(T_j)>\rholimit$ in both cases.

  We obtain
  \begin{equation*}
    \rho(C^{1+\ell}L)<\frac1{1+\frac1{2-\ell+(1+\ell)\rholimit}}\le \rho(S)=\frac1{1+\frac1{1+\rho(T_1)+\rho(T_2)}}\le \frac1{1+\frac1{2-\ell+(1+\ell)(\sqrt{3}-1)}}<\rho(C^{\ell}L).
  \end{equation*}
\end{proof}

\begin{lemma}\label{lemma:caterpillar-A-8-1}
  Let $T\notin\calS$ be an optimal tree of order $n$ containing a
  $CL$ whose outline contains at least two vertices of
  degree $4$. Then $n\equiv 2\pmod 7$ and $T$ is of the shape given in
  Figure~\ref{fig:optimal-trees-2} with
  \begin{align*}
    k_0&=\max\left\{0,\floor{\frac{n-37}{35}}\right\},&
    k_j&=
    \begin{cases}
      \floor{\frac{n-2+7j}{35}}&\text{ if $n\ge 37$,}\\
      \floor{\frac{n-9+7j}{35}}&\text{ if $n\le 30$}
    \end{cases}
  \end{align*}
  for $j\in\{1,2,3,4\}$.
\end{lemma}
\begin{proof}
  By Lemma~\ref{lemma:outline-properties}, the outline of $T$ has no vertex of
  degree $3$.
  By Lemma~\ref{lemma:A-8-1-3-ends}, the outline graph of $T$ is a caterpillar
  tree. By Lemma~\ref{lemma:two-leaves-force-leave-everywhere}, Lemma~\ref{lemma:A-8-1-possible-outline-ends} and
  Lemma~\ref{lemma:outline-ends-LCFT3T4}, it must be of the shape shown in Figure~\ref{fig:caterpillar-A-8-1}
  with $0\le k_1\le k_3$, $0\le k_2\le k_4$, $0\le k_0$, $0\le \ell_1$, \ldots,
  $0\le \ell_s$, $0\le k_5$. By LC6, we have $k_3>0$
  and $k_4>0$. 
  \begin{figure}[htbp]
    \centering
    \includegraphics{caterpillar.7}
    \caption{Decomposition of $T$ for Lemma~\ref{lemma:caterpillar-A-8-1}.}
    \label{fig:caterpillar-A-8-1}
  \end{figure}

  We claim that $s=0$; let us assume, to the contrary, that $s\ge 1$. Then by
  Lemma~\ref{lemma:rho-bounds-A_8_1-optimal}, we have
  \begin{align*}
    \rho(C^{[\ell_1>0]+1}L)&<\rho(S_2)<\rho(C^{[\ell_1>0]}L),\\
    \rho(C^{[k_1>0]+1}L)&<\rho(S_1)<\rho(C^{[k_1>0]}L).
  \end{align*}
  As $\sigma^{-[\ell_1>0]}\rho(S_2)<\rho(L)$ and
  $\sigma^{-[k_1>0]}\rho(S_1)<\rho(L)$, \liref{lemma:outline-ends}{item:outline-ends-LCT1CT2-T-large} and Lemma~\ref{lemma:A-8-1-possible-outline-ends} imply
  \begin{equation}\label{eq:caterpillar-A-8-1-inequalities-left}
  \begin{gathered}
    \begin{aligned}
      k_0&\le\max\{0,k_1-[\ell_1>0]\},& k_1\le k_3&\le k_0+[\ell_1>0]+1,&k_3&\le
      k_1+1,\\
      k_0&\le\max\{0,\ell_1-[k_1>0]\},& \ell_1 &\le k_0+[k_1>0]+1,
    \end{aligned}\\
    \text{If $k_0>0$, then $k_1\ge 2$ and $\ell_1\ge 2$}
  \end{gathered}
\end{equation}
 (the last statement following from the two inequalities for $k_0$) and the analogous inequalities
  \begin{equation}\label{eq:caterpillar-A-8-1-inequalities-right}
    \begin{gathered}
      \begin{aligned}
        k_5&\le\max\{0,k_2-[\ell_s>0]\},& k_2\le k_4&\le k_5+[\ell_s>0]+1,&k_4&\le k_2+1,\\
        k_5&\le\max\{0,\ell_s-[k_2>0]\},& \ell_s &\le k_5+[k_2>0]+1,
      \end{aligned}\\
      \text{If $k_5>0$, then $k_2\ge 2$ and $\ell_s\ge 2$.}
    \end{gathered}
  \end{equation}
  Without loss of generality, we may assume $k_5+k_2\le k_0+k_1$.

  We consider $3$ cases:
  \begin{enumerate}
  \item We assume that $k_0>0$ and $k_2>0$. 
    From \eqref{eq:caterpillar-A-8-1-inequalities-left} and Lemma~\ref{lemma:floor}, we obtain
    \begin{align*}
      k_1&=k_0+1+\floor{\frac{s+1}3},&
      k_3&=k_0+1+\floor{\frac{s+2}3},&
      \ell_1&=k_0+1+t
    \end{align*}
    for some $s\in\{0,1,2\}$ and some $t\in\{0,1\}$.
    
    We replace
    $\calB(L,C^{\ell_1}L,C^{k_0}\calA\calB(C^{k_1}L,C^{k_3}L,L))$ as in
    \liref{lemma:switching-7forks-to-7leaves}{item:switching-7leaves-to-7forks:F7} and obtain a tree $T'$
    fulfilling the LC. 

    Replacing $\calB(C^{k_2}L, C^{k_4}L, L)$ in $T'$ as in
    \liref{lemma:switching-7forks-to-7leaves}{item:switching-7leaves-to-7forks:F6}, we obtain a tree $T''$. We
    have
    \begin{align*}
      \frac{m(T'')}{m(T)}&\ge 0.5181\cdot 1.9302>1.00003&|T''|=|T|,
    \end{align*}
    a contradiction.
  \item We assume that $k_0>0$ and $k_2=0$. By
    \eqref{eq:caterpillar-A-8-1-inequalities-right}, this implies $k_5=0$ and
    $k_4=1$.

    We assume first that $s\ge 2$ and replace
    $\calB(L,C^{\ell_1}L,C^{k_0}\calA\calB(C^{k_1}L,C^{k_3}L,L))$ in $T$ as in
    \liref{lemma:switching-7forks-to-7leaves}{item:switching-7leaves-to-7forks:F7} and obtain a tree $T'$
    fulfilling the LC. 

    We now replace $\calB(L,C^{\ell_s}L,\calA\calB(C^{k_2}L,C^{k_4}L,L))$ in $T'$ as in
    \liref{lemma:switching-7forks-to-7leaves}{item:switching-7leaves-to-7forks:F10} and obtain a tree $T''$.

    We conclude that 
    \begin{align*}
      \frac{m(T'')}{m(T)}&\ge 1.95\cdot 0.5181>1.01,&|T''|=|T'|,
    \end{align*}
    a contradiction.

    Thus we have $s=1$. By \eqref{eq:caterpillar-A-8-1-inequalities-left}, we
    have $2\le k_1\le k_3$. As $k_4<k_1$, and thus
    $\rho(C^{k_4}L)>\rho(C^{k_1}L)$, we obtain
    \begin{equation*}
      2=\rho(C^{k_2}L)+\rho(L)\le \rho(C^{k_3}L)+\rho(L)+0.1153\le
      \rho(C^2L)+\rho(L)+0.1153<2
    \end{equation*}
    from
    Lemma~\ref{lemma:chain-parameter} and
    \liref{lemma:exchange}{item:exchange-B}. This is a
    contradiction.
  \item We assume that $k_0=0$. By
    \eqref{eq:caterpillar-A-8-1-inequalities-left} this implies $k_1\le k_3\le
    2$ and $\ell_1\le 2$. Consequently, we have $k_5+k_2\le k_0+k_1\le 2$, thus 
    $k_5=0$, $k_2\le k_4\le 2$ and $\ell_s\le 2$ by
    \eqref{eq:caterpillar-A-8-1-inequalities-right}.

    We assume first that $s\ge 2$ and
    replace $\calB(L,C^{\ell_1}L,\calA\calB(C^{k_1}L, C^{k_3}L,L))$ in $T$ as in
    \liref{lemma:switching-7forks-to-7leaves}{item:switching-7leaves-to-7forks:F8} and obtain a tree $T'$.

    We now replace $\calB(L,C^{\ell_s}L,\calA\calB(C^{k_2}L,C^{k_4}L,L))$ in $T'$ as in
    \liref{lemma:switching-7forks-to-7leaves}{item:switching-7leaves-to-7forks:F10} and obtain a tree $T''$.
    We conclude that 
    \begin{align*}
      \frac{m(T'')}{m(T)}&\ge 1.95\cdot 0.516>1.006,&|T''|=|T'|,
    \end{align*}
    a contradiction.

    Thus we have $s=1$. It follows that  $27\le n=
    13+7(k_1+k_3+k_2+k_4+\ell_1)\le 83$. In each of the possible cases remaining, it turns out
    that $m(T)<m(T_n^*)$ for the tree $T_n^*$ given in
    Figure~\ref{fig:optimal-trees-6} (or $m(T)<m(T_{34,1}^*)=m(T_{34,2}^*)$ for
    $n=34$), contradiction.
  \end{enumerate}

  So we have shown that $s=0$ and that $T$ therefore has the shape as in
  Figure~\ref{fig:optimal-trees-2}. We set $k=k_0+k_1+k_2+k_3+k_4$ and obtain
  $n=9+7k$ and $n\equiv 2\pmod 7$ in particular.

  We set $a=k_1+k_3$ and $b=k_2+k_4$ and assume that $a\le b$. From
  \liref{lemma:outline-ends}{item:outline-ends-LCT1CT2-T-large} and
  Lemma~\ref{lemma:floor}, we see that 
  \begin{align*}
    k_1&=\lfloor a/2\rfloor,& k_3&=\lfloor (a+1)/2\rfloor,& 
    k_2&=\lfloor b/2\rfloor,& k_4&=\lfloor (b+1)/2\rfloor.
  \end{align*}
  If $b\ge a+2$, we have $k_4\ge \lfloor (a+3)/2\rfloor>k_3$ and therefore 
  $k_1=\floor{a/2}<\floor{b/2}=k_2\le k_1$
  by \liref{lemma:exchange}{item:exchange-symmetric} and Lemma~\ref{lemma:chain-parameter}, a
  contradiction. Thus $b\in\{a,a+1\}$ and $k_1\le k_2\le k_3\le k_4\le k_1+1$
  in both cases.

  By Lemma~\ref{lemma:rho-bounds-A_8_1-optimal} and \liref{lemma:outline-ends}{item:outline-ends-LCT1CT2-T-large}, we have (in analogy to~\eqref{eq:caterpillar-A-8-1-inequalities-left} and~\eqref{eq:caterpillar-A-8-1-inequalities-right})
  \begin{align}
    \label{eq:caterpillar-A-8-1-s-0-k-0}
    k_0&\le \max\{0, k_1-[k_2>0]\},&
    k_4&\le k_0+[k_1>0]+1.
  \end{align}

  If $k_1=0$, we have $k_0=0$ and $0=k_0\le k_1\le k_2\le k_3\le k_4\le k_0+1=1$ by
  \eqref{eq:caterpillar-A-8-1-s-0-k-0} and $n\le 30$, i.e.,
  $k_j=\floor{(n-9+7j)/35}$ for $j\in\{0,1,2,3,4\}$ by Lemma~\ref{lemma:floor}.

  If $k_1>0$, we have $n\ge 37$ and \eqref{eq:caterpillar-A-8-1-s-0-k-0} yields
  \begin{equation*}
    k_0+1\le k_1\le k_2\le k_3\le k_4\le k_0+2
  \end{equation*}
  and therefore $k_j+[j=0]=\lfloor{(k+1+j)/5}\rfloor$ for $j\in\{0,1,2,3,4\}$
  by Lemma~\ref{lemma:floor}.
\end{proof}

\begin{proposition}\label{proposition:A-8-1-classification-optimal}
  Let $T\notin\calS$ be an optimal tree of order $n$ containing a  $CL$ as a
  rooted subtree. Then $n\equiv
  1\pmod 7$, $n\equiv 2\pmod 7$, $n\equiv 4\pmod 7$ or $n\equiv 5\pmod 7$  and $T$ has the shape described in
  Theorem~\ref{theorem:global-structure} for these congruence classes.
\end{proposition}
\begin{proof}
  Let $T'$ be the outline of $T$. If $T'$ has a vertex of degree $3$,
  Lemma~\ref{lemma:outline-properties} shows that $T$ does not contain a
  $CL$ as a rooted subtree.

  If $T'$ has at least two vertices of degree $4$, then there are at least two
  vertices of degree $4$ which are adjacent to at least $3$ special leaves. In
  this case, Lemma~\ref{lemma:caterpillar-A-8-1} yields the required
  result.

  We now consider the case that $T'$ has exactly one vertex $v$ of degree
  $4$. Its neighbours are special leaves $T_0$, $T_1$, $T_2$, $T_3$. By
  Lemma~\ref{lemma:two-leaves-force-leave-everywhere}, we have $T_3=L$ after
  suitable reordering. As $v$ is in the outline of $T$, we have
  $F\notin\{T_0,T_1,T_2\}$. By Lemma~\ref{lemma:outline-ends-LCFT3T4}, we must
  have $T_j=C^{k_j}L$ for some $k_j\ge 0$ for $j\in\{0,1,2\}$. Thus $T$ is of
  the shape given in Figure~\ref{fig:optimal-trees-5}.
  This yields
  $n=7(k_0+k_1+k_2)+5$; in particular $n\equiv 5\pmod 7$.

  Without loss of
  generality, we may assume that $k_0\le k_1\le k_2$. By
  \liref{lemma:outline-ends}{item:outline-ends-LCT1CT2-T-large},
  we have
  \begin{equation*}
    k_0\le k_1\le k_2\le k_0+1.
  \end{equation*}
  From Lemma~\ref{lemma:floor}, we conclude that
  \begin{equation*}
    k_j=\floor{\frac{k_0+k_1+k_2+j}{3}}=\floor{\frac{\frac{n-5}7+j}{3}}
  \end{equation*}
  for $0\le j\le 2$, as required.

  Finally we consider the case that $T'$ has order $1$. This case has been
  considered in Lemma~\ref{lemma:outline-properties}.
\end{proof}

\begin{proof}[Proof of Theorem~\ref{theorem:global-structure}]
  Let $T$ be an optimal tree of order $n$. If $T\notin\calS$, then there are two possibilities: $T$ can be $CL$-free
  or it can contain a $CL$ as a rooted subtree. 
  Then Propositions~\ref{proposition:A-8-1-free-classification-optimal} and
  \ref{proposition:A-8-1-classification-optimal} respectively show that $T$ has the shape
  given in Theorem~\ref{theorem:global-structure} with the parameters as given
  by the theorem. For $n \in \{8,9,12,16\}$, the trees in the exceptional set $\calS$ still have this shape. For $n \in \{6,10,13,20\}$ or $n < 4$, however, it is not possible for a tree of order $n$ to have the shape shown in Figure~\ref{fig:optimal-trees} (since $n$ is too small). In these cases, the optimal tree has to be an element of the 
exceptional set $\calS$, which gives us a unique optimal tree for $n \neq 6$ and two optimal trees for $n = 6$. Finally, let us remark that the asymptotic formul\ae{} given in Theorem~\ref{thm:upper} follow easily from the structure of the trees by means of Lemma~\ref{lemma:chain-parameter}.
\end{proof}

\bibliographystyle{abbrv}
\bibliography{max-card-matching}

\section*{Appendix: Tables of replacements}
\begin{table}[htbp]
  \centering
  \begin{tabular}{@{}r|>{$}l<{$}rrr@{}}
      &T&$n:=|T|$&$m(T)$&$m(T_n^*)$
      \input{replacements-full}    \end{tabular}
\bigskip

    \caption{Replacements for trees: $|T|=|T_n^*|=n$ and $m(T)<m(T_n^*)$
      hold.}
\label{tab:replacements-optimal}
\end{table}

\begin{table}[htbp]
  \centering
  \begin{sideways}%
    \tolerance=10000\hbadness=10000%
    \newlength{\widthorig}%
    \newlength{\widthreplacement}%
    \newlength{\fracspacer}%
    \newlength{\fracdepth}%
    \settowidth{\widthorig}{$\mathcal{A}\mathcal{B}(A_{24}^*, A_{24}^*,A_{24}^*)$}%
    \settowidth{\widthreplacement}{$\mathcal{B}(L, F, \allowbreak\mathcal{A}\mathcal{B}(F, F, CF))$}%
    \settoheight{\fracspacer}{$\frac23$}%
    \settodepth{\fracdepth}{$\frac23$}%
    \addtolength{\fracspacer}{\fracdepth}%
    \newcommand{\insertfracspacer}{\raisebox{-1.2\fracdepth}{\rule{0pt}{1.2\fracspacer}}}%
    \begin{tabular}{@{}r|>{\hangindent1em$}p{\widthorig}<{$}rrr|>{\hangindent1em$}p{\widthreplacement}<{$}rr|>{$}c<{$}@{}}
      &T&$|T|$&$m(T)$&$m_0(T)$&
      T'&$m(T')$&$m_0(T')$&
      \input{replacements}
    \end{tabular}
\end{sideways}
\bigskip

\caption{Replacements for rooted subtrees: $|T'|=|T|$, $\type(T')=\type(T)$
  and $m(T')+\alpha m_0(T')>m(T)+\alpha m_0(T)$ hold for the given range of
  $\alpha$. Here, the additional abbreviation $C_L^kS=\calA\calB(L,L,C_L^{k-1}S)$
  and $C_{CL}^kS=\calA\calB(L,CL,C_L^{k-1}S)$ for $k\ge 1$ have been used,
  where, as usual,  $C_L^0S=S$ and $C_{CL}^0S=S$.}
\label{tab:replacements-alpha-optimal}
\end{table}

\end{document}

%% file: replacements-full.tex
\\\hline\refstepcounter{replacement}\label{replacement-full:4-claw-is-bad}(\thereplacement)
    &    \mathcal{B}(L, L, L, L, \allowbreak\mathcal{A}\mathcal{B}(L, L, L))&10&19&
    21
\\
    &    \mathcal{B}(L, L, L, L, A_{6}^*)&11&24&
    30
\\\hline\refstepcounter{replacement}\label{replacement-full:light-no-degree-4-new}(\thereplacement)
    &    \mathcal{B}(A_7^*, F, F, L)&17&213&
    216
\\\hline\refstepcounter{replacement}\label{replacement-full:heavy-trees}(\thereplacement)
    &    \allowbreak\mathcal{A}\mathcal{B}(L, L, \allowbreak\mathcal{A}\mathcal{B}(L, L, \allowbreak\mathcal{A}\mathcal{B}(L, L, \allowbreak\mathcal{A}\mathcal{B}(L, L, L))))&17&209&
    216

%% file: replacements.tex
\\\hline\refstepcounter{replacement}\label{replacement:good-ends-replacements}(\thereplacement)
&    \allowbreak\mathcal{B}A_{3}^*&4&1&2&
    \mathcal{B}(L, L, L)&3&1&
    \alpha<2
\\
&    \mathcal{B}(L, A_{3}^*)&5&3&2&
    \mathcal{B}(L, L, L, L)&4&1&
    \alpha<1
\\
&    \mathcal{B}(A_{3}^*, A_{3}^*)&7&4&4&
    \mathcal{B}(L, C_LL)&7&4&
    \alpha\ge 0
\\
&    \mathcal{B}(A_{3}^*, A_{3}^*, A_{3}^*)&10&12&8&
    \mathcal{B}(L, F, F)&21&9&
    \alpha\ge 0
\\\hline\refstepcounter{replacement}\label{replacement:rho-A-2-3}(\thereplacement)
&    \allowbreak\mathcal{A}\mathcal{B}(F, F, A_{24}^*)&34&59049&39366&
    C_{CL}^{2}C_L^{2}F&58999&41839&
    \alpha>\frac{50}{2473}\insertfracspacer
\\
&    \allowbreak\mathcal{A}\mathcal{B}(F, A_{14}^*, A_{14}^*)&34&59049&39366&
    C_{CL}^{2}C_L^{2}F&58999&41839&
    \alpha>\frac{50}{2473}\insertfracspacer
\\
&    \allowbreak\mathcal{A}\mathcal{B}(F, A_{14}^*, A_{24}^*)&44&1594323&1062882&
    CC_{CL}^{2}C^{2}L&1618650&1139139&
    \alpha\ge 0
\\
&    \allowbreak\mathcal{A}\mathcal{B}(F, A_{24}^*, A_{24}^*)&54&43046721&28697814&
    C^{2}\allowbreak\mathcal{A}\mathcal{B}(L, C^{2}L, C^{3}L)&44259488&31126973&
    \alpha\ge 0
\\
&    \allowbreak\mathcal{A}\mathcal{B}(A_{14}^*, A_{14}^*, A_{14}^*)&44&1594323&1062882&
    CC_{CL}^{2}C^{2}L&1618650&1139139&
    \alpha\ge 0
\\
&    \allowbreak\mathcal{A}\mathcal{B}(A_{14}^*, A_{14}^*, A_{24}^*)&54&43046721&28697814&
    C^{2}\allowbreak\mathcal{A}\mathcal{B}(L, C^{2}L, C^{3}L)&44259488&31126973&
    \alpha\ge 0
\\
&    \allowbreak\mathcal{A}\mathcal{B}(A_{14}^*, A_{24}^*, A_{24}^*)&64&1162261467&774840978&
    C^{9}L&1209774005&850782533&
    \alpha\ge 0
\\
&    \allowbreak\mathcal{A}\mathcal{B}(A_{24}^*, A_{24}^*, A_{24}^*)&74&31381059609&20920706406&
    C^{10}F&33062296902&23251305273&
    \alpha\ge 0
\\\hline\refstepcounter{replacement}\label{replacement:light}(\thereplacement)
&    \mathcal{B}(L, A_{10}^*)&12&34&21&
    \mathcal{B}(L, L, C_L^{2}L)&41&15&
    \alpha<\frac{7}{6}\insertfracspacer
\\\hline\refstepcounter{replacement}\label{replacement:light-2-3-replacements}(\thereplacement)
&    \mathcal{B}(L, A_{14}^*)&16&135&81&
    \mathcal{B}(L, L, C_L^{3}L)&153&56&
    \alpha<\frac{18}{25}\insertfracspacer
\\
&    \mathcal{B}(L, A_{24}^*)&26&3645&2187&
    \mathcal{B}(L, L, C_LC_{CL}CL)&4235&1551&
    \alpha<\frac{295}{318}\insertfracspacer
\\\hline\refstepcounter{replacement}\label{replacement:light-2-3-replacements-A-10}(\thereplacement)
&    \mathcal{B}(A_{10}^*, A_7^*, A_7^*)&25&2512&1344&
    \mathcal{B}(L, C_LC_{CL}CL)&2684&1551&
    \alpha\ge 0
\\
&    \mathcal{B}(A_{10}^*, A_{10}^*, A_7^*)&28&6573&3528&
    \mathcal{B}(F, CL, C^{2}L)&7759&3696&
    \alpha\ge 0
\\
&    \mathcal{B}(A_{10}^*, A_{10}^*, A_{10}^*)&31&17199&9261&
    \mathcal{B}(F, CL, C^{2}F)&20967&9999&
    \alpha\ge 0
\\\hline\refstepcounter{replacement}\label{replacement:light-2-3-replacements-A-7}(\thereplacement)
&    \mathcal{B}(A_7^*, F, F)&16&141&72&
    \mathcal{B}(L, L, C_L^{3}L)&153&56&
    \alpha<\frac{3}{4}\insertfracspacer
\\
&    \mathcal{B}(A_7^*, A_7^*, F)&19&368&192&
    \mathcal{B}(L, L, C_L^{3}F)&418&153&
    \alpha<\frac{50}{39}\insertfracspacer
\\
&    \mathcal{B}(A_7^*, A_7^*, A_7^*)&22&960&512&
    \mathcal{B}(L, C_L^{4}F)&989&571&
    \alpha\ge 0
\\\hline\refstepcounter{replacement}\label{replacement:heavy}(\thereplacement)
&    \mathcal{B}(L, L, L, C_LL)&9&15&4&
    \mathcal{B}(L, A_7^*)&13&8&
    \alpha>\frac{1}{2}\insertfracspacer
\\
&    \mathcal{B}(L, L, C_L^{4}L)&20&571&209&
    \mathcal{B}(L, F, A_{14}^*)&567&243&
    \alpha>\frac{2}{17}\insertfracspacer
\\\hline\refstepcounter{replacement}\label{replacement:outline-ends-LCFT3T4}(\thereplacement)
&    \mathcal{B}(L, F, CA_{14}^*)&27&5751&2430&
    \mathcal{B}(L, F, \allowbreak\mathcal{A}\mathcal{B}(F, F, CF))&5742&2457&
    \alpha>\frac{1}{3}\insertfracspacer
\\
&    \mathcal{B}(L, F, CA_{24}^*)&37&155277&65610&
    \mathcal{B}(L, F, \allowbreak\mathcal{A}\mathcal{B}(F, F, \allowbreak\mathcal{A}\mathcal{B}(F, F, CF)))&155007&66420&
    \alpha>\frac{1}{3}\insertfracspacer

%% file: max-card-matching.bbl
\begin{thebibliography}{10}

\bibitem{baxter1980hardsquare}
R.~J. Baxter, I.~G. Enting, and S.~K. Tsang.
\newblock Hard-square lattice gas.
\newblock {\em J. Statist. Phys.}, 22(4):465--489, 1980.

\bibitem{brod2006australasian}
D.~Br{\'o}d and Z.~Skupie{\'n}.
\newblock Trees with extremal numbers of dominating sets.
\newblock {\em Australas. J. Combin.}, 35:273--290, 2006.

\bibitem{brod2008recurrence}
D.~Br{\'o}d and Z.~Skupie{\'n}.
\newblock Recurrence among trees with most numerous efficient dominating sets.
\newblock {\em Discrete Math. Theor. Comput. Sci.}, 10(1):43--55, 2008.

\bibitem{cvetkovic1995spectra}
D.~M. Cvetkovi{\'c}, M.~Doob, and H.~Sachs.
\newblock {\em Spectra of graphs}.
\newblock Johann Ambrosius Barth, Heidelberg, third edition, 1995.

\bibitem{gorska2007trees}
J.~G{\'o}rska and Z.~Skupie{\'n}.
\newblock Trees with maximum number of maximal matchings.
\newblock {\em Discrete Math.}, 307(11-12):1367--1377, 2007.

\bibitem{Graham-Knuth-Patashnik:1994}
R.~L. Graham, D.~E. Knuth, and O.~Patashnik.
\newblock {\em Concrete mathematics. {A} foundation for computer science}.
\newblock Addison-Wesley, second edition, 1994.

\bibitem{gutman1980graphs}
I.~Gutman.
\newblock Graphs with greatest number of matchings.
\newblock {\em Publ. Inst. Math. (Beograd) (N.S.)}, 27(41):67--76, 1980.

\bibitem{gutman1999energy}
I.~Gutman.
\newblock The energy of a graph: old and new results.
\newblock In {\em Algebraic combinatorics and applications ({G}\"o\ss
  weinstein, 1999)}, pages 196--211. Springer, Berlin, 2001.

\bibitem{gutman1986mathematical}
I.~Gutman and O.~E. Polansky.
\newblock {\em Mathematical concepts in organic chemistry}.
\newblock Springer-Verlag, Berlin, 1986.

\bibitem{heilmann1972theory}
O.~J. Heilmann and E.~H. Lieb.
\newblock Theory of monomer-dimer systems.
\newblock {\em Comm. Math. Phys.}, 25:190--232, 1972.

\bibitem{Heuberger-Wagner:max-card-matching-sage}
C.~Heuberger and S.~G. Wagner.
\newblock The number of maximum matchings in a tree --- online ressources.
\newblock
  \url{http://www.math.tugraz.at/~cheub/publications/max-card-matching/}.

\bibitem{hosoya1986topological}
H.~Hosoya.
\newblock Topological index as a common tool for quantum chemistry, statistical
  mechanics, and graph theory.
\newblock In {\em Mathematical and computational concepts in chemistry
  (Dubrovnik, 1985)}, Ellis Horwood Ser. Math. Appl., pages 110--123. Horwood,
  Chichester, 1986.

\bibitem{kirk2008largest}
R.~Kirk and H.~Wang.
\newblock Largest number of subtrees of trees with a given maximum degree.
\newblock {\em SIAM J. Discrete Math.}, 22(3):985--995, 2008.

\bibitem{Knuth:1992:two-notes-notat}
D.~E. Knuth.
\newblock Two notes on notation.
\newblock {\em Amer. Math. Monthly}, 99(5):403--422, 1992.

\bibitem{lin1995trees}
S.~B. Lin and C.~Lin.
\newblock Trees and forests with large and small independent indices.
\newblock {\em Chinese J. Math.}, 23(3):199--210, 1995.

\bibitem{lovasz1986matching}
L.~Lov{\'a}sz and M.~D. Plummer.
\newblock {\em Matching theory}, volume 121 of {\em North-Holland Mathematics
  Studies}.
\newblock North-Holland Publishing Co., Amsterdam, 1986.

\bibitem{merrifield1989topological}
R.~E. Merrifield and H.~E. Simmons.
\newblock {\em Topological {M}ethods in {C}hemistry}.
\newblock Wiley, New York, 1989.

\bibitem{prodinger1982fibonacci}
H.~Prodinger and R.~F. Tichy.
\newblock Fibonacci numbers of graphs.
\newblock {\em Fibonacci Quart.}, 20(1):16--21, 1982.

\bibitem{sagan1988independent}
B.~E. Sagan.
\newblock A note on independent sets in trees.
\newblock {\em SIAM J. Discrete Math.}, 1(1):105--108, 1988.

\bibitem{Stein-others:2010:sage-mathem}
W.~A. Stein et~al.
\newblock {\em {S}age {M}athematics {S}oftware ({V}ersion 4.6)}.
\newblock The Sage Development Team, 2010.
\newblock \url{http://www.sagemath.org}.

\bibitem{szekely2007binary}
L.~A. Sz{\'e}kely and H.~Wang.
\newblock Binary trees with the largest number of subtrees.
\newblock {\em Discrete Appl. Math.}, 155(3):374--385, 2007.

\bibitem{wagner2010maxima}
S.~Wagner and I.~Gutman.
\newblock Maxima and minima of the {H}osoya index and the
  {M}errifield-{S}immons index: A survey of results and techniques.
\newblock {\em Acta Appl. Math.}, 112(3):323--346, 2010.

\bibitem{wilf1986number}
H.~S. Wilf.
\newblock The number of maximal independent sets in a tree.
\newblock {\em SIAM J. Algebraic Discrete Methods}, 7(1):125--130, 1986.

\end{thebibliography}
